\documentclass[reqno,a4paper,twoside,draft]{amsart}
\usepackage{enumitem}
\setenumerate{label=\textnormal{(\arabic*)}}

\allowdisplaybreaks[1]

\usepackage{amsmath,amssymb,dsfont,verbatim,bm,mathtools,geometry,fge}
\usepackage[raggedright]{titlesec}

\usepackage{tikz}
\usepackage{float,subfig}
\usepackage{amsrefs}

\titleformat{\chapter}[display]
{\normalfont\huge\bfseries}{\chaptertitlename\\thechapter}{20pt}{\Huge}

\titleformat{\section}
{\normalfont\Large\bfseries\center}{\thesection}{1em}{}

\titleformat{\subsection}
{\normalfont\large\bfseries}{\thesubsection}{1em}{}

\titleformat{\subsubsection}[runin]
{\normalfont\normalsize\bfseries}{\thesubsubsection}{1em}{}

\titleformat{\paragraph}[runin]
{\normalfont\normalsize\bfseries}{\theparagraph}{1em}{}

\titleformat{\subparagraph}[runin]
{\normalfont\normalsize\bfseries}{\thesubparagraph}{1em}{}

\titlespacing*{\chapter} {0pt}{50pt}{40pt}
\titlespacing*{\section} {0pt}{3.5ex plus 1ex minus .2ex}{2.3ex plus .2ex}
\titlespacing*{\subsection} {0pt}{3.25ex plus 1ex minus .2ex}{1.5ex plus .2ex}
\titlespacing*{\subsubsection}{0pt}{3.25ex plus 1ex minus .2ex}{1.5ex plus .2ex}
\titlespacing*{\paragraph} {0pt}{3.25ex plus 1ex minus .2ex}{1em}
\titlespacing*{\subparagraph} {\parindent}{3.25ex plus 1ex minus .2ex}{1em}


\newtheorem{theorem}{Theorem}[section]
\newtheorem{lemma}[theorem]{Lemma}
\newtheorem{proposition}[theorem]{Proposition}
\newtheorem{corollary}[theorem]{Corollary}

\theoremstyle{definition}
\newtheorem{definition}[theorem]{Definition}
\newtheorem{notations}[theorem]{Notations}
\newtheorem{notation}[theorem]{Notation}

\theoremstyle{remark}
\newtheorem{remark}[theorem]{Remark}

\DeclareMathOperator{\Aut}{Aut}

\DeclareMathOperator{\ch}{char}
\DeclareMathOperator{\ide}{id}

\DeclareMathOperator{\lcm}{lcm}

\DeclareMathOperator{\Supp}{Supp}

\DeclareMathOperator{\en}{en}

\DeclareMathOperator{\factors}{factors}

\DeclareMathOperator{\st}{st}

\DeclareMathOperator{\PE}{PE}
\DeclareMathOperator{\dir}{dir}
\DeclareMathOperator{\Dirsup}{Dirsup}
\DeclareMathOperator{\Dirinf}{Dirinf}
\DeclareMathOperator{\Succ}{Succ}
\DeclareMathOperator{\Pred}{Pred}
\DeclareMathOperator{\Dir}{Dir}

\newcommand{\ov}{\overline}

\begin{document}

\title{The Dixmier conjecture and the shape of possible counterexamples}

\author{Jorge A. Guccione}
\address{Departamento de Matem\'atica\\ Facultad de Ciencias Exactas y Naturales-UBA, Pabell\'on~1-Ciudad Universitaria\\ Intendente Guiraldes 2160 (C1428EGA) Buenos Aires, Argentina.}
\address{Instituto de Investigaciones Matem\'aticas ``Luis A. Santal\'o"\\ Facultad de Ciencias Exactas y Natu\-rales-UBA, Pabell\'on~1-Ciudad Universitaria\\ Intendente Guiraldes 2160 (C1428EGA) Buenos Aires, Argentina.}
\email{vander@dm.uba.ar}

\author{Juan J. Guccione}
\address{Departamento de Matem\'atica\\ Facultad de Ciencias Exactas y Naturales-UBA\\ Pabell\'on~1-Ciudad Universitaria\\ Intendente Guiraldes 2160 (C1428EGA) Buenos Aires, Argentina.}
\address{Instituto Argentino de Matem\'atica-CONICET\\ Saavedra 15 3er piso\\ (C1083ACA) Buenos Aires, Ar\-gentina.}
\email{jjgucci@dm.uba.ar}

\thanks{Jorge A. Guccione and Juan J. Guccione were supported by PIP 112-200801-00900 (CONICET)}

\author{Christian Valqui}
\address{Pontificia Universidad Cat\'olica del Per\'u, Secci\'on Matem\'aticas, PUCP, Av. Universitaria 1801, San Miguel, Lima 32, Per\'u.}

\address{Instituto de Matem\'atica y Ciencias Afines (IMCA) Calle Los Bi\'ologos 245. Urb San C\'esar. La Molina, Lima 12, Per\'u.}
\email{cvalqui@pucp.edu.pe}

\thanks{Christian Valqui was supported by PUCP-DGI-2011-0206, PUCP-DGI-2012-0011 and Lucet 90-DAI-L005.}


\begin{abstract} We establish a lower bound for the size of possible counterexamples of the Dixmier Conjecture. We prove that $B>15$, where $B$ is the minimum of the greatest common divisor of the total degrees of $P$ and $Q$, where $(P,Q)$ runs over the counterexamples of the Dixmier Conjecture.
\end{abstract}

\maketitle

\section*{Introduction}

The {\em Weyl algebra} $A_1$ over a field $K$ is the quotient of the free associative and unital $K$-algebra on two generators $X,Y$ by the ideal generated by the relation $[Y,X]=1$. The Weyl algebra is the first of an infinite family of algebras, known as {\em Weyl algebras}, which were introduced by Hermann Weyl to study the Heisenberg uncertainty principle in quantum mechanics. The {\em $n$-th Weyl algebra over $K$} is the associative and unital $K$-algebra $A_n$ generated by the $2n$ variables $X_1,Y_1,\dots, X_n,Y_n$, subject to the relations $[X_i,X_j]= [Y_i,Y_j]=0$  and $[Y_i,X_j]=\delta_{ij}$, where $\delta_{ij}$ is the Kr\"onecker symbol.

In~\cite{D} Dixmier posed six problems, the first of which, also known as the Dixmier conjecture (DC), was the following: is an algebra endomorphism of the Weyl algebra $A_1$ over a field of characteristic zero, necessarily an automorphism? This question makes sense for all the Weyl algebras, and in fact the generalized DC claims that if $\ch(K)=0$, then any endomorphism of $A_n$ must be an automorphism, for each $n\in \mathds{N}$.

Currently, the DC remains open even for the case $n = 1$. Some of the results in this topic are the following: In the early eighties, L. Vaserstein and V. Kac showed that the generalized DC implies the Jacobian conjecture (see~\cite{BCW}, where this result is explicitly established). In 2005 the stable equivalence between the Dixmier and Jacobian conjectures was established by Yoshifumi Tsuchimoto~\cite{T}, and the same result was obtained in an independent way in 2007 by Alexei Belov-Kanel and Maxim Kontsevich~\cite{BK} and by Pascal Kossivi Adjamagbo and Arno van den Essen~\cite{AE}. For a short proof of the equivalence between these two problems see~\cite{B1}. In fact, in the~\cite{AE} paper, there is also established the equivalence between the generalized DC and a similar conjecture about the endomorphisms of a family of Poisson algebras, that the authors call the Poisson Conjecture. Some other papers concerning the DC in a direct or an indirect way are~\cite{B2}, \cite{BL}, \cite{GGV1} and~\cite{Z}.

In this work we deal with the case $n=1$. We set
$$
B:= \begin{cases} \infty &\text{if DC is true,}\\ \min\bigl(\gcd(\deg(P),\deg(Q) \bigr) & \text{where $(P,Q)$ is a counterexample to DC, if DC is false.}
\end{cases}
$$
Here $\deg(P)$ is the total degree of $P$ and we will prove that  $B>15$. For this we consider a counterexample $(P,Q)$ with $\gcd(\deg(P),\deg(Q))=B$ and call it a minimal pair. Such a pair can be supposed to be subrectangular, which means that the support fits into a rectangle and contains its farthest corner. We will ``cut'' the lower right edge of the support. This process requires various geometric conditions on the support, and we will show that if $\gcd(\deg(P),\deg(Q))\le 15$, then these conditions cannot be satisfied, thus proving our result.

There is a strong parallelism with the shape of possible counterexamples of the Jacobian Conjecture in dimension~$2$. In fact, based on a work of Appelgate and Onishi (\cite{AO}), Nagata shows in~\cite{N1} and~\cite{N2}, that~$9$ is a lower bound for $\gcd(\deg(P),\deg(Q))$, where $(P,Q)$ is a counterexample to JC.

In order to carry out our constructions we need to embed $P$ and $Q$ in a bigger algebra $A_1^{(l)}$, obtained adjoining fractional powers of $X$ to $A_1$. As a $K$-linear space $A_1^{(l)}$ is $K[X,X^{-1/l},Y]$ and the relation $[Y,X]=1$ is preserved. This is basically the same construction as in~\cite{J}*{p.~599} where fractional powers of $Y$ are added instead.

In Section~1 we recall some basic concepts. For us a direction is a pair $(\rho,\sigma)$ of coprime integers with $\rho+\sigma\ge 0$. For each such a direction $(\rho,\sigma)$ we associate a $(\rho,\sigma)$-degree $v_{\rho,\sigma}$ on $A_1^{(l)}$ and we call $\ell_{\rho,\sigma}(P)$ the highest $(\rho,\sigma)$-degree term of $P$. Moreover we define $\st_{\rho,\sigma}(P)$ and $\en_{\rho,\sigma}(P)$ to be the starting point and the end point of the support of $\ell_{\rho,\sigma}(P)$.

\begin{figure}[h]
\centering

\begin{tikzpicture}
\draw[step=.5cm,gray,very thin] (0,0) grid (3.8,3.8);
\draw [->] (0,0) -- (0.5,0) node[anchor=north] {$1$} -- (1,0)  -- (2.5,0) node[anchor=north]{$5$} --(4,0) node[anchor=north]{$X$};
\draw [->] (0,0) -- (0,0.5) node[anchor=east] {$1$} -- (0,2.5) node[anchor=east]{$5$} -- (0,4) node[anchor=east]{$Y$};
\draw (2.5,0) -- (3,0.5)
 -- (3, 1.5) node[fill=white, right=4pt]{$\st_{\rho,\sigma}(P)$} -- (2,3)node[fill=white, anchor=south west]{$\en_{\rho,\sigma}(P)$} -- (1,3.5)
-- (0,2.5);
\draw [->] (0,0) -- (1.5,1) node[fill=white, right=1pt] {$(\rho,\sigma)$};
\draw [very thick](3,1.5)--node[fill=white,anchor=south west]{$\ell_{\rho,\sigma}(P)$} (2,3);
\filldraw [gray] (1.5,0)   circle (2pt)
                 (2.5,0)   circle (2pt)
                 (3,0.5)   circle (2pt)
                 (3,1.5)   circle (2pt)
                 (2,3)     circle (2pt)
                 (1,3.5)   circle (2pt)
                 (0.5,3)   circle (2pt)
                 (0.5,1.5) circle (2pt)
                 (1.5,2)   circle (2pt);
\draw[xshift=4.8cm,yshift=2cm]
node[right,text width=6.5cm]{Illustration of $\st_{\rho,\sigma}(P)$ and $\en_{\rho,\sigma}(P)$ for
\begin{align*} P= & X^3 + X^5 + X^6Y + XY^3 + X^6Y^3 \\ &+ X^3Y^4 + XY^6 + X^4Y^6 + X^2Y^7\end{align*}
and $(\rho,\sigma)=(3,2)$. In this example
$$
\ell_{\rho,\sigma}(P)= x^6y^3 + x^4y^6.
$$};
\end{tikzpicture}

\caption{Illustration of $\st$ and $\en$}
\end{figure}

In Section~2 we associate a bracket $[-,-]_{\rho,\sigma}$ with each direction $(\rho,\sigma)$. This bracket is essentially the commutator of the highest terms and coincides, up to signs, with the usual Poisson bracket of the terms of highest order, for example as defined in~\cite{J}*{p.~599}.

Then we analyze the relation of the $(\rho,\sigma)$-bracket with $\st_{\rho,\sigma}$ and $\en_{\rho,\sigma}$. This allows us to explore in Theorem~\ref{central} the geometric implications of an important result of~\cite{J} on the shape of minimal pairs. This result states that there exists a $(\rho,\sigma)$-ho\-mo\-ge\-neous element $F\in A_1^{(l)}$ such that
$$
[P,F]_{\rho,\sigma}=\ell_{\rho,\sigma}(P),
$$
whenever $(P,Q)$ is a counterexample to DC.

The aforementioned geometric implications permit us in Section~5 to ``cut'' the right lower edge of the support of a given minimal pair in $A_1^{(l)}$. The geometric conditions this imposes on the support are given partially in Proposition~\ref{lema general}, and translated into a powerful algebraic condition.

Although it is known that a counterexample $(P,Q)$ to DC can be brought into a subrectangular shape, we have to prove that this can be done without changing $\gcd(\deg(P),\deg(Q))$, and in such a way that the changed pair satisfies definition~\ref{Smp}. This work is done in Section~6, where we also prove that $\deg(P)$ does not divide $\deg(Q)$ and vice versa.

In the last section we translate all geometric conditions into algebraic ones for the possible corners of the support. We finally show that there cannot be a minimal pair $(P,Q)$ with $\gcd(\deg(P),\deg(Q))\le 15$ satisfying these conditions.

We want to emphasize that the central ideas of this article are of geometric nature. However we give detailed algebraic proofs of the geometric facts that we use along the article. This is specially true for section 3, where we analyze the order relation on directions given essentially by
$$
(\rho,\sigma)<(\rho',\sigma') \Longleftrightarrow (\rho,\sigma)\times (\rho', \sigma'):=\det\left(\begin{smallmatrix} \rho &\sigma\\ \rho &\sigma' \end{smallmatrix}\right)>0.
$$
Let us explain the two main geometric results of that section. We first define
$$
\Dir(P):=\{(\rho,\sigma)\in\mathfrak{V}:\#\Supp(\ell_{\rho,\sigma}(P))>1\}.
$$
One important result of section 3 is Proposition~\ref{le basico}, which states that if $(\rho_1,\sigma_1)<(\rho_2,\sigma_2)$ are two consecutive directions in $\Dir(P)$, then $\st_{\rho_2,\sigma_2}(P)$ coincides with $en_{\rho_1,\sigma_1}(P)$ and with $\Supp(\ell_{\rho,\sigma}(P))$ for each intermediate direction $(\rho,\sigma)$. We also use frequently Proposition~\ref{le basico1}, which compares $v_{\rho',\sigma'}\bigl(\st_{\rho,\sigma}(P)\bigr)$ with $v_{\rho',\sigma'}\bigl(\en_{\rho,\sigma}(P)\bigr)$.

\subsection*{Acknowledgment} We wish to thank the referee for useful suggestions and for pointing out the reference~\cite{J}.

\section{Preliminaries}\label{preliminares}

\setcounter{equation}{0}

In this section we fix the notation and we establish some basic results that we will use throughout the paper.

\smallskip

For each $l\in\mathds{N}$, we let $A_1^{(l)}$ denote the Ore extension $A[Y,\ide,\delta]$, where $A$ is the algebra of Laurent polynomials $K[Z_l,Z_l^{-1}]$ and $\delta\colon A\to A$ is the derivation, defined by $\delta(Z_l):=\frac{1}{l} Z_l^{1-l}$. Suppose that $l,h\in\mathds{N}$ are such that $l|h$ and let $d := h/l$. Since
\begin{equation}
[Y,Z_h^d] =\sum_{i=0}^{d-1} Z_h^i[Y,Z_h]Z_h^{d-i-1} =\frac{d}{h} Z_h^{d-h} =\frac{1}{l} (Z_h^d)^{1-l},\label{eq0}
\end{equation}
there is an algebra inclusion $\iota_l^h\colon A_1^{(l)}\to A_1^{(h)}$, such that $\iota_l^h(Z_l)= Z_h^d$ and $\iota_l^h(Y)=Y$.

\smallskip

We will write $X^{\frac{1}{l}}$ and $X^{\frac{-1}{l}}$ instead of $Z_l$ and $Z_l^{-1}$, respectively. Note that $\iota_l^h(X^{\frac{1}{l}}) = (X^{\frac{1}{h}})^d$. We will consider $A_1^{(l)}\subseteq A_1^{(h)}$ via this inclusion. Clearly $A_1$ is included in $A_1^{(1)}$.

\smallskip

Similarly, for each $l\in\mathds{N}$, we consider the commutative $K$-algebra $L^{(l)}$, generated by variables $x^{\frac{1}{l}}$, $x^{\frac{-1}{l}}$ and $y$, subject to the relation $x^{\frac{1}{l}} x^{\frac{-1}{l}} = 1$. In other words $L^{(l)} = K[x^{\frac{1}{l}}, x^{\frac{-1}{l}},y]$. Obviously, there is a canonical inclusion $L^{(l)}\subseteq L^{(h)}$, for each $l,h\in\mathds{N}$ such that $l|h$. We let
$$
\Psi^{(l)}\colon A_1^{(l)}\to L^{(l)}
$$
denote the $K$-linear map defined by $\Psi^{(l)}\bigl(X^{\frac{i}{l}}Y^j\bigr) := x^{\frac{i}{l}} y^j$. Let
$$
\ov{\mathfrak{V}} :=\{(\rho,\sigma)\in\mathds{Z}^2:\text{$\gcd(\rho,\sigma) = 1$ and $\rho+\sigma\ge 0$}\}\quad\text{and}\quad\mathfrak{V} :=\{(\rho,\sigma)\in\ov{\mathfrak{V}}:\rho+\sigma> 0\}.
$$

\begin{definition} For all $(\rho,\sigma)\in\ov{\mathfrak{V}}$ and $(i/l,j)\in\frac{1}{l}\mathds{Z}\times \mathds{Z}$, we write
$$
v_{\rho,\sigma}(i/l,j):=\rho i/l+\sigma j.
$$
\end{definition}

\begin{notations}\label{not valuaciones para polinomios} Let $(\rho,\sigma)\in\ov{\mathfrak{V}}$. For $P =\sum a_{\frac{i}{l},j} x^{\frac{i}{l}} y^j\in L^{(l)}\setminus\{0\}$, we define:

\begin{itemize}

\smallskip

\item[-] The {\em support} of $P$ as $\Supp(P) :=\left\{\left(i/l,j\right): a_{\frac{i}{l},j}\ne 0\right\}$.

\smallskip

\item[-] The {\em $(\rho,\sigma)$-degree} of $P$ as $v_{\rho,\sigma}(P):=\max\left\{v_{\rho,\sigma}\bigl(i/l,j\bigr): a_{\frac{i}{l},j}\ne 0\right\}$.

\smallskip

\item[-] The {\em $(\rho,\sigma)$-leading term} of $P$ as
$
\ell_{\rho,\sigma}(P):=\displaystyle{\sum_{\{\rho\frac{i}{l} +\sigma j = v_{\rho,\sigma}(P)\}}} a_{\frac{i}{l},j} x^{\frac{i}{l}} y^j.
$

\smallskip

\item[-] $w(P):=\left(i_0/l,i_0/l-v_{1,-1}(P)\right)$, where
$$
i_0 :=\max\left\{i:\left(i/l,i/l-v_{1,-1}(P)\right)\in\Supp(\ell_{1,-1}(P))\right\}.
$$

\smallskip

\item[-] $\ell_c(P):= a_{\frac{i_0}{l},j_0}$, where $\left(i_0/l,j_0\right)= w(P)$.

\smallskip

\item[-] $\ell_t(P):= a_{\frac{i_0}{l},j_0} x^{\frac{i_0}{l}}y^{j_0}$, where $\left(i_0//l,j_0\right)= w(P)$.

\smallskip

\item[-] $\ov{w}(P):=\left(i_0/l-v_{-1,1}(P),i_0/l\right)$, where
$$
i_0 :=\max\left\{i:\left(i/l-v_{-1,1}(P),i/l\right)\in\Supp(\ell_{-1,1}(P))\right\}.
$$

\smallskip

\item[-] $\ov{\ell}_c(P):= a_{\frac{i_0}{l}j_0}$, where $\left(i_0/l,j_0\right)=\ov{w}(P)$.

\smallskip

\item[-] $\ov{\ell}_t(P):= a_{\frac{i_0}{l}j_0} x^{\frac{i_0}{l}}y^{j_0}$, where $\left(i_0/l,j_0\right)=\ov{w}(P)$.

\end{itemize}
\end{notations}

\begin{notations}\label{not valuaciones para alg de Weyl} Let $(\rho,\sigma)\in\ov{\mathfrak{V}}$. For $P\in A_1^{(l)}\setminus\{0\}$, we define:

\begin{itemize}

\smallskip

\item[-] The {\em support} of $P$ as $\Supp(P) :=\Supp\bigl(\Psi^{(l)}(P)\bigr)$.

\smallskip

\item[-] The {\em $(\rho,\sigma)$-degree} of $P$ as $v_{\rho,\sigma}(P):= v_{\rho,\sigma}\bigl(\Psi^{(l)}(P)\bigr)$.

\smallskip

\item[-] The {\em $(\rho,\sigma)$-leading term} of $P$ as $\ell_{\rho,\sigma}(P):=\ell_{\rho,\sigma}\bigl(\Psi^{(l)}(P)\bigr)$.

\smallskip

\item[-] $w(P):= w\bigl(\Psi^{(l)}(P)\bigr)$ and $\ov{w}(P):=\ov{w}\bigl(\Psi^{(l)}(P)\bigr)$.

\smallskip

\item[-] $\ell_c(P):=\ell_c\bigl(\Psi^{(l)}(P)\bigr)$ and $\ov{\ell}_c(P):=\ov{\ell}_c\bigl(\Psi^{(l)} (P)\bigr)$.

\smallskip

\item[-] $\ell_t(P):=\ell_c(P) X^{\frac{i_0}{l}}Y^{j_0}$, where $\left(i_0/l,j_0\right) = w(P)$.

\smallskip

\item[-] $\ov{\ell}_t(P):=\ov{\ell}_c(P) X^{\frac{i_0}{l}}Y^{j_0}$, where $\left(i_0/l,j_0\right)=\ov{w}(P)$.

\end{itemize}
\end{notations}

\begin{remark} To abbreviate expressions we set $v_{\rho,\sigma}(0) = -\infty$ for all $(\rho,\sigma)\in\ov{\mathfrak{V}}$.
\end{remark}

\begin{notation}\label{not elementos rho-sigma homogeneos} We say that $P\in L^{(l)}$ is {\em $(\rho,\sigma)$-homogeneous} if $P = 0$ or $P =\ell_{\rho,\sigma}(P)$. Moreover we say that $P\in A_1^{(l)}$ is {\em $(\rho,\sigma)$-homogeneous} if $\Psi^{(l)}(P)$ is so.
\end{notation}

\begin{definition}\label{Comienzo y Fin de un elemento de A_1^{(l)}} Let $P\in A_1^{(l)}\setminus\{0\}$. For $(\rho,\sigma)\in\ov{\mathfrak{V}}\setminus\{(1,-1)\}$ and $(\rho',\sigma')\in\ov{\mathfrak{V}}\setminus\{(-1,1)\}$, we write
$$
\st_{\rho,\sigma}(P) := w(\ell_{\rho,\sigma}(P))\quad\text{and}\quad\en_{\rho',\sigma'}(P) :=\ov{w}(\ell_{\rho',\sigma'}(P)).
$$
\end{definition}

\begin{lemma} (Compare with~\cite{D}*{Lemma 2.1})\label{le conmutacion en A_1l} For each $l\in\mathds{N}$, we have
$$
Y^jX^{\frac{i}{l}} =\sum_{k=0}^j k!\binom{j}{k}\binom{i/l}{k} X^{\frac{i}{l}-k} Y^{j-k}.
$$
\end{lemma}

\begin{proof} By induction on $j$. For the inductive step use that
$$
[Y^j,X^{\frac{i}{l}}] = [Y,X^{\frac{i}{l}}] Y^{j-1}+Y[Y^{j-1},X^{\frac{i}{l}}].
$$
The case $j=1$ is true since $[Y,X^{\frac{i}{l}}] =\frac{i}{l} X^{\frac{i}{l}-1}$ by~\eqref{eq0}.
\end{proof}

For $l\in\mathds{N}$ and $j\in\mathds{Z}$, we set
$$
A_{1,j/l}^{(l)}:=\left\{P\in A_1^{(l)}\setminus\{0\} : P\text{ is $(1,-1)$-homogeneous and } v_{1,-1}(P) = j/l\right\}\cup\{0\}.
$$

\begin{remark}\label{graduacion} By Lemma~\ref{le conmutacion en A_1l}, the algebra $A_1^{(l)}$ is $\frac{1}{l}\mathds{Z}$-graded. Its homogeneous component of de\-gree $\frac{j}{l}$ is $\cramped{A_{1,j/l}^{(l)}}$. Moreover, by~\cite{D}*{3.3}, we know that $\cramped{A_{1,0}^{(l)} = K[XY]}$.
\end{remark}

\begin{proposition} (Compare with~\cite{D}*{Lemma 2.4})\label{pr v de un producto} Let $P,Q\in A_1^{(l)}\setminus\{0\}$. The following assertions hold:

\begin{enumerate}

\smallskip

\item $w(PQ) = w(P) + w(Q)$ and $\ov{w}(PQ) =\ov{w}(P) +\ov{w}(Q)$. In particular $PQ\ne 0$.

\smallskip

\item $\ell_{\rho,\sigma}(PQ) =\ell_{\rho,\sigma}(P)\ell_{\rho,\sigma}(Q)$ for all $(\rho,\sigma)\in\mathfrak{V}$.

\smallskip

\item $v_{\rho,\sigma}(PQ) = v_{\rho,\sigma}(P) + v_{\rho,\sigma}(Q)$ for all $(\rho,\sigma)\in\ov{\mathfrak{V}}$.

\smallskip

\item $\st_{\rho,\sigma}(PQ)=\st_{\rho,\sigma}(P)+\st_{\rho,\sigma}(Q)$  for all $(\rho,\sigma)\in\mathfrak{V}$.

\smallskip

\item $\en_{\rho,\sigma}(PQ)=\en_{\rho,\sigma}(P)+\en_{\rho,\sigma}(Q)$ for all $(\rho,\sigma)\in\mathfrak{V}$.

\end{enumerate}
The same properties hold for $P,Q\in L^{(l)}\setminus\{0\}$.
\end{proposition}

\begin{proof} For $P,Q\in A_1^{(l)}\setminus\{0\}$ this follows easily from Lemma~\ref{le conmutacion en A_1l},  using that $\rho+\sigma >0$ if $(\rho,\sigma)\in\mathfrak{V}$. The proof for $P,Q\in L^{(l)}\setminus\{0\}$ is easier.
\end{proof}

Let $A,B\in\mathds{R}^2$. We say that $A$ and $B$ are {\em aligned} if $A\times B:=\det\left(\begin{smallmatrix} A\\ B\end{smallmatrix}\right) = 0$.

\begin{definition}\label{def alineados} Let $P,Q\in L^{(l)}\setminus\{0\}$. We say that $P$ and $Q$ are {\em aligned} and write $P\sim Q$, if $w(P)$ and $w(Q)$ are so. Moreover we say that $P,Q\in A_1^{(l)}\setminus\{0\}$ are aligned if $\Psi^{(l)}(P)\sim\Psi^{(l)}(Q)$.
\end{definition}

\begin{remark} Note that:
\begin{itemize}

\smallskip

\item[-] $P\sim Q$ if and only if $\ell_{1,-1}(P)\sim\ell_{1,-1}(Q)$.

\smallskip

\item[-] $\sim$ is not an equivalence relation (it is so restricted to $\{P:w(P)\ne (0,0)\}$).

\smallskip

\item[-] If $P\sim Q$ and $w(P)\ne (0,0)\ne w(Q)$, then $w(P)=\lambda w(Q)$ with $\lambda\ne 0$.

\end{itemize}
\end{remark}

\begin{proposition}\label{pr conmutadores con exponentes fraccinarios} Let $P,Q\in A_1^{(l)}\setminus\{0\}$. The following assertions hold:

\begin{enumerate}

\smallskip

\item $P\nsim Q$ if and only if $w\bigl([P,Q]\bigr) = w(P)+w(Q)-(1,1)$.

\smallskip

\item $\ov{w}(P)\nsim\ov{w}(Q)$ if and only if $\ov w\bigl([P,Q]\bigr) =\ov w(P)+\ov w(Q)-(1,1)$.

\end{enumerate}
\end{proposition}

\begin{proof} We prove the first statement, and leave the second one, which is similar, to the reader. Set $w(P) =\left(\frac{r}{l},s\right)$ and $w(Q) =\left(\frac{u}{l},v\right)$. By Lemma~\ref{le conmutacion en A_1l}
$$
\ell_t\bigl([P,Q]\bigr) =\Biggl(\binom{s}{1}\binom{u/l}{1} -\binom{v}{1}\binom{r/l}{1}\Biggr)\ell_c(P)\ell_c(Q)X^{\frac{r+u}{l}-1}Y^{s+v-1}
$$
if and only if
$$
\binom{s}{1}\binom{u/l}{1} -\binom{v}{1}\binom{r/l}{1} = (r/l,s)\times (u/l,v)\ne 0.
$$
So, $w\bigl([P,Q]\bigr) = w(P)+w(Q)-(1,1)$ if and only if $P\nsim Q$, as desired.
\end{proof}

\begin{remark}\label{re v de un conmutador} For all $P,Q\in A_1^{(l)}\setminus\{0\}$ and each $(\rho,\sigma)\in\ov{\mathfrak{V}}$,
$$
v_{\rho,\sigma}([P,Q])\le v_{\rho,\sigma}(P) + v_{\rho,\sigma}(Q) - (\rho+\sigma).
$$
\end{remark}

\begin{remark}\label{re v de un conmutador con 1,-1} If $P,Q\in A_1^{(l)}\setminus\{0\}$ satisfy $v_{1,-1}(P)\le 0$ and $v_{1,-1}(Q)\le 0$, then $v_{1,-1}([Q,P])< 0$. In fact, write $P = P_0+P_1$ and $P = Q_0+Q_1$ with $P_0$ and $Q_0$ the $(1,-1)$-homogeneous components of degree $0$ of $P$ and $Q$ respectively. Since, by Remark~\ref{graduacion} we have $[P_0,Q_0]=0$, it follows from the above remark that
$$
v_{1,-1}([P,Q]) = v_{1,-1}([P_0,Q_1] + [P_1,Q]) < 0.
$$
\end{remark}

\section{The bracket associated with a direction}

\setcounter{equation}{0}

In this section we define bracket $[P,Q]_{\rho,\sigma}$ for each direction $(\rho,\sigma)$ and each $P,Q\in A_1^{(l)}$. This bracket is essentially the commutator of the highest $(\rho,\sigma)$-degree terms of its arguments. It coincides, up to signs, with the usual Poisson bracket $\{\ell_{\rho,\sigma}(P), \ell_{\rho,\sigma}(P)\}$ as defined in~\cite{J}*{p.~599}. Then, in Proposition~\ref{extremosnoalineados}, we analyze the relation of the $(\rho,\sigma)$-bracket with $\st_{\rho,\sigma}$ and $\en_{\rho,\sigma}$. Finally we prove Theorem~\ref{f[] en A_1^{(l)}}, which states that, under some mild conditions,  if $[P,Q]_{\rho,\sigma}=0$, then $\ell_{\rho,\sigma}(P)$ and $\ell_{\rho,\sigma}(Q)$ are scalar multiples of powers of one element $R$.

\begin{definition}\label{def rho-sigma proporcionales} Let $(\rho,\sigma)\in\ov{\mathfrak{V}}$ and $P,Q\in A_1^{(l)}\setminus\{0\}$. We say that $P$ and $Q$ are {\em $(\rho,\sigma)$-pro\-portional} if $v_{\rho,\sigma}([P,Q]) < v_{\rho,\sigma}(P) + v_{\rho,\sigma}(Q) - (\rho+\sigma)$.
\end{definition}

\begin{definition}\label{def rho-sigma corchete} Let $l\in\mathds{N}$ and $(\rho,\sigma)\in\ov{\mathfrak{V}}$. We define
$$
[-,-]_{\rho,\sigma}\colon\bigl(A_1^{(l)}\setminus\{0\}\bigr)\times\bigl(A_1^{(l)} \setminus\{0\}\bigr)\longrightarrow L^{(l)}_{\rho,\sigma},
$$
by
$$
[P,Q]_{\rho,\sigma}:=\begin{cases} 0 &\text{if $P$ and $Q$ are $(\rho,\sigma)$-proportional,}\\ \ell_{\rho,\sigma}([P,Q]) &\text{if $P$ and $Q$ are not $(\rho,\sigma)$-proportional.}
\end{cases}
$$
\end{definition}

\smallskip

From now on in order to simplify expressions we set $\mathfrak{V}^0 := \{(\rho,\sigma)\in \mathfrak{V} : \rho > 0\}$.

\begin{lemma}\label{pr conmutadores de elementos homogeneos con exponentes fraccinarios} Let $(\rho,\sigma)\in \mathfrak{V}^0$ and let $P$ and $Q$ be $(\rho,\sigma)$-homogeneous elements of $A_1^{(l)}\setminus\{0\}$.

\begin{enumerate}

\smallskip

\item If $w(P)\nsim w(Q)$, then $[P,Q]\ne 0$ and $w([P,Q]) = w\bigl(\ell_{\rho,\sigma}([P,Q])\bigr)$.

\smallskip

\item If $w(P)+w(Q)-(1,1)=w\bigl(\ell_{\rho,\sigma}([P,Q])\bigr)$, then $w(P)\nsim w(Q)$.

\smallskip

\item If $\ov{w}(P)\nsim\ov{w}(Q)$, then $[P,Q]\ne 0$ and $\ov{w}([P,Q])=\ov{w}\bigl(\ell_{\rho,\sigma} ([P,Q])\bigr)$.

\smallskip

\item If $\ov w(P)+\ov w(Q)-(1,1)=\ov w\bigl(\ell_{\rho,\sigma}([P,Q])\bigr)$, then $\ov w(P)\nsim\ov w(Q)$.

\end{enumerate}

\end{lemma}

\begin{proof} We only prove statements~(1) and~(2), since the other ones are similar. Write
$$
\qquad P =\sum_{i=0}^{\alpha}\lambda_i X^{\frac{r}{l}-\frac{i\sigma} {\rho}} Y^{s+i}\quad\text{and}\quad Q =\sum_{j=0}^{\beta}\mu_j X^{\frac{u}{l}-\frac{j\sigma}{\rho}} Y^{v+j},
$$
with $\lambda_0,\lambda_{\alpha},\mu_0,\mu_{\beta}\ne 0$. Note that $\rho>0$ implies $w(P) = (r/l,s)$ and $w(Q) = (u/l,v)$. Since, by Lemma~\ref{le conmutacion en A_1l},
$$
X^{\frac{i}{l}}Y^j X^{\frac{i'}{l}}Y^{j'} =\sum_{k=0}^j  k!\binom{j}{k}\binom{i'/l}{k} X^{\frac{i+i'}{l}-k}Y^{j+j'-k},
$$
we obtain that
\begin{equation}
[P,Q] =\sum_{i=0}^{\alpha}\sum_{j=0}^{\beta}\sum_{k=0}^{\max\{s+i,v+j\}}\lambda_i\mu_j c_{ijk} X^{\frac{r+u}{l}-\frac{(i+j)\sigma} {\rho}-k} Y^{s+v+i+j-k},\label{cheee}
\end{equation}
where
$$
c_{ijk} = k!\binom{s+i}{k}\binom{u/l-j\sigma/\rho}{k} - k!\binom{v+j}{k}\binom{r/l-i\sigma/\rho}{k}.
$$
Note that $c_{ij0} = 0$. Assume now that $w(P)\nsim w(Q)$. Since $\rho>0$ this implies $c_{001}\ne 0$, and so
$$
\ell_{\rho,\sigma}([P,Q]) =\sum_{i=0}^{\alpha}\sum_{j=0}^{\beta}\lambda_i\mu_j c_{ij1} x^{\frac{r+u}{l}-\frac{(i+j)\sigma} {\rho}-1} y^{s+v+i+j-1},
$$
because $\rho+\sigma>0$. Using again that $\rho>0$ and $\rho+\sigma>0$, we obtain that
$$
w([P,Q]) =\left(\frac{r+u}{l}-1,s+v-1\right) = w\bigl(\ell_{\rho,\sigma}([P,Q])\bigr),
$$
which proves~(1). For statement~(2) note that since $\rho>0$ and $\rho+\sigma>0$, it follows from~\eqref{cheee}, that
$$
w(P)+w(Q)-(1,1)=w\bigl(\ell_{\rho,\sigma}([P,Q])\bigr)\Longrightarrow c_{001}\ne 0,
$$
and then $w(P)\nsim w(Q)$.
\end{proof}

\begin{proposition}\label{extremosnoalineados} Let $P,Q,R\!\in\! A_1^{(l)}\setminus \{0\}$ be such that $[P,Q]_{\rho,\sigma}\!=\!\ell_{\rho,\sigma}(R)$, where $(\rho,\sigma)\in\mathfrak{V}^0$. We have:

\begin{enumerate}

\smallskip

\item $\st_{\rho,\sigma}(P)\nsim\st_{\rho,\sigma}(Q)$ if and only if $\st_{\rho,\sigma}(P) +\st_{\rho,\sigma}(Q) - (1,1) =\st_{\rho,\sigma}(R)$.

\smallskip

\item $\en_{\rho,\sigma}(P)\nsim\en_{\rho,\sigma}(Q)$ if and only if $\en_{\rho,\sigma}(P)+\en_{\rho,\sigma}(Q)-(1,1)=\en_{\rho,\sigma}(R)$.

\end{enumerate}
\end{proposition}
\begin{figure}[h]
\centering
\begin{tikzpicture}
\draw[step=.5cm,gray,very thin] (0,0) grid (4.3,4.7);
\draw [->] (0,0) -- (5,0) node[anchor=north]{$X$};
\draw [->] (0,0) --  (0,5) node[anchor=east]{$Y$};
\draw (0,0.5) -- (1.5,2) node[fill=white,right=2pt]{$\st_{2,1}(P)$} -- (0.5,4)  node[fill=white, above right=0.5pt]{$\en_{2,1}(P)$} -- (0,4);
\draw[thick] (1.5,2) -- (0.5,4);
\filldraw [gray] (1.5,2)    circle (2pt)
                  (0.5,4)   circle (2pt)
                  (2.5,1)   circle (2pt)
                  (3.5,2.5) circle (2pt)
                  (2.5,4.5) circle (2pt);
\draw (2.5,0) -- (2.5,1) node[fill=white,right=2pt]{$\st_{2,1}(Q)=\en_{2,1}(Q)$} -- (1.5,1.5) -- (0,1.5);
\draw[thick] (3.5,2.5)node[fill=white,right=2pt]{$\st_{2,1}(R)$} -- node[fill=white,above right=1pt]{$\ell_{2,1}(R)$} (2.5,4.5) node[fill=white, above right=1pt]{$\en_{2,1}(R)$};
\end{tikzpicture}
\caption{Proposition~\ref{extremosnoalineados}.}
\end{figure}

\begin{proof} We only are going to prove the first statement since the second one is similar. Let $P_1$ and $Q_1$ be $(\rho,\sigma)$-homogeneous elements of $A_1^{(l)}\setminus\{0\}$, such that
\begin{equation}
v_{\rho,\sigma}(P-P_1) < v_{\rho,\sigma}(P_1)\quad\text{and}\quad v_{\rho,\sigma}(Q-Q_1) < v_{\rho,\sigma}(Q_1).\label{eq13}
\end{equation}
Since
$$
[P,Q] = [P_1,Q_1] + [P_1,Q-Q_1] + [P-P_1,Q],
$$
and, by Remark~\ref{re v de un conmutador},
\begin{align*}
v_{\rho,\sigma}([P_1,Q-Q_1])&\le v_{\rho,\sigma}(P_1) + v_{\rho,\sigma}(Q-Q_1) - (\rho+\sigma)\\
& < v_{\rho,\sigma}(P_1) + v_{\rho,\sigma}(Q_1) - (\rho+\sigma)\\
& = v_{\rho,\sigma}(P) + v_{\rho,\sigma}(Q) - (\rho+\sigma)
\shortintertext{and}
v_{\rho,\sigma}([P-P_1,Q]) & < v_{\rho,\sigma}(P) + v_{\rho,\sigma}(Q) - (\rho+\sigma),
\end{align*}
from the fact that $P$ and $Q$ are not $(\rho,\sigma)$-proportional, it follows that
\begin{equation}
v_{\rho,\sigma}([P,Q]-[P_1,Q_1]) < v_{\rho,\sigma}([P,Q]).\label{eq14}
\end{equation}
By inequalities~\eqref{eq13} and~\eqref{eq14},
$$
\ell_{\rho,\sigma}(P) =\ell_{\rho,\sigma}(P_1),\quad\ell_{\rho,\sigma}(Q) =\ell_{\rho,\sigma}(Q_1)\quad\text{and}\quad\ell_{\rho,\sigma}([P,Q]) =\ell_{\rho,\sigma}([P_1,Q_1]),
$$
which implies
\begin{align*}
&\st_{\rho,\sigma}(P) =\st_{\rho,\sigma}(P_1) = w(P_1),\qquad\st_{\rho,\sigma}(Q) =\st_{\rho,\sigma}(Q_1) = w(Q_1)
\shortintertext{and}
& w\bigl(\ell_{\rho,\sigma}([P_1,Q_1])\bigr) = w\bigl(\ell_{\rho,\sigma}([P,Q])\bigr) = w\bigl(\ell_{\rho,\sigma}(R)\bigr) =\st_{\rho,\sigma}(R).
\end{align*}
Consequently, if $\st_{\rho,\sigma}(P)\nsim\st_{\rho,\sigma}(Q)$, then by statement~(1) of Proposition~\ref{pr conmutadores con exponentes fraccinarios} and statement~(1) of Lemma~\ref{pr conmutadores de elementos homogeneos con exponentes fraccinarios},
$$
\st_{\rho,\sigma}(P)+\st_{\rho,\sigma}(Q)-(1,1) = w([P_1,Q_1]) =\st_{\rho,\sigma}(R),
$$
as desired. Conversely, if $\st_{\rho,\sigma}(P)+\st_{\rho,\sigma}(Q)-(1,1)=\st_{\rho,\sigma}(R)$, then
$$
w(P_1)+w(Q_1)-(1,1)= w\bigl(\ell_{\rho,\sigma}([P_1,Q_1])\bigr),
$$
which, by statement~(2) of Lemma~\ref{pr conmutadores de elementos homogeneos con exponentes fraccinarios}, implies $w(P_1)\nsim w(Q_1)$, as we want.
\end{proof}

\begin{proposition}\label{calculo del corchete} Let $(\rho,\sigma)\in\mathfrak{V}^0$ and $P,Q\in A_1^{(l)}\setminus\{0\}$. If
$$
\qquad\ell_{\rho,\sigma}(P) =\sum_{i=0}^{\alpha}\lambda_i x^{\frac{r}{l}-\frac{i\sigma} {\rho}} y^{s+i}\quad\text{and}\quad\ell_{\rho,\sigma}(Q) =\sum_{j=0}^{\beta}\mu_j x^{\frac{u}{l}-\frac{j\sigma}{\rho}} y^{v+j},
$$
with $\lambda_0,\lambda_{\alpha},\mu_0,\mu_{\beta}\ne 0$, then
$$
[P,Q]_{\rho,\sigma} =\sum\lambda_i\mu_j c_{ij} x^{\frac{r+u}{l}-\frac{(i+j)\sigma} {\rho}-1} y^{s+v+i+j-1},
$$
where $c_{ij} =\bigl(\frac{u}{l}-\frac{j\sigma}{\rho},v+j\bigr)\times\bigl(\frac{r}{l} -\frac{i\sigma}{\rho},s+i\bigl)$.
\end{proposition}

\begin{proof} Write
$$
\qquad P =\sum_{i=0}^{\alpha}\lambda_i X^{\frac{r}{l}-\frac{i\sigma} {\rho}} Y^{s+i} + R_P\quad\text{and}\quad Q =\sum_{j=0}^{\beta}\mu_j X^{\frac{u}{l}-\frac{j\sigma}{\rho}} Y^{v+j} + R_Q.
$$
Since $v_{\rho,\sigma}(R_P) < v_{\rho,\sigma}(P)$ and $v_{\rho,\sigma}(R_Q) < v_{\rho,\sigma}(Q)$, from Re\-mark~\ref{re v de un conmutador} it follows that
\begin{equation}
[P,Q] =\sum_{i=0}^{\alpha}\sum_{j=0}^{\beta}\lambda_i\mu_j\left[X^{\frac{r}{l} -\frac{i\sigma}{\rho}} Y^{s+i}, X^{\frac{u}{l}-\frac{j\sigma}{\rho}} Y^{v+j}\right] + R,
\label{eq1}
\end{equation}
with $v_{\rho,\sigma}(R) < v_{\rho,\sigma}(P) + v_{\rho,\sigma}(Q) - (\rho +\sigma)$. On the other hand, since, by Lemma~\ref{le conmutacion en A_1l},
$$
X^{\frac{i}{l}}Y^j X^{\frac{i'}{l}}Y^{j'} =\sum_{k=0}^j  k!\binom{j}{k}\binom{i'/l}{k} X^{\frac{i+i'}{l}-k}Y^{j+j'-k},
$$
and $\rho+\sigma>0$, we have
\begin{equation}
\left[X^{\frac{r}{l} -\frac{i\sigma}{\rho}} Y^{s+i}, X^{\frac{u}{l}-\frac{j\sigma}{\rho}} Y^{v+j}\right] = c_{ij} X^{\frac{r+u}{l}-\frac{(i+j)\sigma} {\rho}-1} Y^{s+v+i+j-1} + R_{ij},\label{eq2}
\end{equation}
with $v_{\rho,\sigma}(R_{ij}) < v_{\rho,\sigma}(P) + v_{\rho,\sigma}(Q) - (\rho +\sigma)$. Combining~\eqref{eq1} with~\eqref{eq2}, we obtain that
$$
[P,Q]=\sum_{i=0}^{\alpha}\sum_{j=0}^{\beta}\lambda_i\mu_j c_{ij} X^{\frac{r+u}{l}-\frac{(i+j)\sigma} {\rho}-1} Y^{s+v+i+j-1} + R_{PQ},
$$
with $v_{\rho,\sigma}(R_{PQ}) < v_{\rho,\sigma}(P) + v_{\rho,\sigma}(Q) - (\rho +\sigma)$. Using now that
$$
v_{\rho,\sigma}\left(X^{\frac{r+u}{l}-\frac{(i+j)\sigma} {\rho}-1} Y^{s+v+i+j-1}\right) = v_{\rho,\sigma}(P) + v_{\rho,\sigma}(Q) - (\rho +\sigma),
$$
the result follows immediately.
\end{proof}

\begin{corollary}\label{ell depende del ell} Let $(\rho,\sigma)\in\mathfrak{V}^0$ and $P,Q,P_1,Q_1\in A_1^{(l)}\setminus\{0\}$. If
$$
\ell_{\rho,\sigma}(P) =\ell_{\rho,\sigma}(P_1)\qquad\text{and}\qquad\ell_{\rho,\sigma}(Q) =\ell_{\rho,\sigma}(Q_1),
$$
then $[P,Q]_{\rho,\sigma}=[P_1,Q_1]_{\rho,\sigma}$.
\end{corollary}

\begin{proof} By Proposition~\ref{calculo del corchete}.
\end{proof}

\begin{corollary}\label{extremos alineados} Let $(\rho,\sigma)\in\mathfrak{V}^0$ and $P,Q\in A_1^{(l)}\setminus\{0\}$. If $[P,Q]_{\rho,\sigma}= 0$, then
$$
\st_{\rho,\sigma}(P)\sim\st_{\rho,\sigma}(Q)\quad\text{and}\quad\en_{\rho,\sigma}(P) \sim\en_{\rho,\sigma}(Q).
$$
\end{corollary}

\begin{proof} This follows immediately from Proposition~\ref{calculo del corchete}, since
$$
\st_{\rho,\sigma}(P)\times\st_{\rho,\sigma}(Q)=c_{00}\quad\text{and}\quad \en_{\rho,\sigma}(P)\times\en_{\rho,\sigma}(Q)=c_{\alpha\beta},
$$
where we are using the same notations as in the statement of that result.
\end{proof}

\begin{definition}\label{polinomio asociado f^{(l)}} Let $P\in L^{(l)}\setminus\{0\}$ and $(\rho,\sigma)\in \mathfrak{V}^0$. If
$$
\ell_{\rho,\sigma}(P) =\sum_{i=0}^{\gamma} a_i x^{\frac{r}{l} -\frac{i\sigma}{\rho}} y^{s+i}\qquad\text{with $a_0,a_{\gamma}\ne 0$},
$$
we set
$$
f^{(l)}_{P,\rho,\sigma} :=\sum_{i=0}^{\gamma} a_i x^i\in K[x].
$$
Furthermore, for $P\in A_1^{(l)}$ we set $f^{(l)}_{P,\rho,\sigma} := f^{(l)}_{\Psi^{(l)}(P),\rho,\sigma}$. Note that $f^{(l)}_{P,\rho,\sigma} = f^{(l)}_{\ell_{\rho,\sigma}(P),\rho,\sigma}$
\end{definition}

\begin{remark}\label{notita} Note that
\begin{equation}
\st_{\rho,\sigma}(P) =\Bigl(\frac{r}{l},s\Bigr),\quad\en_{\rho,\sigma}(P) =\Bigl(\frac{r}{l}-\frac{\gamma\sigma}{\rho},s +\gamma\Bigr)\quad\text{and}\quad\ell_{\rho,\sigma}(P)= x^{\frac{r}{l}}y^s f^{(l)}_{P,\rho,\sigma}(x^{-\frac{\sigma} {\rho}}y).\label{eq57}
\end{equation}
\end{remark}

\begin{remark}\label{f de un producto} Let $(\rho,\sigma)\in\mathfrak{V}^0$. From Proposition~\ref{pr v de un producto} it follows immediately that
\begin{equation*}
f^{(l)}_{PQ,\rho,\sigma}=f^{(l)}_{P,\rho,\sigma}f^{(l)}_{Q,\rho,\sigma}\quad\text{for $P,Q\in A_1^{(l)}\setminus\{0\}$.}
\end{equation*}
The same result holds for $P,Q\in L^{(l)}\setminus\{0\}$.
\end{remark}

\smallskip

Statement~(2) of the following theorem justifies the terminology ``$(\rho,\sigma)$-proportional'' introduced in Definition~\ref{def rho-sigma proporcionales}.

\begin{theorem} (Compare with~\cite{D}*{Lemma 2.7.ii$_3$})\label{f[] en A_1^{(l)}} Let $P,Q\in A_1^{(l)}\setminus\{0\}$ and let $(\rho,\sigma)\in \mathfrak{V}^0$. Set $a:=\frac{1} {\rho} v_{\rho,\sigma}(Q)$ and $b:=\frac{1}{\rho} v_{\rho,\sigma}(P)$.

\begin{enumerate}

\smallskip

\item If $[P,Q]_{\rho,\sigma}\ne 0$, then there exist $h\in\mathds{N}_0$ and $c\in\mathds{Z}$, such that
$$
x^h f_{[P,Q]} = cf_P f_Q+ ax f'_P f_Q-bxf'_Qf_P,
$$
where $f_P:=f^{(l)}_{P,\rho,\sigma}$, $f_Q:=f^{(l)}_{Q,\rho,\sigma}$ and $f_{[P,Q]}:=f^{(l)}_{[P,Q],\rho,\sigma}$.

\smallskip

\item If $[P,Q]_{\rho,\sigma}= 0$ and $a,b>0$, then there exist $\lambda_P,\lambda_Q\in K^{\times}$, $m,n\in\mathds{N}$ and a $(\rho,\sigma)$-ho\-mo\-ge\-neous polynomial $R\!\in\! L^{(l)}$, with $\gcd(m,n)\! =\! 1$ and $m/n\! =\! b/a$, such that
$$
\ell_{\rho,\sigma}(P) =\lambda_P R^m\quad\text{and}\quad\ell_{\rho,\sigma}(Q) =\lambda_Q R^n.
$$

\end{enumerate}

\end{theorem}

\begin{proof} Write
$$
\qquad\ell_{\rho,\sigma}(P) =\sum_{i=0}^{\alpha}\lambda_i x^{\frac{r}{l}-\frac{i\sigma} {\rho}} y^{s+i}\quad\text{and}\quad\ell_{\rho,\sigma}(Q) =\sum_{j=0}^{\beta}\mu_j x^{\frac{u}{l}-\frac{j\sigma}{\rho}} y^{v+j},
$$
with $\lambda_0,\lambda_\alpha,\mu_0,\mu_\beta\ne 0$. By Proposition~\ref{calculo del corchete},
$$
[P,Q]_{\rho,\sigma} =\sum\lambda_i\mu_j c_{ij} x^{\frac{r+u}{l}-\frac{(i+j)\sigma} {\rho}-1} y^{s+v+i+j-1},
$$
where $c_{ij} :=\bigl(\frac{u}{l}-\frac{j\sigma}{\rho},v+j\bigr)\times\bigl(\frac{r}{l}- \frac{i\sigma}{\rho},s+i\bigr)$. Set
$$
F(x):=\sum_{i,j}\lambda_i\mu_j c_{ij} x^{i+j}.
$$
Note that if $[P,Q]_{\rho,\sigma} = 0$, then $F=0$, and if $[P,Q]_{\rho,\sigma}\ne 0$, then $F = x^h f_{[P,Q]}$, where $h$ is the multiplicity of $x$ in $F$. Also note that
$$
a=\left(u/l,v\right)\times\left(-\sigma/\rho,1\right)\quad\text{and}\quad b= -\left(-\sigma/\rho,1\right)\times \left(r/l,s\right).
$$
Let $c:=\left(u/l,v\right)\times\left(r/l,s\right)$. Clearly $c_{ij} = c+ia-jb$. Since
$$
\sum_{i,j}\lambda_i\mu_j x^{i+j} = f_Pf_Q ,\quad\sum_{i,j} i\lambda_i\mu_j x^{i+j}=xf_P'f_Q\quad\text{and} \quad\sum_{i,j} j\lambda_i\mu_j x^{i+j} = xf_Q'f_P,
$$
we have
\begin{align}
F = cf_Pf_Q+axf_P'f_Q-bxf_Q'f_P.\label{eq8}
\end{align}
Statement~(1) follows immediately from this fact. Assume now that $[P,Q]_{\rho,\sigma}= 0$ and that $a,b>0$. In this case $F = 0$ and, in particular, $c = c_{00} =\frac{F(0)}{\lambda_0\mu_0} = 0$. Hence,~\eqref{eq8} becomes
\begin{equation}
af_P'f_Q-bf_Q'f_P = 0.\label{eq9}
\end{equation}
Let $\bar{l}\in\mathds{N}$ be such that $\bar a:=\bar{l}a$ and $\bar b:=\bar{l}b$ are natural numbers. Since~\eqref{eq9} implies $(f_P^{\bar a}/f_Q^{\bar b})'=0$, there exists $\lambda\in K^\times$, such that $f_P^{\bar a} =\lambda f_Q^{\bar b}$. Hence, there are $g\in K[x]$ and $\lambda_P,\lambda_Q\in K^\times$, sa\-tis\-fying
\begin{equation}
f_P =\lambda_P g^m\quad\text{and}\quad f_Q=\lambda_Q g^n,\label{eq10}
\end{equation}
where $m: =\bar b/\gcd(\bar a,\bar b)$ and $n :=\bar a/\gcd(\bar a,\bar b)$. By Remark~\ref{notita} and Corollary~\ref{extremos alineados}
$$
\Bigl(\frac r l, s\Bigr) = \st_{\rho,\sigma}(P)\sim  \st_{\rho,\sigma}(Q) = \Bigl(\frac u l, v\Bigr)
$$
So, since $a,b>0$, there exists $\lambda\in \mathds{Q}$ such that
$\bigl(r/l, s\bigr) = \lambda \bigl(u/l, v\bigr)$. Applying $v_{\rho,\sigma}$ to this equality we obtain that $\rho b = \lambda \rho a$, from which $\lambda = b/a = m/n$. Thus $m(u,l v) = n(r,l s)$, and so there exists $(p,\bar q)\in\mathds{Z} \times \mathds{N}_0$, such that
\begin{equation*}
(u,l v) = n(p,\bar q)\quad\text{and}\quad (r,l s) = m(p,\bar q).
\end{equation*}
In particular $l|n\bar q$ and $l|m\bar q$, and so $l|\bar q$, since $m$ and $n$ are coprime. Hence,
\begin{equation}
\left(u/l,v\right) = n\left(p/l,q\right)\quad\text{and}\quad\left(r/l,s\right) = m\left(p/l,q\right),\label{eq11}
\end{equation}
where $q:=\bar q/l$. If $g =\sum_{i=0}^{\gamma}\nu_i x^i$ with $\nu_{\gamma}\ne 0$, then, by~\eqref{eq10} and~\eqref{eq11},
$$
R:=\sum_{i=0}^{\gamma}\nu_i x^{\frac{p}{l}-i\frac{\sigma}{\rho}} y^{q+i}
$$
satisfies
$$
\ell_{\rho,\sigma}(P)= x^{\frac{r}{l}}y^s f_P(x^{-\frac{\sigma}{\rho}}y) =\lambda_P\bigl(x^{\frac{p}{l}}y^q g(x^{-\frac{\sigma}{\rho}}y)\bigr)^m =\lambda_P R^m
$$
and
$$
\ell_{\rho,\sigma}(Q)= x^{\frac{u}{l}}y^v f_Q(x^{-\frac{\sigma}{\rho}}y) =\lambda_Q\bigl(x^{\frac{p}{l}}y^q g(x^{-\frac{\sigma}{\rho}}y)\bigr)^n =\lambda_Q R^n.
$$
In order to finish the proof of statement~(2) it suffices to check that $R\in L^{(l)}$. First note that $R$ belongs to the field of fractions of $L^{(l)}$, because $R^m,R^n\in L^{(l)}$ and $\gcd(m,n) = 1$. But then $R\in L^{(l)}$, since $R^m\in L^{(l)}$.
\end{proof}

\section{Order on directions}

\setcounter{equation}{0}
In this section we establish an order relation on the directions and we associate with each $P\in A_1^{(l)}$ a finite ordered set of directions $\Dir(P)$. The main result is Proposition~\ref{le basico}, which states that if $(\rho_1,\sigma_1)< (\rho_2,\sigma_2)$ are two consecutive directions in $\Dir(P)$, then $\st_{\rho_2,\sigma_2}(P)$ coincides with $en_{\rho_1,\sigma_1}(P)$ and with $\Supp(\ell_{\rho,\sigma}(P))$ for each intermediate direction $(\rho,\sigma)$. Another useful result is Proposition~\ref{le basico1}, which compares $v_{\rho',\sigma'}\bigl(\st_{\rho,\sigma}(P)\bigr)$ with $v_{\rho',\sigma'} \bigl(\en_{\rho,\sigma}(P)\bigr)$ for $(\rho',\sigma')\ne(\rho,\sigma)$.

\smallskip

We define an order relation on $\mathfrak{V}$ by setting $(\rho_1,\sigma_1)\le (\rho,\sigma)$ if $(\rho_1,\sigma_1) \times (\rho,\sigma)\ge 0$. We can extend this order to all of $\ov{\mathfrak{V}}$ by setting
$$
(1,-1)<(\rho,\sigma) <(-1,1)\quad\text{for all $(\rho,\sigma)\in \mathfrak{V}$.}
$$
Note that if $(\rho,\sigma),(\rho_1\sigma_1)\in\ov{\mathfrak{V}}$ and $\{(\rho,\sigma), (\rho_1\sigma_1)\}\ne \{(1,-1),(-1,1)\}$, then
$$
(\rho_1,\sigma_1)<(\rho,\sigma) \Longleftrightarrow (\rho_1,\sigma_1)\times (\rho,\sigma)>0.
$$
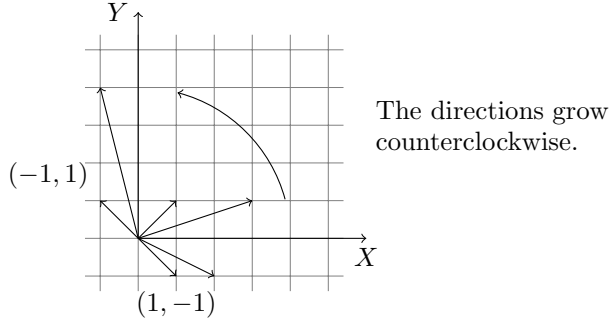
\begin{figure}[h]
\centering
\begin{tikzpicture}
\draw[step=.5cm,gray,very thin] (-0.7,-0.7) grid (2.7,2.7);
\draw [->] (0,0) -- (3,0) node[anchor=north]{$X$};
\draw [->] (0,0) --  (0,3) node[anchor=east]{$Y$};
\draw[->] (0,0) -- (0.5,-0.5) node[below=2pt]{$(1,-1)$};
\draw[->] (0,0) -- (1,-0.5);
\draw[->] (0,0) -- (1.5,0.5);
\draw[->] (0,0) -- (0.5,0.5);
\draw[->] (0,0) -- (-0.5,2);
\draw[->] (0,0) -- (-0.5,0.5)node[above left=1pt]{$(-1,1)$};
\draw[->] (15:2cm) arc (15:75:2cm);
\draw[xshift=3cm, yshift=1.5cm]
node[right,text width=3.7cm]{The directions grow counterclockwise.};
\end{tikzpicture}
\caption{Order relation in $\overline{\mathfrak{V}}$.}
\end{figure}

\begin{definition}\label{forma debil} Let $P\in A_1^{(l)}\setminus\{0\}$. We define the set of {\em directions associated with} $P$ as
$$
\Dir(P):=\{(\rho,\sigma)\in\mathfrak{V}:\#\Supp(\ell_{\rho,\sigma}(P))>1\},
$$
and we set $\ov{\Dir}(P):=\Dir(P)\cup\{(1,-1),(-1,1)\}$. We make a similar definition for $P\in L^{(l)}\setminus\{0\}$.
\end{definition}

For each $(r/l,s)\in\frac{1}{l}\mathds{Z}\times\mathds{Z}\setminus\mathds{Z}(1,1)$ there exists a unique $(\rho,\sigma)\in\mathfrak{V}$, denoted by $\dir(r/l,s)$, such that $v_{\rho,\sigma}(r/l,s)=0$. In fact clearly
\begin{equation}
(\rho,\sigma) :=\begin{cases}\left(-ls/d,r/d\right)&\text{ if $r-ls>0$,}\\ \left(ls/d,-r/d\right)&\text{ if $r-ls<0$,}\end{cases}\label{val}
\end{equation}
where $d:=\gcd(r,ls)$, satisfies the required condition, and the uniqueness is evident.
\begin{figure}[h]
\centering
\begin{tikzpicture}
\draw[step=.5cm,gray,very thin] (0,0) grid (3.3,3.3);
\draw [->] (0,0) -- (3.5,0) node[anchor=north]{$X$};
\draw [->] (0,0) --  (0,3.5) node[anchor=east]{$Y$};
\draw (2,0) -- (2.5,0.5) -- (3,1.5) -- (2.5,2.5) -- (1,3) -- (0,3);
\filldraw [gray] (2.5,0.5)  circle (2pt)
                  (3,1.5)   circle (2pt)
                  (2.5,2.5) circle (2pt)
                  (1,3)     circle (2pt)
                  (0,3)     circle (2pt);
\draw[xshift=3.2cm, yshift=2cm]
node[right,text width=8.4cm]{$\overline{\Dir}(P)=\{(1,-1),(2,-1),(2,1),(1,3),(1,0),(-1,1)\}$.};
\end{tikzpicture}
\caption{Definition~\ref{forma debil}.}
\end{figure}

\begin{remark}\label{valuacion depende de extremos} Note that
$$
(\rho,\sigma) = \dir\bigl(\en_{\rho,\sigma}(P) -\st_{\rho,\sigma}(P)\bigr)\qquad \text{for all $P\in A_1^{(l)}\setminus\{0\}$ and $(\rho,\sigma)\in\Dir(P)$}.
$$
\end{remark}

Our next purpose is to prove Proposition~\ref{le basico} below. For $P\in A_1^{(l)}\setminus\{0\}$ and $(\rho,\sigma)\in\mathfrak{V}$, we consider the following two sets of directions:
$$
\Dirsup_P(\rho,\sigma):=\left\{\dir\left(\left(i/l,j\right)-\en\right):\left(i/l,j \right)\in\Supp(P)\text{ and } v_{-1,1}\left(i/l,j\right)>v_{-1,1}(\en)\right\}
$$
and
$$
\Dirinf_P(\rho,\sigma):=\left\{\dir\left(\left(i/l,j\right)-\st\right):\left(i/l,j \right)\in\Supp(P)\text{ and } v_{1,-1}\left(i/l,j\right)> v_{1,-1}(\st)\right\},
$$
where for the sake of brevity we set $\en:=\en_{\rho,\sigma}(P)$ and $\st:=\st_{\rho,\sigma}(P)$.

\begin{lemma}\label{separacion} Let $P\in A_1^{(l)}\setminus\{0\}$ and $(\rho,\sigma)\in\mathfrak{V}$.
\begin{enumerate}

\smallskip

\item If $(\rho_1,\sigma_1)\in\Dirsup_P(\rho,\sigma)$, then $(\rho_1,\sigma_1) > (\rho,\sigma)$.

\smallskip

\item If $(\rho_1,\sigma_1)\in\Dirinf_P(\rho,\sigma)$, then $(\rho_1,\sigma_1) < (\rho,\sigma)$.

\end{enumerate}

\end{lemma}

\begin{proof} We only prove statement~(1) and leave the other one to the reader. Clearly, if
$$
(i/l,j)\in\Supp(P)\quad\text{and}\quad v_{\rho,\sigma}(i/l,j) = v_{\rho,\sigma}(P),
$$
then $(i/l,j)\in\Supp\bigl(\ell_{\rho,\sigma}(P)\bigr)$, and so $v_{-1,1}(i/l,j)\le v_{-1,1}(\en)$.  Consequently, if
$$
(i/l,j)\in\Supp(P)\quad\text{and}\quad v_{-1,1}(i/l,j) > v_{-1,1}(\en),
$$
then $v_{\rho,\sigma}(i/l,j) < v_{\rho,\sigma}(P) = v_{\rho,\sigma}(\en)$. This means
\begin{equation}
v_{\rho,\sigma}(a,b)<0,\label{eq ij menor que P}
\end{equation}
where $(a,b) := (i/l,j)-\en$. Note that $v_{-1,1}(i/l,j)> v_{-1,1}(\en)$ now reads
$$
b-a = v_{-1,1}(a,b)>0.
$$
But then
$$
(\rho_1,\sigma_1):=\dir\bigl((i/l,j)-\en\bigr) =\dir(a,b)=\lambda (b,-a),
$$
for some $\lambda>0$. Hence
$$
0 > v_{\rho,\sigma}(a,b)= a\rho+b\sigma= -\frac {1}{\lambda}(\sigma_1\rho-\rho_1\sigma)
= -\frac {1}{\lambda}(\rho,\sigma)\times (\rho_1,\sigma_1).
$$
This yields $(\rho,\sigma)\times(\rho_1,\sigma_1)>0$, and so $(\rho_1,\sigma_1)>(\rho,\sigma)$, as desired.
\end{proof}

\begin{lemma}\label{le recorte de dominio} Let $P$, $(\rho,\sigma)$, $\st$ and $\en$ be as before. We have:
\begin{enumerate}

\smallskip

\item If $(i/l,j)\in\Supp(P)$, $(\rho',\sigma')>(\rho,\sigma)$ and $v_{-1,1}(i/l,j)\le v_{-1,1}(\en)$, then
\begin{equation}
v_{\rho',\sigma'}(i/l,j)\le v_{\rho',\sigma'}(\en).\label{eqnu1}
\end{equation}
Moreover, if $(\rho',\sigma')\ne (-1,1)$, then equality holds if and only if $\left(i/l,j\right)=\en$.

\smallskip

\item If $(i/l,j)\in\Supp(P)$, $(\rho',\sigma')<(\rho,\sigma)$ and $v_{1,-1}(i/l,j)\le v_{1,-1}(\st)$, then
$$
v_{\rho',\sigma'}(i/l,j)\le v_{\rho',\sigma'}(\st).
$$
Moreover, if $(\rho',\sigma')\ne (1,-1)$, then equality holds if and only if $\left(i/l,j\right)=\st$.
\end{enumerate}
\end{lemma}

\begin{proof} We prove statement~(1) and leave the proof of statement~(2), which is similar, to the reader. Set $(a,b):=\left(i/l,j\right)-\en$. Since, by  hypothesis, $\rho\sigma'-\sigma\rho' >0$ and $b-a\le 0$, it is true that
\begin{equation}
b\rho\sigma'+\sigma\rho_1 a-a\rho\sigma'-b\sigma\rho'\le 0,\label{eq condiciones del lemma}
\end{equation}
and the equality holds if and only if $b\!=\!a$. We also know that $v_{\rho,\sigma} (i/l,j)\le v_{\rho,\sigma}(\en)$, which means that $\rho a+\sigma b\le 0$. Since $\rho'+\sigma'\ge 0$, we obtain
\begin{equation}
\rho'\rho a+\sigma'\sigma b+\rho'\sigma b+\sigma'\rho a =(\rho a+\sigma b)(\rho'+\sigma')\le 0.\label{eq28}
\end{equation}
Summing up~\eqref{eq condiciones del lemma} and~\eqref{eq28}, we obtain
$$
0\ge\rho\rho' a+\sigma\sigma' b+\rho\sigma' b+\sigma\rho' a=(\rho+\sigma)(\rho' a+\sigma' b),
$$
and so $v_{\rho',\sigma'}(a,b)\le 0$, as desired. Moreover, if the equality is true, then~\eqref{eq condiciones del lemma} is also an equality, and so $b=a$. Hence $0 = v_{\rho',\sigma'}(a,a) = (\rho'+\sigma') a$, which implies that $a=0$ or $(\rho',\sigma') = (-1,1)$. Thus, is $(\rho',\sigma')\ne (-1,1)$ and equality holds in~\eqref{eqnu1}, then $(i/l,j) =\en$.
\end{proof}

\begin{notation} Let $P$ and $(\rho,\sigma)$ be as before.

\begin{itemize}

\smallskip

\item[-] If $\Dirsup_P(\rho,\sigma)\ne\emptyset$, then we set $\Succ_P(\rho,\sigma):=\min\Dirsup_P(\rho, \sigma)$.

\smallskip

\item[-] if $\Dirinf_P(\rho,\sigma)\ne\emptyset$, then we set $\Pred_P(\rho,\sigma):=\max\Dirinf_P(\rho, \sigma)$.

\end{itemize}

\end{notation}

\begin{lemma}\label{predysucc} Let $P$, $(\rho,\sigma)$, $\st$ and $\en$ be as before. We have:

\begin{enumerate}

\smallskip

\item $\Succ_P(\rho,\sigma)\in\Dir(P)$ and $\en=\st_{\Succ_P(\rho,\sigma)}(P)$.

\smallskip

\item $\Pred_P(\rho,\sigma)\in\Dir(P)$ and $\st=\en_{\Pred_P(\rho,\sigma)}(P)$.
\end{enumerate}
\end{lemma}

\begin{proof} We only prove statement~(1), since~(2) is similar. Set  $(\rho_1,\sigma_1):=\Succ_P(\rho,\sigma)$. By definition, there exists an $(i_0/l,j_0)\in\Supp(P)$, such that
$$
v_{-1,1}(i_0/l,j_0)>v_{-1,1}(\en)\quad\text{and}\quad (\rho_1,\sigma_1)= \dir\bigl((i_0/l,j_0)-\en\bigr).
$$
Therefore,
\begin{equation}
(i_0/l,j_0)\ne\en\quad\text{and}\quad v_{\rho_1,\sigma_1}(\en) = v_{\rho_1,\sigma_1}(i_0/l,j_0).\label{eq29}
\end{equation}
Hence $(\rho_1,\sigma_1)\ne(-1,1)$, since otherwise, $v_{\rho_1,\sigma_1}(\en) < v_{\rho_1,\sigma_1}(i_0/l,j_0)$. We claim that
$$
v_{\rho_1,\sigma_1}(P) = v_{\rho_1,\sigma_1}(\en),
$$
which, by~\eqref{eq29}, proves that $(\rho_1,\sigma_1)\in\Dir(P)$. In fact, assume on the contrary that there exists $(i/l,j)\in\Supp(P)$ with
\begin{equation}\label{ineq}
v_{\rho_1,\sigma_1}(i/l,j)>v_{\rho_1,\sigma_1}(\en)
\end{equation}
By statement~(1) of Lemma~\ref{separacion} and statement~(1) of Lemma~\ref{le recorte de dominio},
$$
v_{-1,1}(i/l,j) > v_{-1,1}(\en),
$$
and consequently $(a,b):=(i/l,j)-\en$ satisfies $b-a>0$. Hence
$$
(\rho_2,\sigma_2):=\dir\bigl((i/l,j)-\en\bigr) =\dir(a,b)=\lambda (b,-a)
$$
with $\lambda>0$. Now~\eqref{ineq} leads to
$$
0 < (\rho_1,\sigma_1).(a,b)=\frac{1}{\lambda}(\rho_2\sigma_1-\sigma_2\rho_1)
=\frac{1}{\lambda}(\rho_2,\sigma_2)\times(\rho_1,\sigma_1),
$$
which implies that $(\rho_2,\sigma_2)<(\rho_1,\sigma_1)$. But this fact is impossible, since $(\rho_1,\sigma_1)$ is minimal in $\Dirsup_P(\rho,\sigma)$ and $(\rho_2,\sigma_2)\in\Dirsup_P(\rho,\sigma)$. This proves the claim and so $\Succ_P(\rho,\sigma)\in\Dir(P)$. Finally we will check that $\en=\st_{\rho_1,\sigma_1}(P)$. For this, it suffices to prove that
$$
(i/l,j)\in\Supp(\ell_{\rho_1,\sigma_1}(P))\Longrightarrow v_{1,-1}(i/l,j)\le v_{1,-1}(\en)
$$
or, equivalently, that $v_{-1,1}(i/l,j)\ge v_{-1,1}(\en)$. To do this we first note that by statement~(1) of Lemma~\ref{separacion} we have $(\rho_1,\sigma_1) > (\rho,\sigma)$. Since, moreover $(i/l,j)\in\Supp(P)$ and $(\rho_1,\sigma_1)\ne (-1,1)$, using statement~(1) of Lemma~\ref{le recorte de dominio}, it follows that
$$
v_{-1,1}(i/l,j)< v_{-1,1}(\en)\Longrightarrow v_{\rho_1,\sigma_1}(i/l,j)< v_{\rho_1,\sigma_1}(\en),
$$
which is a contradiction.
\end{proof}

\begin{proposition}\label{le basico} Let $P\!\in\! A_1^{(l)}\!\setminus\!\{0\}$ and let $(\rho_1,\sigma_1)\!>\!(\rho_2,\sigma_2)$ be consecutive elements in~$\ov{\Dir}(P)$.

\begin{enumerate}

\smallskip

\item If $(\rho_1,\sigma_1)\in\Dir(P)$ and $(\rho_1,\sigma_1)> (\rho,\sigma)\ge (\rho_2,\sigma_2)$, then $(\rho_1,\sigma_1) = \Succ_P(\rho,\sigma)$.

\smallskip

\item If $(\rho_2,\sigma_2)\in\Dir(P)$ and $(\rho_1,\sigma_1)\ge (\rho,\sigma) > (\rho_2,\sigma_2)$, then $(\rho_2,\sigma_2) = \Pred_P(\rho,\sigma)$.

\smallskip

\item If $(\rho_1,\sigma_1)> (\rho,\sigma) >(\rho_2,\sigma_2)$, then $\{\st_{\rho_1,\sigma_1}(P)\} =\Supp(\ell_{\rho,\sigma}(P)) =\{\en_{\rho_2,\sigma_2}(P)\}$.

\end{enumerate}

\end{proposition}
\begin{figure}[h]
\centering
\begin{tikzpicture}
\draw[step=.5cm,gray,very thin] (0,0) grid (3.2,3.4);
\draw [->] (0,0) -- (4,0) node[anchor=north]{$X$};
\draw [->] (0,0) --  (0,3.7) node[anchor=east]{$Y$};
\draw (3,0) -- (2.5,2) -- (1.5,3) -- (0,2.5);
\draw[very thin] (0.5,3.25) -- (2.5,2.75);
\draw [->]  (0,0) -- (-0.5,1.5);
\draw [->]  (0,0) -- (1,1);
\draw [->]  (0,0) -- (0.5,2);
\draw (-1.2,1) node {$(\rho_1,\sigma_1)$};
\draw (1,1.5) node[fill=white] {$(\rho,\sigma)$};
\draw (1.5,0.5) node[fill=white] {$(\rho_2,\sigma_2)$};
\filldraw [gray] (1.5,3) circle (2pt)  node[black, above=2pt, fill=white]{$A$};
\draw[xshift=3.9cm, yshift=2cm]
node[right,text width=4.5cm]{$A:=\st_{\rho_1,\sigma_1}(P)=\en_{\rho_2,\sigma_2}(P)$};
\draw[xshift=3.9cm, yshift=1.5cm]
node[right,text width=4.5cm]{$\{A\}=\Supp(\ell_{\rho,\sigma}(P))$};
\end{tikzpicture}
\caption{Proposition~\ref{le basico}.}
\end{figure}

\begin{proof} (1)\enspace By statement~(1) of Lemma~\ref{separacion} and statement~(1) of Lemma~\ref{predysucc}, the existence of $\Succ_P(\rho,\sigma)$ implies
$$
(\rho,\sigma)<\Succ_P(\rho,\sigma)\quad\text{and}\quad\Succ_P(\rho,\sigma)\in\Dir(P).
$$
Hence $(\rho_1,\sigma_1)\le\Succ_P(\rho,\sigma)$. Consequently, we are reduced to prove that $\Succ_P(\rho,\sigma)$ exists and that $(\rho_1,\sigma_1)\ge\Succ_P(\rho,\sigma)$. For the existence it suffices to check that $\Dirsup_P(\rho,\sigma)\ne\emptyset$. Assume on the contrary that $\Dirsup_P(\rho,\sigma) =\emptyset$. Then, by definition
$$
v_{-1,1}(i/l,j)\le v_{-1,1}(\en_{\rho,\sigma})(P)\quad\text{for all $(i/l,j)\in\Supp(P)$.}
$$
Therefore, since $(\rho_1,\sigma_1)\ne (-1,1)$, from statement~(1) of Lemma~\ref{le recorte de dominio} it follows that
$$
\Supp(\ell_{\rho_1,\sigma_1}(P)) =\{\en_{\rho,\sigma}(P)\},
$$
and so $(\rho_1,\sigma_1)\notin\Dir(P)$, which is a contradiction. Now we will prove that $(\rho_1,\sigma_1)\ge\Succ_P(\rho,\sigma)$. Since $(\rho_1,\sigma_1)$ is the minimum element of $\Dir(P)$ that is greater than $(\rho,\sigma)$, it suffices to prove that there exists no $(\rho_3,\sigma_3)\in\Dir(P)$ such that $\Succ_{\rho,\sigma}(P)> (\rho_3,\sigma_3)>(\rho,\sigma)$. In other words that
\begin{equation*}
\Succ_{\rho,\sigma}(P)> (\rho_3,\sigma_3)>(\rho,\sigma)\Longrightarrow (\rho_3,\sigma_3)\notin\Dir(P).
\end{equation*}
So assume that $\Succ_P(\rho,\sigma)> (\rho_3,\sigma_3) > (\rho,\sigma)$ and take  $(i/l,j)\in \Supp(\ell_{\rho_3,\sigma_3}(P))$. We claim that $(i/l,j)=\en_{\rho,\sigma}(P)$, which will show that $\Supp(\ell_{\rho_3,\sigma_3}(P))=\{\en_{\rho,\sigma}(P)\}$, and consequently, that $(\rho_3,\sigma_3)\notin \Dir(P)$. If $v_{-1,1}(i/l,j)\le v_{-1,1}(\en_{\rho,\sigma}(P))$, then the claim follows from statement~(1) of Lemma~\ref{le recorte de dominio}, applied to $(\rho_3,\sigma_3)$ instead of $(\rho_1,\sigma_1)$. Thus we can assume without loss of generality, that $v_{-1,1}(i/l,j)\ge v_{-1,1}(\en_{\rho,\sigma}(P))$. Now, by statement~(1) of Lemma~\ref{predysucc}, we know that $\st_{\Succ_P(\rho,\sigma)}(P) =\en_{\rho,\sigma}(P)$, and so
$$
v_{1,-1}(i/l,j)\le v_{1,-1}(\en) = v_{1,-1}\bigl(\st_{\Succ_P(\rho,\sigma)}(P)\bigr).
$$
Consequently, applying statement~(2) of Lemma~\ref{le recorte de dominio}, with $\Succ_P(\rho,\sigma)$ instead of $(\rho,\sigma)$ and $(\rho_3,\sigma_3)$ instead of $(\rho_1,\sigma_1)$, and taking into account that $(i/l,j)\in\Supp(\ell_{\rho_3,\sigma_3}(P))$, we obtain
$$
v_{\rho_3,\sigma_3}(i/l,j) = v_{\rho_3,\sigma_3}\bigl(\st_{\Succ_{\rho,\sigma}(P)}(P)\bigr).
$$
But, since $(\rho_3,\sigma_3)\ne (1,-1)$, from statement~(2) of Lemma~\ref{le recorte de dominio} it follows  that
$$
(i/l,j) =\st_{\Succ_{\rho,\sigma}(P)}(P) =\en_{\rho,\sigma}(P),
$$
which proves the claim.

\smallskip

\noindent (2)\enspace It is similar to the proof of statement~(1).

\smallskip

\noindent (3)\enspace Since $\{\en_P(\rho,\sigma)\} =\Supp(\ell_{\rho,\sigma}(P))$, from statement~(1) and  Lemma~\ref{predysucc} it follows that
$$
(\rho_1,\sigma_1)\in\Dir(P) \Longrightarrow\{\st_{\rho_1,\sigma_1}(P)\} =\Supp(\ell_{\rho,\sigma}(P)).
$$
This conclude the proof of the first equality in~(3) when $(\rho_1,\sigma_1)<(-1,1)$. A symmetric argument shows that
$$
(\rho_2,\sigma_2)\in\Dir(P) \Longrightarrow\{\en_{\rho_2,\sigma_2}(P)\} =\Supp(\ell_{\rho,\sigma}(P)),
$$
which gives the second equality in~(3) when $(\rho_2,\sigma_2)>(1,-1)$. Assume now that
$$
(\rho_1,\sigma_1)=(-1,1)\qquad\text{and}\qquad (\rho_2,\sigma_2)\ne (1,-1).
$$
Then $\Dirsup_{\rho_2,\sigma_2}(P)=\emptyset$ by statement~(1) of Lemma~\ref{separacion} and statement~(1) of Lemma~\ref{predysucc}. Hence
$$
v_{-1,1}(i/l,j)\le v_{-1,1}(\en_{\rho_2,\sigma_2}(P)),
$$
for all $(i/l,j)\in\Supp(P)$. Consequently, $\en_{\rho_2,\sigma_2}(P)\in\Supp\bigl(\ell_{-1,1} (P)\bigr)$, and so
$$
\st_{-1,1}(P) =\en_{\rho_2,\sigma_2}(P)+(a,a),
$$
for some $a\ge 0$. But necessarily $a=0$, since $a>0$ leads to the contradiction
$$
v_{\rho_2,\sigma_2}(\st_{-1,1}(P))= v_{\rho_2,\sigma_2}(\en_{\rho_2,\sigma_2}(P)+(a,a))= v_{\rho_2,\sigma_2}(P)+ a(\rho_2+\sigma_2).
$$
Thus
$$
\{\st_{\rho_1,\sigma_1}(P)\} =\{\en_{\rho_2,\sigma_2}(P)\} =\Supp(\ell_{\rho,\sigma}(P)).
$$
Similarly, if $(\rho_1,\sigma_1)\ne (-1,1)$ and $(\rho_2,\sigma_2)= (1,-1)$, then
$$
\{\st_{\rho_1,\sigma_1}(P)\} =\{\en_{\rho_2,\sigma_2}(P)\} =\Supp(\ell_{\rho,\sigma}(P)).
$$
Finally assume that $(\rho_1,\sigma_1) = (-1,1)$ and $(\rho_2,\sigma_2) = (1,-1)$. Then $\Dir(P) = \emptyset$ and hence, by Lemma~\ref{predysucc} we know that $\Dirsup_P(\rho,\sigma) =\Dirinf_P(\rho,\sigma)=\emptyset$. Therefore
\begin{equation}
v_{-1,1}(P) = v_{-1,1}(\en_{\rho,\sigma}(P))\quad\text{and}\quad v_{1,-1}(P) = v_{1,-1}(\st_{\rho,\sigma}(P)).\label{eq30}
\end{equation}
But, since $\en_{\rho,\sigma}(P)=\st_{\rho,\sigma}(P)$, equalities~\eqref{eq30} imply that $P =\ell_{-1,1}(P)$. So,
$$
\{\en_{1,-1}(P)\}=\{w(P)\}=\{\ov{w}(P)\}=\{\st_{-1,1}(P)\} =\Supp\bigl(\ell_{\rho,\sigma}(P)\bigr),
$$
as desired.
\end{proof}

\begin{remark}\label{sucesor y predecesor} Let $(\rho,\sigma)\in \mathfrak{V}$. By statement~(1) of Lemma~\ref{separacion}, statement~(1) of Lemma~\ref{predysucc} and statement~(1) of Proposition~\ref{le basico}, we know that
$$
\text{there exists } \Succ_P(\rho,\sigma) \Leftrightarrow \Dirsup_P(\rho,\sigma)\ne \emptyset\Leftrightarrow \{(\rho',\sigma')\in \Dir(P):(\rho,\sigma)< (\rho',\sigma')\}\ne \emptyset,
$$
and, in this case,
$$
\Succ_P(\rho,\sigma) = \min\{(\rho',\sigma')\in \Dir(P):(\rho,\sigma)< (\rho',\sigma')\}.
$$
Similarly,
$$
\text{there exists }\Succ_P(\rho,\sigma) \Leftrightarrow  \Dirinf_P(\rho,\sigma)\ne \emptyset\Leftrightarrow \{(\rho',\sigma')\in \Dir(P):(\rho',\sigma')< (\rho,\sigma)\} \ne \emptyset,
$$
and, in this case,
$$
\Pred_P(\rho,\sigma) = \max\{(\rho',\sigma')\in \Dir(P):(\rho',\sigma')< (\rho,\sigma)\}.
$$
\end{remark}

\begin{proposition}\label{le basico1} Let $P\in A_1^{(l)}\setminus\{0\}$ and $(\rho,\sigma)\in Val(P)$. We have:

\begin{enumerate}

\smallskip

\item If $(\rho',\sigma')\in\ov{\mathfrak{V}}$ satisfy $(\rho,\sigma)<(\rho',\sigma')$, then $v_{\rho',\sigma'}\bigl(\st_{\rho,\sigma}(P)\bigr) < v_{\rho',\sigma'}\bigl(\en_{\rho,\sigma}(P)\bigr)$,

\smallskip

\item If $(\rho',\sigma')\in\ov{\mathfrak{V}}$ satisfy $(\rho',\sigma')<(\rho,\sigma)$, then $v_{\rho',\sigma'}\bigl(\st_{\rho,\sigma}(P)\bigr) > v_{\rho',\sigma'}\bigl(\en_{\rho,\sigma}(P)\bigr)$.

\smallskip

\end{enumerate}
The same properties hold for $P\in L^{(l)}\setminus\{0\}$.
\end{proposition}

\begin{proof} We only prove the first statement because the second one is similar. If $(\rho',\sigma')\ne (-1,1)$ we apply statement~(1) of Lemma~\ref{le recorte de dominio} with $(i/l,j) = \st_{\rho,\sigma}(P)$. The case $(\rho',\sigma') = (-1,1)$ is straightforward, since $\rho+\sigma>0$ and $\st_{\rho,\sigma}(P) \ne \en_{\rho,\sigma}(P)$.
\end{proof}

\section{Fixed points of $\bm{(\rho,\sigma)}$-brackets}

\setcounter{equation}{0}

In this section we make explicit in formula~\eqref{eq central} a result of~\cite{J}, and analyze its consequences on the shape of minimal pairs.

\smallskip

By~\cite{J}*{Prop.~3.2} the Poisson bracket defined there in ~\cite{J}*{p.~599}, satisfies
$$
\{\ell_{\rho,\sigma}(P),\ell_{\rho,\sigma}(Q)\}=-[P,Q]_{\rho,\sigma}\qquad\text{for all $(\rho,\sigma)\in \mathfrak{V}$ and $P,Q\in A_1^{(l)}$}.
$$
Assume now that $[Q,P]=1$. It is clear that the conditions of~\cite{J}*{Lemma 3.4} are satisfied for $x=P$, $y=Q$, $r=\rho$ and $s=\sigma$. By~\cite{J}*{Corollary 3.5}, there exists an $R\in K[P,Q]$ such that $[P,R]_{\rho,\sigma}\ne 0$ and
$$
[P,(\Psi^{(l)})^{-1}([P,R]_{\rho,\sigma})]_{\rho,\sigma}=0.
$$
Consequently, if $v_{\rho,\sigma}(P)>0$, then we can apply~\cite{J}*{Lemma 2.2} with $f=\ell_{\rho,\sigma}(P)$, $g=\ell_{\rho,\sigma}(R)$, $r=\rho$ and $s=\sigma$. So, there exists a $(\rho,\sigma)$-homogeneous element $h\in L^{(l)}$ with $v_{\rho,\sigma}(h) = \rho+\sigma$, such that
$$
\{\ell_{\rho,\sigma}(P),h\} = \ell_{\rho,\sigma}(P).
$$
Moreover, if $P$ and $Q$ are in $A_1$, then $h\in L$.

\begin{theorem}\label{central} Let $P\in A_1^{(l)}$ and let $(\rho,\sigma)\in \mathfrak{V}$ be such that $v_{\rho,\sigma}(P)>0$. If $[Q,P]=1$ for some $Q\in A_1^{(l)}$, then there exists a $(\rho,\sigma)$-ho\-mo\-ge\-neous element $F\in A_1^{(l)}$ such that $v_{\rho,\sigma}(F)=\rho+\sigma$ and
%
%
%
\begin{equation}\label{eq central}
[P,F]_{\rho,\sigma}=\ell_{\rho,\sigma}(P).
\end{equation}
Moreover, we have
\begin{enumerate}

\smallskip

\item $\st_{\rho,\sigma}(P)\sim\st_{\rho,\sigma}(F)$ or $\st_{\rho,\sigma}(F)=(1,1)$.

\smallskip

\item $\en_{\rho,\sigma}(P)\sim\en_{\rho,\sigma}(F)$ or $\en_{\rho,\sigma}(F)=(1,1)$.

\smallskip

\item $\st_{\rho,\sigma}(P)\nsim (1,1)\nsim\en_{\rho,\sigma}(P)$.

\smallskip

\item If $P,Q\in A_1$, then we can take $F\in A_1$.

\end{enumerate}
\end{theorem}

\begin{proof} Let $h$ be as above and let $F:=-(\Psi^{(l)})^{-1}(h)$. Equality~\eqref{eq central} follows easily from the previous discussion. Statements~(1) and~(2) follow directly from~\eqref{eq central} and Proposition~\ref{extremosnoalineados}. For the third statement, assume that $\st_{\rho,\sigma}(P)\sim (1,1)$.
We claim that this implies that $\st_{\rho,\sigma}(F) = (1,1)$. Otherwise, by statement~(1) we have
$$
\st_{\rho,\sigma}(F)\sim \st_{\rho,\sigma}(P)\sim (1,1),
$$
which implies $\st_{\rho,\sigma}(F)\sim (1,1)$, since $\st_{\rho,\sigma}(F)\ne (0,0)\ne \st_{\rho,\sigma}(P)$. So there exists $\lambda \in \mathds{Q}\setminus\{1\}$ such that $\st_{\rho,\sigma}(F) = \lambda (1,1)$. But this is impossible because $v_{\rho,\sigma}(F) = \rho + \sigma$. Hence the claim is true, and so
$$
\st_{\rho,\sigma}(P) + \st_{\rho,\sigma}(F)-(1,1) = \st_{\rho,\sigma}(P),
$$
which by Proposition~\ref{extremosnoalineados} leads to the contradiction
$$
\st_{\rho,\sigma}(P)\nsim \st_{\rho,\sigma}(F) = (1,1).
$$
Similarly $\en_{\rho,\sigma}(P)\nsim (1,1)$. Finally, statement~(4) is an immediate consequence of the fact that if $P,Q\in A_1$, then $h\in L$.
\end{proof}

\begin{definition}\label{def ecuacion PE} Let $k\!\in\!\mathds{N}$, $j\!\in\! \mathds{N}_0$, $\varepsilon,b\!>\!0$ and $c\!\in\!\mathds{Q}$. We say that a pair $(f,g)$ of poly\-nomials in $K[x]$ satisfies $\PE(k,j,\varepsilon,b,c)$ if there is some $h\in\mathds{N}_0$, such that
\begin{equation}
x^hf^{k+j}=cf^kg+ax(f^k)'g-bxg'f^k\label{eqdifferencial}
\end{equation}
is satisfied, where $a=\frac{j}{k}b+\varepsilon$.
\end{definition}

Note that equation~\eqref{eqdifferencial} implies that each irreducible factor of $g$ that does not divide $xf$ has multiplicity~$1$.

\begin{proposition} (Compare with~\cite{M}*{Appendix I})\label{pr conmutadormantienealC} If the pair $(f,g)$ satisfy $\PE(k,j,\varepsilon, b,c)$ and $f(0)\ne 0\ne g(0)$, then each irreducible factor $u$ of $f$, with multiplicity $m_u$ in $f$, has multiplicity $jm_u+1$ in $g$. Consequently $g = f^j\ov{g}$ for some $\ov{g}\in K[x]$ separable and the number of different irreducible factors of $f$ is lower than or equal to the degree of $\ov{g}$.

\end{proposition}

\begin{proof} We can assume that $K$ is algebraically closed. Take an irreducible monic factor $u$ of $f$. Since $f(0)\ne 0$, there exists $d\in K^{\times}$ such that $u = x+d$. Write $f = u^s\ov{f}$ and $g = u^r\ov{g}$, with $r,s\in\mathds{N}$ such that $u$ does not divide $\ov{f}\ov{g}$. Then
$$
f'=su^{s-1}u'\ov{f}+u^s\ov{f}'\quad\text{and}\quad g'=ru^{r-1}u'\ov{g}+u^r\ov{g}',
$$
and so~\eqref{eqdifferencial} reads
$$
x^h\ov{f}^{k+j}u^{s(k+j)}=u^{r+ks-1}(aks-br)xu'\ov{f}^k\ov{g}+u^{r+ks}\ov{f}^{k-1} \bigl(\ov{f}\ov{g}c + kax\ov{g}\ov{f}' -bx\ov{f}\ov{g}'\bigr).
$$
We claim that $aks-br\ne 0$. In fact, on the contrary $s(k+j)\ge r+ks$ and so $sj\ge r$. Since $\varepsilon, b,s,k>0$, this leads to the contradiction
$$
aks-br=\left(\frac jk b+\varepsilon\right)ks-br=ks\varepsilon+b(js-r)\ge ks\varepsilon >0.
$$
Since $u$ does not divide $xu'\ov{f}^k\ov{g}$, we have $s(k+j) = r+ks-1$. That is $r=js+1$, which proves the first assertion. The remaining assertions now follow easily.
\end{proof}

\begin{corollary}\label{pavadass} Let $(\rho,\sigma)\in\mathfrak{V}^0$ and let $P,F\in A_1^{(l)}\setminus\{0\}$. Assume that $F$ is $(\rho,\sigma)$-homogeneous and that $[P,F]_{\rho,\sigma}=\ell_{\rho,\sigma}(P)$. Set
$$
f_F:=f^{(l)}_{F,\rho,\sigma}\quad\text{and}\quad f_P:=f^{(l)}_{P,\rho,\sigma},
$$
where $f^{(l)}_{F,\rho,\sigma}$ and $f^{(l)}_{P,\rho,\sigma}$ are as in Definition~\ref{polinomio asociado f^{(l)}}. Then
\begin{enumerate}

\smallskip

\item $f_F$ is separable and every irreducible factor of $f_P$ divides $f_F$.

\smallskip

\item Suppose that $f_F,f_P\in K[x^r]$ for some $r\in\mathds{N}$ and let $\ov f_F$ and $f_F$ denote the univariate polynomials defined by $f_P(x)=\ov f_P(x^r)$ and $f_F(x)=\ov f_F(x^r)$. Then $\ov f_F$ is separable and every irreducible factor of $\ov f_P$ divides $\ov f_F$.

\smallskip

\item If $P,F\in A_1$ and $v_{0,1}(\en_{\rho,\sigma} (F))-v_{0,1}(\st_{\rho,\sigma} (F))=\rho$, then the multiplicity of each linear factor (in an algebraic closure of $K$) of $f_P$ is equal to
$$
\frac{1}{\rho} \deg(f_P)= \frac{1}{\rho}\bigl(v_{0,1}(\en_{\rho,\sigma} (P))-v_{0,1}(\st_{\rho,\sigma}(P) \bigr).
$$
\end{enumerate}
\end{corollary}

\begin{proof} By statement~(1) of Theorem~\ref{f[] en A_1^{(l)}}, there exist $h\ge 0$ and $c\in\mathds{Z}$, such that
\begin{equation}
x^h f_P = cf_P f_F+ ax f'_P f_F-bxf'_Ff_P,\label{eqq1}
\end{equation}
where $a:=\frac{1}{\rho} v_{\rho,\sigma}(F)$ and $b:=\frac{1}{\rho} v_{\rho,\sigma}(P)$. So, the pair $(f_P,f_F)$ satisfies condition $\PE(1,0,a,b,c)$ of Definition~\ref{def ecuacion PE}. Consequently, since $f_P\ne 0\ne f_F$, statement~(1) follows from Proposition~\ref{pr conmutadormantienealC}. Using now that
$$
xf_P'(x)=rt\ov f_P'(t)\quad\text{and}\quad xf_F'(x)=rt\ov f_F'(t),\qquad\text{where $t = x^r$,}
$$
we deduce from~\eqref{eqq1} that the pair $(\ov f_P,\ov f_F)$ satisfies condition $\PE(1,0,ra,rb,c)$. Therefore, we can apply Proposition~\ref{pr conmutadormantienealC} to obtain statement~(2). Finally, we prove statement~(3). Write
$$
F = \sum_{i=0}^{\alpha} b_i X^{u-i\sigma} Y^{v+i\rho}\quad \text{and}\quad \ell_{\rho,\sigma}(P) = \sum_{i=0}^{\gamma} c_i x^{m-i\sigma} y^{n+i\rho}
$$
with $b_0\ne 0$, $b_{\alpha}\ne 0$, $c_0\ne 0$ and $c_{\gamma}\ne 0$. By definition
$$
f_F = \sum_{i=0}^{\alpha} b_ix^{i\rho}\quad\text{and}\quad f_P = \sum_{i=0}^{\gamma} c_ix^{i\rho}.
$$
Moreover, since $\alpha\rho = v_{0,1}(\en_{\rho,\sigma} (F))-v_{0,1}(\st_{\rho,\sigma} (F))$ it follows from the hypothesis that $\alpha=1$. Hence
$$
f_F (x)= b_0 + b_1 x^{\rho} = \mu (x^{\rho}-\lambda)=\ov f_F(x^{\rho}),
$$
where $\mu:= b_0$ and $\lambda:= b_0/b_1$. Consequently, from statement~(2) it follows that $f_P$ is of the form $f_P(x) = \mu_P(x^{\rho}-\lambda)^{\gamma}$, which proves~(3).
\end{proof}

\section{Cutting the right lower edge}

\setcounter{equation}{0}

The central result of this section is Proposition~\ref{preparatoria}, which loosely spoken ``cuts'' the right lower edge of the support of a pair $(P,Q)$ satisfying certain conditions. In Proposition~\ref{lema general} we use this result and establish a strong algebraic condition on the ``corner'' resulting from this cut.

%
%

\smallskip

Given $\varphi\in\Aut(A_1^{(l)})$, we let $\varphi_L$ denote the automorphism of $L^{(l)}$ given by
$$
\varphi_L(x^{1/l}):=\Psi^{(l)}(\varphi(X^{1/l}))\qquad\text{and}\qquad \varphi_L(y):=\Psi^{(l)} (\varphi(Y)).
$$

\begin{proposition}\label{pr ell por automorfismos} Let $(\rho,\sigma)\in\mathfrak{V}^0$ and $\lambda\in K$. Assume that $\rho|l$ and consider the automorphism of $A_1^{(l)}$ defined by $\varphi(X^{1/l}) := X^{1/l}$ and $\varphi(Y) := Y+\lambda X^{\sigma/\rho}$. Then
$$
\ell_{\rho,\sigma}(\varphi(P)) =\varphi_L(\ell_{\rho,\sigma}(P))\quad\text{and}\quad v_{\rho,\sigma}(\varphi(P)) = v_{\rho,\sigma}(P)\quad\text{for all $P\in A_1^{(l)}\setminus\{0\}$}.
$$
Furthermore,
$$
\ell_{\rho_1,\sigma_1}(\varphi(P))=\ell_{\rho_1,\sigma_1}(P)\quad\text{for all $(\rho,\sigma)< (\rho_1,\sigma_1) < (-1,1)$.}
$$
\end{proposition}

\begin{proof} By statement~(3) of Proposition~\ref{pr v de un producto},
$$
v_{\rho,\sigma}\bigl(\varphi(X^{\frac{i}{l}} Y^j)\bigr) = i v_{\rho,\sigma}(X^{\frac{1}{l}})+ j v_{\rho,\sigma} (Y+\lambda X^{\frac{\sigma}{\rho}}) =\frac{i}{l}\rho + j\sigma = v_{\rho,\sigma}(X^{\frac{i}{l}} Y^j),
$$
for all $i\in\mathds{Z}$ and $j\in\mathds{N}_0$. Since $\varphi$ is injective this implies that
\begin{equation}
v_{\rho,\sigma}(\varphi(P)) = v_{\rho,\sigma}(P)\quad\text{for all $P\in A_1^{(l)}\setminus\{0\}$.}\label{eq23}
\end{equation}
We fix now a $P\in A_1^{(l)}\setminus\{0\}$ and write
$$
P =\sum_{i=0}^n\lambda_i X^{\frac{r}{l}-i\frac{\sigma}{\rho}}Y^{s+i} + R,
$$
with $v_{\rho,\sigma}(R) < v_{\rho,\sigma}(P)$. By equality~\eqref{eq23} and statement~(2) of Proposition~\ref{pr v de un producto}
\begin{align*}
\ell_{\rho,\sigma}(\varphi(P)) &=\ell_{\rho,\sigma}\left(\sum_{i=0}^n\lambda_i\varphi\bigl(X^{\frac{r}{l}- i\frac{\sigma}{\rho}}Y^{s+i}\bigr)\right)\\
& =\ell_{\rho,\sigma}\left(\sum_{i=0}^n\lambda_i X^{\frac{r}{l}-i\frac{\sigma}{\rho}}(Y+\lambda X^{\frac{\sigma}{\rho}})^{s+i}\right)\\
& =\sum_{i=0}^n\lambda_i\ell_{\rho,\sigma}\bigl(X^{\frac{r}{l}-i\frac{\sigma}{\rho}} (Y+\lambda X^{\frac{\sigma}{\rho}})^{s+i}\bigr)\\
& =\sum_{i=0}^n\lambda_i x^{\frac{r}{l}-i\frac{\sigma}{\rho}}(y+\lambda x^{\frac{\sigma}{\rho}})^{s+i}\\
& =\varphi_L\left(\sum_{i=0}^n\lambda_i x^{\frac{r}{l}-i\frac{\sigma}{\rho}} y^{s+i}\right)\\
& =\varphi_L(\ell_{\rho,\sigma}(P)),
\end{align*}
as desired. Let $(\rho_1,\sigma_1)\in\mathfrak{V}$ such that $(\rho,\sigma)<(\rho_1,\sigma_1)$. Then $\rho_1\sigma <\rho\sigma_1$, and so
$$
\ell_{\rho_1,\sigma_1}(Y+\lambda X^{\frac{\sigma}{\rho}}) = y,
$$
since $\rho>0$. Hence, by statement~(2) of Proposition~\ref{pr v de un producto},
$$
\ell_{\rho_1,\sigma_1}\bigl(\varphi(X^{\frac{i}{l}} Y^j)\bigr)=\ell_{\rho_1,\sigma_1}\bigl(X^{\frac{i}{l}} (Y+\lambda X^{\frac{\sigma}{\rho}})^j\bigr)= x^{\frac{i}{l}} y^j,
$$
and so $v_{\rho_1,\sigma_1}\bigl(\varphi(X^{\frac{i}{l}} Y^j)\bigr) =\frac{i}{l}\rho_1 + j\sigma_1 = v_{\rho_1,\sigma_1}(X^{\frac{i}{l}} Y^j)$, which implies that
\begin{equation}
v_{\rho_1,\sigma_1}(\varphi(R)) = v_{\rho_1,\sigma_1}(R)\quad\text{for all $R\in A_1^{(l)}\setminus \{0\}$.}\label{eq24}
\end{equation}
Fix now $P\in A_1^{(l)}\setminus\{0\}$ and write
$$
P =\sum_{\{(i/l,j):\rho_1 i/l +\sigma_1 j = v_{\rho_1,\sigma_1}(P)\}}\lambda_{i/l,j} X^{\frac{i}{l}}Y^j + R,
$$
with $v_{\rho_1,\sigma_1}(R)<v_{\rho_1,\sigma_1}(P)$. Again by equality~\eqref{eq23} and statement~(2) of Proposition~\ref{pr v de un producto}
\begin{align*}
\ell_{\rho_1,\sigma_1}(\varphi(P)) &=\ell_{\rho_1,\sigma_1}\biggl(\sum_{i/l,j} \lambda_{i/l,j} \varphi(X^{\frac{i}{l}}Y^j) \biggr)\\
& =\ell_{\rho_1,\sigma_1}\biggl(\sum_{i/l,j}\lambda_{i/l,j} X^{\frac{i}{l}}(Y+\lambda X^{\frac{\sigma}{\rho}})^j\biggr)\\
& =\sum_{i/l,j}\lambda_{i/l,j}\ell_{\rho_1,\sigma_1}\bigl(X^{\frac{i}{l}}(Y+\lambda X^{\frac{\sigma}{\rho}})^j\bigr)\\
& =\sum_{i/l,j}\lambda_{i/l,j} x^{\frac{i}{l}}y^j\\
& =\ell_{\rho_1,\sigma_1}(P),
\end{align*}
as we want.
\end{proof}

\begin{remark}\label{re ell por automorfismos} A similar argument shows that for $\lambda\in K$ and $n\in\mathds{N}$, the automorphism of $A_1$ defined by $\varphi(X):=X$ and $\varphi(Y):= Y-\lambda X^n$ satisfies
$$
\ell_{1,n}(\varphi(P)) = \varphi_L(\ell_{1,n}(P))\quad\text{and}\quad v_{1,n}(\varphi(P)) = v_{1,n}(P)\quad\text{for all $P\in A_1\setminus\{0\}$}.
$$
Similarly, if $\varphi$ is the automorphism of $A_1$ defined by $\varphi(X):=X-\lambda Y^n$ and $\varphi(Y) := Y$, then
$$
\ell_{n,1}(\varphi(P)) = \varphi_L(\ell_{n,1}(P))\quad\text{and}\quad v_{n,1}(\varphi(P)) = v_{n,1}(P)\quad\text{for all $P\in A_1\setminus\{0\}$}.
$$
\end{remark}

From now on we assume that $K$ is algebraically closed unless otherwise stated.

\smallskip

Let $P,Q\in A_1^{(l)}$ and let $(\rho,\sigma)\in\mathfrak{V}^0$. Write
$$
\st_{\rho,\sigma}(P) =\Bigl(\frac{r}{l},s\Bigr)\quad\text{and}\quad\mathfrak{f}(x) := x^s f^{(l)}_{P,\rho,\sigma}(x).
$$
Let $\varphi\in\Aut(A_1^{(l')})$ be the automorphism defined by
$$
\varphi(X^{\frac{1}{l'}}) := X^{\frac{1}{l'}}\quad\text{and}\quad\varphi(Y):= Y +\lambda X^{\frac{\sigma}{\rho}},
$$
where $l':=\lcm(l,\rho)$ and $\lambda$ is any element of $K$ such that the multiplicity $m_{\lambda}$ of $x-\lambda$ in $\mathfrak{f}(x)$ is maximum.

\begin{proposition} (Compare with~\cite{J}*{Corollary 2.6})\label{preparatoria} If
\begin{itemize}

\smallskip

\item[\rm{(a)}] $[Q,P]=1$,

\smallskip

\item[\rm{(b)}] $(\rho,\sigma)\in\Dir(P)\cap\Dir(Q)$,

\smallskip

\item[\rm{(c)}] $v_{\rho,\sigma}(P)>0$ and $v_{\rho,\sigma}(Q)>0$,

\smallskip

\item[\rm{(d)}] $[P,Q]_{\rho,\sigma}=0$,

\smallskip

\item[\rm{(e)}] $\frac{v_{\rho,\sigma}(Q)}{v_{\rho,\sigma}(P)}\notin\mathds{N}$ and $\frac{v_{\rho,\sigma}(P)}{v_{\rho,\sigma}(Q)}\notin\mathds{N}$,

\smallskip

\item[\rm{(f)}] $v_{1,-1}\bigl(\en_{\rho,\sigma}(P)\bigr)<0$ and $v_{1,-1}\bigl(\en_{\rho,\sigma}(Q)\bigr)< 0$,

\smallskip

\end{itemize}
then, there exists $(\rho',\sigma')\in\mathfrak{V}$ such that
\begin{enumerate}

\smallskip

\item $(\rho',\sigma')<(\rho,\sigma)$ and $(\rho',\sigma')\in\Dir(\varphi(P))\cap\Dir(\varphi(Q))$,

\smallskip

\item $v_{1,-1}\bigl(\en_{\rho',\sigma'}(\varphi(P))\bigr)<0$ and $v_{1,-1}\bigl(\en_{\rho',\sigma'}(\varphi(Q))\bigr)<0$,

\smallskip

\item $v_{\rho',\sigma'}(\varphi(P))>0$ and $v_{\rho',\sigma'}(\varphi(Q))>0$,

\smallskip

\item $\frac{v_{\rho',\sigma'}(\varphi(P))}{v_{\rho',\sigma'}(\varphi(Q))}= \frac{v_{\rho,\sigma}(P)} {v_{\rho,\sigma}(Q)}$,

\smallskip

\item for all $(\rho,\sigma) <(\rho'',\sigma'') < (-1,1)$ the equalities
$$
\ell_{\rho'',\sigma''}(\varphi(P))=\ell_{\rho'',\sigma''}(P)\quad\text{and}\quad \ell_{\rho'',\sigma''}(\varphi(Q)) =\ell_{\rho'',\sigma''}(Q)
$$
hold,

\smallskip

\item $\en_{\rho',\sigma'}(\varphi(P)) =\st_{\rho,\sigma}(\varphi(P)) =\Bigl(\frac{r}{l} +\frac{s\sigma}{\rho}-m_{\lambda} \frac{\sigma}{\rho}, m_{\lambda}\Bigr)$,

\smallskip

\item $\en_{\rho',\sigma'}(\varphi(Q)) =\st_{\rho,\sigma}(\varphi(Q))$ and $\en_{\rho',\sigma'}(\varphi(P)) =\frac{v_{\rho,\sigma}(P)}{v_{\rho,\sigma}(Q)}\en_{\rho',\sigma'}(\varphi(Q))$,

\smallskip

\item it is true that
$$
\qquad v_{0,1}(\en_{\rho',\sigma'}(\varphi(P))) <v_{0,1} (\en_{\rho,\sigma}(P))\quad\text{or} \quad\en_{\rho',\sigma'}(\varphi(P)) =\en_{\rho,\sigma}(P).
$$
Furthermore, in the second case $\en_{\rho,\sigma}(P)+(\sigma/\rho,-1)\in\Supp(P)$,

\smallskip

\item $v_{\rho,\sigma}(\varphi(P)) = v_{\rho,\sigma}(P)$ and $v_{\rho,\sigma}(\varphi(Q)) = v_{\rho,\sigma}(Q)$,

\smallskip

\item $[\varphi(Q),\varphi(P)]_{\rho,\sigma} = 0$,

\smallskip

\item there exists  a $(\rho,\sigma)$-homogeneous element $F\in A_1^{(l)}$, which is not a monomial, such that
$$
[P,F]_{\rho,\sigma} =\ell_{\rho,\sigma}(P)\quad\text{and}\quad v_{\rho,\sigma}(F) =\rho+\sigma.
$$
Furthermore, if $\en_{\rho,\sigma}(F) = (1,1)$, then $\st_{\rho,\sigma}(\varphi(P)) =\en_{\rho,\sigma}(P)$,

\smallskip

\end{enumerate}
Note that

\begin{itemize}

\smallskip

\item[-] if $l'=1$, then $\varphi$ induces an automorphism of $A_1$,

\smallskip

\item[-] $\rho'>0$, since $\rho> 0$ means $(\rho,\sigma)<(0,1)$, and so $(\rho',\sigma')<(\rho,\sigma)< (0,1)$ implies $\rho'>0$,

\smallskip

\item[-] $v_{\rho,\sigma}(F) =\rho+\sigma>0$ implies that $\st_{\rho,\sigma}(F)\ne (0,0)\ne\en_{\rho,\sigma}(F)$.

\end{itemize}

\end{proposition}

\begin{figure}[h]
\centering

\begin{tikzpicture}[scale=0.75]
\draw[step=.5cm,gray,very thin] (0,0) grid (6.3,8.3);
\draw[step=.5cm,gray,very thin] (8,0) grid (14.3,8.3);
\draw [->] (1,0) -- (4,6) -- (4.5,8);
\draw (5.2,6) node[fill=white]{$\en_{\rho,\sigma}(P)$};
\draw (5.3,3.1) node[fill=white]{$\ell_{\rho,\sigma}$};
\draw[<-, very thin] (2.8,3.5) .. controls (4,3.1) ..  (4.8,3.1);
\draw[<-, very thin] (11.3,4.4) .. controls (12.5,4.7) ..  (13.8,4.7);
\fill[gray!15!white] (0,0) -- (1,0) -- (4,6) -- (4.4375,7.75)-- (0,7.75) -- (0,0);
\draw (1.5,4.3) node{$P$};
\draw [->] (8.5,0) -- (10.5,3) -- (12,6)--  (12.5,8);
\fill[gray!15!white] (8,0) -- (8.5,0) -- (10.5,3) -- (12,6)-- (12.4375,7.75) -- (8,7.75) -- (8,0);
\draw (9.5,4.3) node{$\varphi(P)$};
\draw [->] (0,0) -- (6.4,0) node[anchor=north]{$X$};
\draw [->] (8,0) -- (14.4,0) node[anchor=north]{$X$};
\draw [->] (0,0) -- (0,8.4) node[anchor=east]{$Y$};
\draw [->] (8,0) -- (8,8.4) node[anchor=east]{$Y$};
\draw [dashed] (9,0) -- (10.5,3);
\draw[<-, very thin] (9.4,1.3) .. controls (10.25,0.7) ..  (11.2,0.7);
\draw (12,0.7) node[fill=white]{$\ell_{\rho',\sigma'}$};
\draw (13.55,3.2) node[fill=white]{$\en_{\rho',\sigma'}\bigl(\varphi(P)\bigr)$};
\draw (13.6,2.4) node[fill=white] {$=\st_{\rho,\sigma}\bigl(\varphi(P)\bigr)$};
\draw[<-, very thin] (10.6,3) .. controls (11,2.8) ..  (11.9,2.75);
\draw[<-, very thin] (12.1,6) .. controls (12.5,6.2) ..  (13,6.25);
\draw (13.8,6.7) node[fill=white]{$\en_{\rho,\sigma}(P)$};
\draw (14.4,5.9) node[fill=white] {$=\en_{\rho,\sigma}\bigl(\varphi(P)\bigr)$};
\draw (14.4,4.7) node[fill=white]{$\ell_{\rho,\sigma}$};
\draw[->] (0,0) -- (1,-0.5);
\draw (1.7,-0.45) node{$(\rho,\sigma)$};
\draw[->] (8,0) -- (9,-0.5);
\draw (9.7,-0.45) node{$(\rho,\sigma)$};
\draw[->] (8,0) -- (9.5,-1);
\draw (10.35,-1) node{$(\rho',\sigma')$};
\end{tikzpicture}

\caption{Illustration of Proposition~\ref{preparatoria}}
\end{figure}
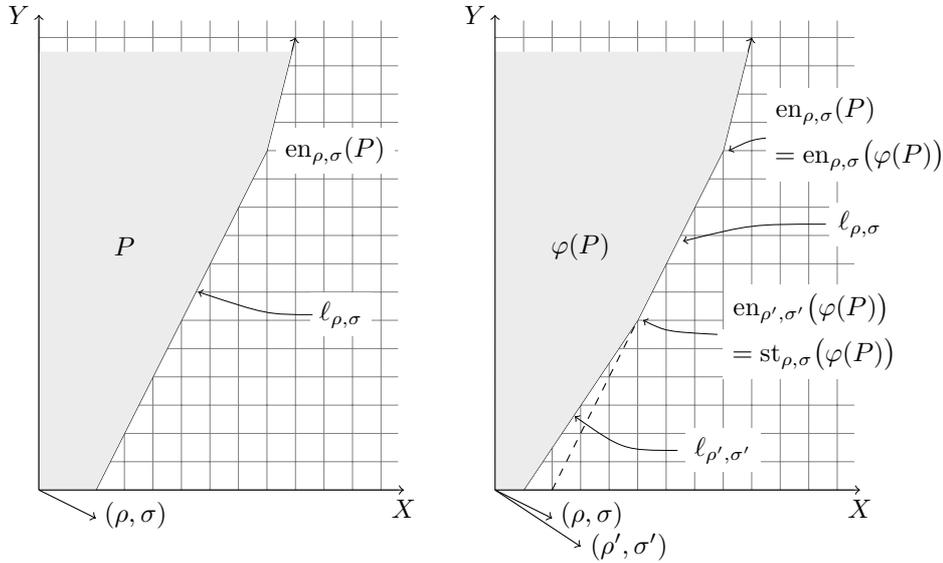

\begin{proof} By Theorem~\ref{central} we can find a $(\rho,\sigma)$-homogeneous element $F\in A_1^{(l)}$ such that
\begin{equation}
[P,F]_{\rho,\sigma} =\ell_{\rho,\sigma}(P)\qquad\text{and}\qquad v_{\rho,\sigma}(F) =\rho+\sigma.\label{eqnue5}
\end{equation}
Let $\#\factors(f^{(l)}_{P,\rho,\sigma})$ denote the number of linear different factors of $f^{(l)}_{P,\rho,\sigma}$. We claim that
\begin{align}
1\le\#\factors(f^{(l)}_{P,\rho,\sigma})\le\deg(f^{(l)}_{F,\rho,\sigma}),\label{eq36}
\end{align}
In fact, the first is true because $(\rho,\sigma)\in\Dir(P)$, while the second one follows from statement~(1) of Corollary~\ref{pavadass}. Note that, by the very definition of $f^{(l)}_{F,\rho,\sigma}$, condition~\eqref{eq36} implies that $F$ is not a monomial.

\smallskip

By~\eqref{eq57} and the definition of $\ell_{\rho,\sigma}(P)$ there exist $b_0,\dots,b_{\gamma}\in K$ with $b_0\ne 0$ and $b_{\gamma}\ne 0$, such that
$$
\ell_{\rho,\sigma}(P) =\sum_{i=0}^{\gamma} b_i x^{\frac{r}{l}-\frac{i\sigma}{\rho}} y^{s+i},
$$
and, again by~\eqref{eq57},
\begin{equation}
\en_{\rho,\sigma}(P) =\Bigl(\frac{r}{l}-\frac{\gamma\sigma}{\rho},s+\gamma\Bigr).\label{eqnue21}
\end{equation}
By Definition~\ref{polinomio asociado f^{(l)}},
$$
\mathfrak{f}(x) =\sum_{i=0}^{\gamma} b_i x^{i+s}.
$$
Let $(M_0,M) :=\en_{\rho,\sigma}(F)$. By the second equality in~\eqref{eqnue5}
\begin{equation}
M_0 =\frac{\rho+\sigma -\sigma M}{\rho}.\label{eqnue9}
\end{equation}
We assert that
\begin{equation}
1\le\#\factors(\mathfrak{f})\le M.\label{numero de factores}
\end{equation}
The first inequality is true by~\eqref{eq36}. In order to prove the second one we begin by noting that, by statement~(1) of Theorem~\ref{central},
\begin{equation}
\st_{\rho,\sigma}(F)=(1,1)\quad\text{or}\quad\st_{\rho,\sigma}(F)\sim \st_{\rho,\sigma}(P),\label{eqnue6}
\end{equation}
and that $\st_{\rho,\sigma}(F)\ne (0,0)$ by the second equality in~\eqref{eqnue5}. Hence, if $s>0$, then $\st_{\rho,\sigma}(F)\ne (u,0)$. Consequently, by~\eqref{eq57} and~\eqref{eq36},
$$
\#\factors(\mathfrak{f})=\#\factors(f^{(l)}_{P,\rho,\sigma})+1\le \deg(f^{(l)}_{F,\rho,\sigma})+1\le M.
$$
On the other hand, if $s=0$, then again by~\eqref{eq57} and~\eqref{eq36},
$$
\#\factors(\mathfrak{f})=\#\factors(f^{(l)}_{P,\rho,\sigma})\le \deg(f^{(l)}_{F,\rho,\sigma})\le M,
$$
as desired, proving the assertion.

\smallskip

For the sake of simplicity we set $N :=\gamma+s$. Since $\deg\mathfrak{f} = N$, by~\eqref{numero de factores} there exists at least one factor $x-\lambda$ of $\mathfrak{f}$ with multiplicity $m_{\lambda}$ greater than or equal to $N/M$. We take $\lambda\in K$ such that the multiplicity of $x-\lambda$ in $\mathfrak{f}(x)$ is maximum. We have
\begin{equation}
\mathfrak{f}(x) =\sum_{i = m_{\lambda}}^N a_i (x-\lambda)^i\quad\text{with $a_i\in K$, $a_{m_{\lambda}},a_N\ne 0$ and $m_{\lambda}\ge\frac{N}{M}$}.\label{eqnue8}
\end{equation}
Note that, since $\st_{\rho,\sigma}(P) = (r/l,s)$, by the third equality in~\eqref{eq57} we have
$$
\ell_{\rho,\sigma}(P) = x^{\frac{r}{l}}y^s f^{(l)}_{P,\rho,\sigma}(x^{-\frac{\sigma} {\rho}}y)= x^{\frac{r}{l}+s\frac{\sigma}{\rho}}(x^{-\frac{\sigma}{\rho}}y)^s f^{(l)}_{P,\rho,\sigma}(x^{-\frac{\sigma} {\rho}}y)= x^{\frac{k}{l'}}\mathfrak{f}(x^{-\frac{\sigma} {\rho}}y),
$$
where $k:=\frac{rl'}{l}+\frac{l's\sigma}{\rho}$. So, by Proposition~\ref{pr ell por automorfismos},
$$
\ell_{\rho,\sigma}(\varphi(P)) =\varphi_L\bigl(\ell_{\rho,\sigma}(P)\bigr) = x^{\frac{k}{l'}}\mathfrak{f}(x^{-\frac{\sigma}{\rho}}y+\lambda) =\sum_{i=m_{\lambda}}^N a_i x^{\frac{k}{l'}} (x^{-\frac{\sigma}{\rho}}y)^i,
$$
since $\varphi_L(x^{1/l'}) = x^{1/l'}$ and $\varphi_L(x^{-\sigma/\rho}y)=x^{-\sigma/\rho}y+\lambda$. But then, by the first equality in~\eqref{eq57},
\begin{equation}\label{starting de P}
\st_{\rho,\sigma}(\varphi(P))=\left(\frac{k}{l'}-m_{\lambda}\frac{\sigma}{\rho}, m_{\lambda}\right) =\left(\frac{r}{l}+\frac{s\sigma}{\rho}-m_{\lambda} \frac{\sigma}{\rho},m_{\lambda}\right).
\end{equation}
Note also that by~\eqref{eqnue21},
\begin{equation}
\en_{\rho,\sigma}(P) =\Bigl(\frac{r}{l}-\frac{\gamma\sigma}{\rho},N\Bigr) =\Bigl(\frac{k}{l'}-\frac{N\sigma}{\rho},N\Bigr).\label{eqnue22}
\end{equation}
We claim that
\begin{equation}
v_{1,-1}(\st_{\rho,\sigma}(\varphi(P)))<0.\label{eq50}
\end{equation}
First note that by Proposition~\ref{pr ell por automorfismos} (with $l$ replaced by $l'$),
\begin{equation}
v_{\rho,\sigma}(\varphi(P)) = v_{\rho,\sigma}(P)\qquad\text{and}\qquad v_{\rho,\sigma}(\varphi(Q)) = v_{\rho,\sigma}(Q).\label{eq49}
\end{equation}
So, statement~(9) holds. Apply Theorem~\ref{central} to $\varphi(P)$, $\varphi(Q)$, $(\rho,\sigma)$ and $l'$. Statement~(3) of that theorem gives
\begin{equation}\label{vunomenosunononulo}
 v_{1,-1}(\st_{\rho,\sigma}(\varphi(P)))\ne 0.
\end{equation}
On the other hand, statement~(2) of the same theorem gives
$$
\en_{\rho,\sigma}(F) = (1,1)\quad\text{or}\quad\en_{\rho,\sigma}(F)\sim\en_{\rho,\sigma}(P).
$$
In the first case
\begin{equation}
\Supp(F)\subseteq\{(1,1),(1+\sigma/\rho,0)\},\label{eqnue7}
\end{equation}
and so $\deg(f^{(l)}_{F,\rho,\sigma})\le 1$. Hence, by~\eqref{eq36},
\begin{equation}
1\le\#\factors(f^{(l)}_{P,\rho,\sigma})\le\deg(f^{(l)}_{F,\rho,\sigma})\le 1,\label{eqnue24}
\end{equation}
and consequently in~\eqref{eqnue7} the equality holds. Thus, $\st_{\rho,\sigma}(F) = (1+\sigma/\rho,0)$, which implies that $l'=l$ and, by~\eqref{eqnue6}, also implies that $s=0$. Therefore $k=r$, $N=\gamma$ and $f^{(l)}_{P,\rho,\sigma} =\mathfrak{f}$. So, by~\eqref{eqnue24}
$$
\mathfrak{f} = a_N (x-\lambda)^N,
$$
where $a_N$ is as in~\eqref{eqnue8}. But then $m_{\lambda} = N =\gamma$ and so, by~\eqref{starting de P} and~\eqref{eqnue22},
$$
\st_{\rho,\sigma}(\varphi(P)) =\left(\frac{r}{l}-\gamma\frac{\sigma}{\rho},N\right) =\en_{\rho,\sigma}(P),
$$
which finishes the proof of statement~(11) and yields~\eqref{eq50}, since $v_{1,-1}(\en_{\rho,\sigma}(P))<0$.

\smallskip

In the second case, by~\eqref{eqnue22}
$$
(M_0,M) :=\en_{\rho,\sigma}(F)\sim\en_{\rho,\sigma}(P) = (N_0,N),
$$
where
\begin{equation}
N_0 :=\frac{r}{l}-\frac{\gamma\sigma}{\rho} =\frac{k}{l'}-\frac{N\sigma}{\rho}.\label{eqnue11}
\end{equation}
Since, by~\eqref{numero de factores}
$$
M\ge 1\quad\text{and}\quad N=\deg(\mathfrak{f})\ge\#\factors(\mathfrak{f})\ge 1,
$$
we have $\frac{N_0}{N} =\frac{M_0}{M}$. Hence, by~\eqref{eqnue9}, \eqref{eqnue8} and~\eqref{eqnue11},
$$
\frac{k\rho}{l'(\rho+\sigma)}=\frac{k/l'}{1+\sigma/\rho} =\frac{N_0+N\frac{\sigma}{\rho}}{M_0+M\frac{\sigma}{\rho}} = \frac{N(N_0/N+\sigma/\rho)} {M(M_0/M+\sigma/\rho)} = \frac{N}{M}\le m_{\lambda},
$$
which, combined with~\eqref{starting de P}, gives
$$
v_{1,-1}(\st_{\rho,\sigma}(\varphi(P)))=\frac{k}{l'}-m_{\lambda}\frac{\sigma}{\rho} -m_{\lambda} =\left(\frac{\sigma +\rho}{\rho}\right) \left(\frac{k\rho}{l'(\rho+\sigma)}-m_{\lambda}\right)\le 0.
$$
Taking into account~\eqref{vunomenosunononulo}, this yields~\eqref{eq50}, ending the proof of the claim. Now, by statement~(d) and statement~(2) of Theorem~\ref{f[] en A_1^{(l)}}, there exist relatively prime $\bar{m},\bar{n}\in\mathds{N}$, $\lambda_P,\lambda_Q\in K^{\times}$ and a $(\rho,\sigma)$-homogeneous $R\in L^{(l)}$ such that
\begin{equation}
\frac{\bar{n}}{\bar{m}} =\frac{v_{\rho,\sigma}(Q)}{v_{\rho,\sigma}(P)},\quad\ell_{\rho,\sigma}(P) =\lambda_P R^{\bar{m}}\quad\text{and}\quad\ell_{\rho,\sigma}(Q) =\lambda_Q R^{\bar{n}}.\label{eq46}
\end{equation}
Hence, again by Proposition~\ref{pr ell por automorfismos},
\begin{equation}
\ell_{\rho,\sigma}(\varphi(P))=\lambda_P\varphi_L (R)^{\bar{m}}\quad\text{and} \quad\ell_{\rho,\sigma} (\varphi(Q)) = \lambda_Q \varphi_L(R)^{\bar{n}}.\label{eq43}
\end{equation}
Consequently, by statements~(4) and (5) of Proposition~\ref{pr v de un producto},
\begin{align}
&\st_{\rho,\sigma}(\varphi(P))=\bar{m}\st_{\rho,\sigma}(\varphi_L(R)),\quad \en_{\rho,\sigma}(\varphi(P)) =\bar{m}\en_{\rho,\sigma}(\varphi_L(R))\label{eq51}
\shortintertext{and}
&\st_{\rho,\sigma}(\varphi(Q)) =\bar{n}\st_{\rho,\sigma}(\varphi_L(R)),\quad\en_{\rho,\sigma}(\varphi(Q)) =\bar{n}\en_{\rho,\sigma}(\varphi_L(R)),\label{eq52}
\end{align}
and so
\begin{equation}
\st_{\rho,\sigma}(\varphi(P)) =\frac{\bar{m}}{\bar{n}}\st_{\rho,\sigma}(\varphi(Q)) \quad\text{and}\quad \en_{\rho,\sigma}(\varphi(P)) =\frac{\bar{m}}{\bar{n}} \en_{\rho,\sigma}(\varphi(Q)).\label{eq53}
\end{equation}
We assert that
\begin{equation}
v_{0,1}(\st_{\rho,\sigma}(\varphi_L(R)))\ge 1.\label{eq47}
\end{equation}
In fact, otherwise $v_{0,1}(\st_{\rho,\sigma}(\varphi_L(R)))=0$, and so
$$
\st_{\rho,\sigma}(\varphi_L(R)) = (h,0)\quad\text{with $h\in\frac{1}{l'}\mathds{Z}$.}
$$
Then
\begin{equation}
v_{\rho,\sigma}(\varphi_L(R))=v_{\rho,\sigma}(\st_{\rho,\sigma}(\varphi_L(R)))=\rho h<0,\label{eqnue14}
\end{equation}
since $\rho>0$ and, by~\eqref{eq50} and~\eqref{eq51},
$$
h = v_{1,-1}(\st_{\rho,\sigma}(\varphi_L(R)))<0.
$$
But, by statement~(3) of Proposition~\ref{pr v de un producto}, the second equality in~\eqref{eq46}, the first equalities in~\eqref{eq49}, and that by hypothesis, $v_{\rho,\sigma}(P)>0$, we have
$$
v_{\rho,\sigma}(\varphi_L(R))=v_{\rho,\sigma}(R)>0,
$$
which contradicts~\eqref{eqnue14}. Hence inequality~\eqref{eq47} is true. Take now
\begin{align*}
& (\rho',\sigma'):=\max\{(\rho'',\sigma'')\in\ov{\Dir}(\varphi(P)):(\rho'',\sigma'')< (\rho,\sigma)\}
\shortintertext{and}
&(\bar{\rho},\bar{\sigma}):=\max\{(\rho'',\sigma'')\in\ov{\Dir}(\varphi(Q)): (\rho'',\sigma'') < (\rho,\sigma)\}.
\end{align*}
By statement~(3) of Proposition~\ref{le basico}
\begin{equation}
\en_{\rho',\sigma'}(\varphi(P)) = \st_{\rho,\sigma}(\varphi(P))\quad\text{and} \quad \en_{\bar{\rho},\bar{\sigma}}(\varphi(Q)) = \st_{\rho,\sigma}(\varphi(Q)).\label{eq48}
\end{equation}
Combining the first equality with equality~\eqref{starting de P}, we obtain statement~(6). Moreover, by the first equalities in~\eqref{eq49} and~\eqref{eq48},
$$
v_{\rho,\sigma}(\en_{\rho',\sigma'}(\varphi(P))) = v_{\rho,\sigma}(\st_{\rho,\sigma}(\varphi(P))) = v_{\rho,\sigma}(\varphi(P)) =v_{\rho,\sigma}(P)>0,
$$
where the last inequality is true by hypothesis. Consequently
\begin{equation}
\en_{\rho',\sigma'}(\varphi(P))\ne (0,0).\label{eq40}
\end{equation}
We claim that
\begin{equation}
(\rho',\sigma') = (\bar{\rho},\bar{\sigma}).
\label{eqnue3}
\end{equation}
In order to prove this we proceed by contradiction. Assume that $(\rho',\sigma')>(\bar{\rho},\bar{\sigma})$. Then
\begin{equation}
\st_{\rho,\sigma}(\varphi(Q))=\en_{\rho',\sigma'}(\varphi(Q))=\st_{\rho',\sigma'} (\varphi(Q)),\label{eq37}
\end{equation}
where the first equality follows from statement~(3) of Proposition~\ref{le basico}, and the second one, from the fact that $(\rho',\sigma')\not \in\ov{\Dir}(\varphi(Q))$. Furthermore
\begin{equation}
\en_{\rho',\sigma'}(\varphi(P))\ne\st_{\rho',\sigma'}(\varphi(P))\label{eq39}
\end{equation}
since $(\rho',\sigma')\in\Dir(\varphi(P))$. Now, by~\eqref{eq48}, \eqref{eq53} and~\eqref{eq37},
\begin{equation}
\begin{aligned}
\en_{\rho',\sigma'}(\varphi(P))& =\st_{\rho,\sigma}(\varphi(P))\\
& =\frac{\bar{m}}{\bar{n}}\st_{\rho,\sigma}(\varphi(Q))\\
& =\frac{\bar{m}}{\bar{n}}\en_{\rho',\sigma'}(\varphi(Q))\\
& =\frac{\bar{m}}{\bar{n}}\st_{\rho',\sigma'}(\varphi(Q)).
\end{aligned}\label{eq38}
\end{equation}
We assert that
\begin{equation}
\en_{\rho',\sigma'}(\varphi(P))\nsim\st_{\rho',\sigma'}(\varphi(P)).\label{eq41}
\end{equation}
Otherwise, by the inequalities in~\eqref{eq40} and ~\eqref{eq39} there exists $\mu\in K\setminus\{1\}$ such that
\begin{equation*}
\st_{\rho',\sigma'}(\varphi(P)) =\mu\en_{\rho',\sigma'}(\varphi(P)),
\end{equation*}
which implies that
\begin{equation}
v_{\rho',\sigma'}(\varphi(P)) =\mu v_{\rho',\sigma'} (\varphi(P)).\label{eq42}
\end{equation}
On the other hand, by~\eqref{eq38}
$$
v_{\rho',\sigma'}(\varphi(Q)) =\frac{\bar{n}}{\bar{m}}v_{\rho',\sigma'}(\varphi(P)),
$$
which combined with equality~\eqref{eq42}, gives
$$
v_{\rho',\sigma'}(\varphi(P))=0=v_{\rho',\sigma'}(\varphi(Q)).
$$
But this contradicts Remark~\ref{re v de un conmutador}, since $[\varphi(Q),\varphi(P)] = 1$ and $\rho'+\sigma'>0$. Hence the condition~\eqref{eq41} is satisfied. Combining this fact with~\eqref{eq38}, we obtain
$$
\st_{\rho',\sigma'}(\varphi(Q))\nsim\st_{\rho',\sigma'}(\varphi(P)).
$$
Hence $[\varphi(Q),\varphi(P)]_{\rho',\sigma'}\ne 0$, by Corollary~\ref{extremos alineados}. Then, since $[\varphi(Q),\varphi(P)] = 1$ and $\rho'>0$, it follows from statement~(1) of Proposition~\ref{extremosnoalineados} that
\begin{equation}
\st_{\rho',\sigma'}(\varphi(P))+\st_{\rho',\sigma'}(\varphi(Q))-(1,1)= \st_{\rho',\sigma'}(1)=(0,0),\label{eq44}
\end{equation}
which implies that
\begin{equation}
v_{0,1}(\st_{\rho',\sigma'}(\varphi(Q)))\in\{0,1\},\label{eq45}
\end{equation}
because the second coordinates in~\eqref{eq44} are non-negative. But, by the first equalities in~\eqref{eq52}, \eqref{eq37}, inequality~\eqref{eq37}, statement~(4) of Proposition~\ref{pr v de un producto}, and the fact that $\bar{n}>1$ by the first equality in~\eqref{eq46},
$$
v_{0,1}(\st_{\rho',\sigma'}(\varphi(Q)))= v_{0,1}(\st_{\rho,\sigma}(\varphi(Q)))= \ov n v_{0,1}(\st_{\rho,\sigma} (\varphi_L(R)))> 1,
$$
which contradicts~\eqref{eq45}. Consequently, $(\rho',\sigma') > (\bar{\rho},\bar{\sigma})$ is impossible. Similarly one can prove that $(\rho',\sigma')<(\bar{\rho},\bar{\sigma})$ is also impossible, and so~\eqref{eqnue3} is true.

\smallskip

Using~\eqref{eq53}, \eqref{eq48}, \eqref{eqnue3}, and the fact that $\bar{m}/\bar{n} = v_{\rho,\sigma}(P)/ v_{\rho,\sigma}(Q)$, we obtain
$$
\en_{\rho',\sigma'}(\varphi(Q)) =\st_{\rho,\sigma}(\varphi(Q)))
$$
and
\begin{equation}\label{eq esquinas}
\en_{\rho',\sigma'}(\varphi(P)) =\st_{\rho,\sigma}(\varphi(P)) =\frac{v_{\rho,\sigma}(P)}{v_{\rho,\sigma}(Q)} \en_{\rho',\sigma'}(\varphi(Q)),
\end{equation}
which proves statement~(7) and combined with inequality~\eqref{eq50}, also proves statement~(2). Hence $(\rho',\sigma')\ne (1,-1)$, since otherwise
$$
v_{1,-1}(\varphi(P))<0\quad\text{and}\quad v_{1,-1}(\varphi(Q))<0,
$$
which is impossible, because it contradicts Remark~\ref{re v de un conmutador}, since $[\varphi(Q),\varphi(P)] = 1$. This concludes the proof of statement~(1). Now statement~(3) follows, since by~\eqref{eq esquinas},
$$
v_{\rho',\sigma'}(\varphi(P))\le 0\Longleftrightarrow v_{\rho',\sigma'}(\varphi(Q))\le 0,
$$
and so, again by Remark~\ref{re v de un conmutador}, the falseness of statement~(3) implies
$$
v_{\rho',\sigma'}(1) = v_{\rho',\sigma'}\bigr([\varphi(Q),\varphi(P)]\bigl)\le v_{\rho',\sigma'}(\varphi(Q)) + v_{\rho',\sigma'}(\varphi(P))-(\rho'+\sigma')<0,
$$
which is impossible. Statement~(4) also follows from~\eqref{eq esquinas}, statement~(5), from  Proposition~\ref{pr ell por automorfismos} (with $l$ replaced by $l'$) and statement~(10), from statement~(9) and the facts that
$$
[\varphi(Q),\varphi(P)]= [Q,P]\qquad\text{and}\qquad [Q,P]_{\rho,\sigma} = 0.
$$
Finally we prove statement~(8). Note that by statement~(6) and~\eqref{starting de P}
\begin{equation}
\en_{\rho',\sigma'}(\varphi(P)) =\st_{\rho,\sigma}(\varphi(P))=\left(\frac{k}{l'}-m_{\lambda} \frac{\sigma}{\rho},m_{\lambda}\right),\label{eqnue23}
\end{equation}
and so by~\eqref{eqnue22},
$$
v_{0,1}(\en_{\rho',\sigma'}(\varphi(P))) = m_{\lambda}\le N = v_{0,1}(\en_{\rho,\sigma}(P)).
$$
Furthermore, if the equality holds, then by~\eqref{eqnue8}, \eqref{eqnue22} and~\eqref{eqnue23},
$$
\en_{\rho',\sigma'}(\varphi(P))=\en_{\rho,\sigma}(P)\quad\text{and}\quad x^s f^{(l)}_{P,\rho,\sigma}(x) =\mathfrak{f}(x) = a_N (x-\lambda)^N,
$$
where $a_N$ is as in~\eqref{eqnue8}. But
$$
\deg(f^{(l)}_{P,\rho,\sigma})>0\qquad\text{and}\qquad x\nmid f^{(l)}_{P,\rho,\sigma},
$$
since $(\rho,\sigma)\in\Dir(P)$ and $f^{(l)}_{P,\rho,\sigma}(0)\ne 0$. Hence, by the last equality that $\lambda\ne 0$, $s = 0$ and $\st_{\rho,\sigma}(P) = (k/l',0)$. So, by the third equality in~\eqref{eq57},
$$
\ell_{\rho,\sigma}(P) = x^{\frac{k}{l'}}\mathfrak{f}(x^{-\frac{\sigma} {\rho}}y)= a_N x^{\frac{k}{l'}}(x^{-\frac{\sigma} {\rho}}y-\lambda)^N =\sum_{i=0}^N a_N\binom{N}{i}\lambda^{N-i}x^{\frac{k}{l'}-i\frac{\sigma}{\rho}}y^i.
$$
Consequently,
$$
\Bigl(\frac{k}{l'}-\frac{(N-1)\sigma}{\rho},N-1\Bigr)\in\Supp(P),
$$
since $\lambda\ne 0$. This finishes the proof because
$$
\en_{\rho,\sigma}(P)+\Bigl(\frac{\sigma}{\rho},-1\Bigr) =\Bigl(\frac{k}{l'}-\frac{(N-1)\sigma}{\rho}, N-1\Bigr)
$$
by equality~\eqref{eqnue22}.
\end{proof}

\begin{definition}\label{def de val} Let $l\in\mathds{N}$. For each $(r,s)\in \frac{1}{l}\mathds{Z}\times\mathds{Z}\setminus\mathds{Z}(1,1)$, we define $\dir(r,s)$ to be the unique $(\rho,\sigma)\in\mathfrak{V}$ such that $v_{\rho,\sigma}(r,s)=0$.
\end{definition}

\begin{remark}\label{re val} Note that if $P\in A_1^{(l)}\setminus\{0\}$ and $(\rho,\sigma)\in\Dir(P)$, then
$$
(\rho,\sigma) =\dir\bigl(\en_{\rho,\sigma}(P) -\st_{\rho,\sigma}(P)\bigr).
$$
\end{remark}

In the following proposition we will take $(\rho,\sigma)\in\mathfrak{V}$ with $\sigma\le 0$. Note that combining this with $\rho+\sigma>0$ we obtain $\rho>0$, and so $(\rho,\sigma)\in\mathfrak{V}^0$. Note also that $\sigma\le 0$ is equivalent to $(\rho,\sigma)\le (1,0)$. Consequently if $(\rho',\sigma')\le (\rho,\sigma)$, then $(\rho',\sigma')\le (1,0)$, and so $\rho'>0$. We will use implicitly these facts.

\begin{proposition}\label{lema general} Let $P,Q\!\in\! A_1^{(l)}$ and let $(\rho,\sigma)\!\in\!\mathfrak{V}$ with $\sigma\!\le\! 0$, such that conditions~(a), (b), (c) and~(e) of Proposition~\ref{preparatoria} are satisfied. Assume that $\frac{v_{\rho,\sigma}(Q)}{v_{\rho,\sigma}(P)}\!=\!\frac nm$ with $n,m\!>\!1$ and $\gcd(n,m)\!=\!1$. Then
$$
\frac 1m\en_{\rho,\sigma}(P)\ne\left(\bar{r}-\frac 1l,\bar{r}\right),
$$
for all $\bar{r}\ge 2$.
\end{proposition}

\begin{proof} We will assume that
\begin{equation}
\frac 1m\en_{\rho,\sigma}(P) =\left(\bar{r}-\frac 1l,\bar{r}\right),\label{pepi3}
\end{equation}
for some fixed $\bar{r}\ge 2$ and we will prove successively the following two statements:
\begin{enumerate}

\smallskip

\item $[P,Q]_{\rho,\sigma} = 0$, $v_{1,-1}(\en_{\rho,\sigma}(P)) <0$ and $v_{1,-1}(\en_{\rho,\sigma}(Q)) <0$.

\smallskip

\item $\rho|l$ and there exist
$$
\varphi\in\Aut(A_1^{(l)})\quad\text{and}\quad (\rho_1,\sigma_1)\in\mathfrak{V}\,\text{ with $(\rho_1,\sigma_1) < (\rho,\sigma)$,}
$$
such that $P_1:=\varphi(P)$, $Q_1:=\varphi(Q)$ and $(\rho_1,\sigma_1)$ satisfy conditions~(a), (b), (c) and~(e) of Proposition~\ref{preparatoria} (more precisely, these conditions are satisfied with $(P_1,Q_1)$ playing the role of $(P,Q)$ and $(\rho_1,\sigma_1)$ playing the role of $(\rho,\sigma)$). Furthermore,
$$
\qquad\qquad\frac{v_{\rho_1,\sigma_1}(P_1)}{v_{\rho_1,\sigma_1}(Q_1)}= \frac{v_{\rho,\sigma}(P)} {v_{\rho,\sigma}(Q)}\qquad\text{and}\qquad\frac 1m\en_{\rho_1,\sigma_1}(P_1)=\left(\bar{r}-\frac 1l,\bar{r}\right).
$$

\smallskip

\end{enumerate}
Statement~(2) yields an infinite, descending chain of directions $(\rho_k,\sigma_k)$, such that $\rho_k|l$. But there are only finitely many $\rho_k$'s with $\rho_k|l$. Moreover, $0<-\sigma_k<\rho_k$, so there are only finitely many $(\rho_k,\sigma_k)$ possible, which provide us with the desired contradiction.

\smallskip

We first prove statement~(1). Set $A:=\frac 1m\en_{\rho,\sigma}(P)$ and suppose $[P,Q]_{\rho,\sigma}\ne 0$. Since
$$
v_{\rho,\sigma}(P) = v_{\rho,\sigma}(mA),\quad v_{\rho,\sigma}(Q) = v_{\rho,\sigma}(nA)\quad\text{and}\quad [P,Q]=1,
$$
under this assumption we have
$$
v_{\rho,\sigma}\left(mA+nA - (1,1)\right) = v_{\rho,\sigma}(P) + v_{\rho,\sigma}(Q) - v_{\rho,\sigma}(1,1) = 0.
$$
Consequently,
\begin{equation}
v_{\rho,\sigma}(A) =\frac{\rho+\sigma}{m+n}\quad\text{and}\quad\rho \bigl(m\bar{r}l+ n\bar{r}l-m-n-l\bigr) = -\sigma\bigl(m\bar{r}l+n\bar{r}l-l\bigr),\label{pepi}
\end{equation}
where for the second equality we use assumption~\eqref{pepi3}. Let
$$
d :=\gcd(m\bar{r}l+n\bar{r}l-l,m+n-m\bar{r}l-n\bar{r}l+l) =\gcd(l,m+n).
$$
From the second equality in~\eqref{pepi}, it follows that
\begin{equation}
\rho =\frac{m\bar{r}l+n\bar{r}l-l}d\quad\text{and}\quad\sigma =\frac{m+n-m\bar{r}l-n\bar{r}l+l}{d} =\frac{m+n}{d}-\rho,\label{pepi2}
\end{equation}
and so $\rho+\sigma = (m+n)/d$. Hence, by the first equality in~\eqref{pepi},
$$
v_{\rho,\sigma}(P)=mv_{\rho,\sigma}(A)=\frac {m(\rho+\sigma)}{m+n}=\frac md\quad\text{and}\quad v_{\rho,\sigma}(Q)= nv_{\rho,\sigma}(A)=\frac nd.
$$
We will see that we are lead to
\begin{equation}
v_{1,-1}(P)\le 0\qquad\text{and}\qquad v_{1,-1}(Q)\le 0,\label{pepi1}
\end{equation}
which contradicts Remark~\ref{re v de un conmutador con 1,-1}, since $[P,Q]=1$. In order to prove~\eqref{pepi1}, it suffices to check that if $(i,j)\in\frac 1l\mathds{Z}\times\mathds{N}_0$ and $i>j$, then $v_{\rho,\sigma}(i,j)>\max\{v_{\rho,\sigma}(P), v_{\rho,\sigma}(Q)\}$. But, writing $(i,j)=(j+\frac sl,j)$ with $s\in\mathds{N}$, we obtain
\begin{align*}
v_{\rho,\sigma}(i,j)=&\rho j+\rho\frac sl+\frac{m+n}{d}j-\rho j &&\text{by~\eqref{pepi2}}\\
=&\frac{s(m\bar{r}+n\bar{r}-1)}{d}+\frac{m+n}{d}j &&\text{by~\eqref{pepi2}}\\
\ge &\frac{(m+n)\bar{r}-1}{d}\\
\ge &\frac{m+n}{d} &&\text{since $\bar{r}\ge 2$ and $m+n\ge 1$.}\\
> &\max\{m/d,n/d\}\\
= &\max\{v_{\rho,\sigma}(P), v_{\rho,\sigma}(Q)\}.
\end{align*}
This concludes the proof that $[P,Q]_{\rho,\sigma}\! =\! 0$. Now, by Corollary~\ref{extremos alineados} and the assump\-tion~\eqref{pepi3}, we have
\begin{align*}
& v_{1,-1}(\en_{\rho,\sigma}(P)) =  mv_{1,-1}\left(\bar{r}-\frac 1l,\bar{r}\right)<0
\shortintertext{and}
& v_{1,-1}(\en_{\rho,\sigma}(Q)) =\frac{n}{m}v_{1,-1}(\en_{\rho,\sigma}(P))<0,
\end{align*}
which finishes the proof of statement~(1).

\smallskip

We now prove statement~(2). By statement~(1), the hypothesis of Proposition~\ref{preparatoria} are sa\-tisfied. Let $(\rho',\sigma')$ and $\varphi$ be as in its statement. Set
$$
P_1:=\varphi(P),\quad Q_1:=\varphi(Q)\quad\text{and}\quad (\rho_1,\sigma_1):= (\rho',\sigma').
$$
By statements~(1), (3) and~(4) of Proposition~\ref{preparatoria}, we know that
$$
\frac{v_{\rho_1,\sigma_1}(P_1)}{v_{\rho_1,\sigma_1}(Q_1)}= \frac{v_{\rho,\sigma}(P)}{v_{\rho,\sigma}(Q)},
$$
and that conditions~(b), (c) and~(e) of that proposition are satisfied for $P_1$, $Q_1$ and $(\rho_1,\sigma_1)$. Moreover condition~(a) follows immediately from the fact that $\varphi$ is an algebra automorphism. It remains to prove that
\begin{equation}
\rho\mid l\quad\text{and}\quad\frac 1m\en_{\rho_1,\sigma_1}(P_1)=\left(\bar{r}- \frac 1l,\bar{r}\right).\label{pepi13}
\end{equation}
By statement~(11) of Proposition~\ref{preparatoria}, there is a $(\rho,\sigma)$-homogeneous element $F$, which is not a monomial, such that
\begin{equation}
[P,F]_{\rho,\sigma} =\ell_{\rho,\sigma}(P)\quad\text{and}\quad v_{\rho,\sigma}(F) =\rho+\sigma.\label{pepi6}
\end{equation}
By statement~(2) of Proposition~\ref{extremosnoalineados},
$$
\en_{\rho,\sigma}(F) = (1,1)\quad\text{or}\quad\en_{\rho,\sigma}(F)\sim\en_{\rho,\sigma}(P).
$$
By statements~(6) and~(11) of Proposition~\ref{preparatoria}, in  the first case we have
$$
\frac{1}{m}\en_{\rho_1,\sigma_1}(P_1) =\frac{1}{m}\st_{\rho,\sigma}(P_1) =\frac{1}{m}\en_{\rho,\sigma}(P) =\left(\bar{r}-\frac 1l,\bar{r}\right).
$$
Hence, by statement~(8) of the same proposition,
$$
\en_{\rho,\sigma}(P)+\left(\frac{\sigma}{\rho},-1\right)\in\Supp(P)\subseteq \frac{1}{l}\mathds{Z}\times\mathds{N}_0.
$$
Since $\en_{\rho,\sigma}(P)\in\frac{1}{l}\mathds{Z}\times\mathds{N}_0$, this implies that
$$
\left(\frac{\sigma}{\rho},-1\right)\in\frac{1}{l}\mathds{Z}\times\mathds{Z},
$$
and so $\rho|\sigma l$. But then $\rho|l$, since $\gcd(\rho,\sigma)=1$. This finishes the proof of condition~\eqref{pepi13} when $\en_{\rho,\sigma}(F) = (1,1)$.

\smallskip

Assume now that $\en_{\rho,\sigma}(F)\sim\en_{\rho,\sigma}(P)$. Then, since $\left(\bar{r}-\frac 1l,\bar{r}\right)$ is indivisible in $\frac 1l\mathds{Z}\times\mathds{N}_0$, we have
\begin{equation}
\en_{\rho,\sigma}(F) =\mu\frac 1m\en_{\rho,\sigma}(P) =\mu\left(\bar{r}-\frac 1l,\bar{r}\right)\quad\text{with $\mu\in\mathds{N}$.}\label{pepi5}
\end{equation}
We claim that
$$
\mu=1,\quad\bar{r}=2\quad\text{and}\quad\rho = l.
$$
By statement~(2) of Theorem~\ref{f[] en A_1^{(l)}} there exist $\lambda_P,\lambda_Q\in K^{\times}$ and a $(\rho,\sigma)$-ho\-mo\-ge\-neous polynomial $R\!\in\! L^{(l)}$,  such that
\begin{equation}
\ell_{\rho,\sigma}(P) =\lambda_P R^m\quad\text{and}\quad\ell_{\rho,\sigma}(Q) =\lambda_Q R^n.\label{pepi14}
\end{equation}
Note that $R$ is not a monomial, since $(\rho,\sigma)\in\Dir(P)$. By statement~(5) of Proposition~\ref{pr v de un producto} and the assumption~\eqref{pepi3}, we have
\begin{equation}
\en_{\rho,\sigma}(R) =\left(\bar{r}-\frac 1l,\bar{r}\right).\label{pepi8}
\end{equation}
Hence, by statement~(2) of Proposition~\ref{le basico1},
\begin{equation}
v_{1,-1}(\st_{\rho,\sigma}(R))>v_{1,-1}(\en_{\rho,\sigma}(R))=-\frac 1l.\label{pepi4}
\end{equation}
Since, by statement~(3) of Theorem~\ref{central} and statement~(4) of Proposition~\ref{pr v de un producto},
$$
v_{1,-1}(\st_{\rho,\sigma}(R)) =\frac 1m v_{1,-1}(\st_{\rho,\sigma}(P))\ne 0,
$$
from inequality~\eqref{pepi4} it follows that
\begin{equation}
v_{1,-1}(\st_{\rho,\sigma}(R))> 0.\label{pepi10}
\end{equation}
Moreover, by equality~\eqref{pepi5} and the second equality in~\eqref{pepi6}
$$
v_{\rho,\sigma}\left(\mu\left(\bar{r}-\frac 1l,\bar{r}\right)-(1,1)\right)=0,
$$
which implies that
\begin{equation}
\rho(\mu\bar{r} l -\mu -l) = -\sigma(\mu\bar{r} l - l).\label{pepi7}
\end{equation}
Let
\begin{equation}
d :=\gcd(\mu\bar{r} l -\mu -l,\mu\bar{r} l - l) =\gcd(\mu,l).\label{pepi12}
\end{equation}
By equality~\eqref{pepi7}
\begin{equation*}
\rho=\frac{\mu\bar{r} l -l}d\quad\text{and}\quad\sigma=\frac{\mu -\mu\bar{r} l + l}{d} =\frac{\mu}{d}-\rho.
\end{equation*}
Hence
\begin{equation}
\rho=\frac{\mu\bar{r} l -l}d\quad\text{and}\quad\rho+\sigma=\frac\mu d.\label{pepi9}
\end{equation}
So,
\begin{equation}
v_{\rho,\sigma}\left(j+\frac sl,j\right)=\frac{\mu j}{d}+\frac {(\mu\bar{r}-1)s}d\ge\frac{\mu\bar{r}-1}{d}\quad\text{for all $j\in\mathds{N}_0$ and $s\in\mathds{N}$.}\label{pepi11}
\end{equation}
If $\bar{r}>2$ or $\mu>1$, this yields
$$
v_{\rho,\sigma}\left(j+\frac sl,j\right)>\frac 1d = v_{\rho,\sigma}(R),
$$
where the last equality follows from~\eqref{pepi8} and~\eqref{pepi9}. Hence, no $(i,j)\in\frac 1l\mathds{Z}\times\mathds{N}$ with $i>j$ lies in the support of $R$, and so $v_{1,-1}(\st_{\rho,\sigma}(R))\le 0$, which contradicts inequality~\eqref{pepi10}. Thus, necessarily $\bar{r}=2$ and $\mu = 1$, which, by equality~\eqref{pepi12} and the first equality in~\eqref{pepi9}, implies $d=1$ and $\rho=l$. This finishes the proof of the claim. Combining this with~\eqref{pepi5} and~\eqref{pepi8}, we obtain
\begin{equation}
\en_{\rho,\sigma}(F) =\en_{\rho,\sigma}(R) =\left(2-\frac{1}{\rho},2\right).\label{pepi16}
\end{equation}
Now, by~\eqref{eq57}, there exists $\gamma\in\{1,2\}$, such that
$$
\st_{\rho,\sigma}(R) =\left(2-\frac{1}{\rho} +\frac {\gamma\sigma}{\rho},2-\gamma\right)=\left(1 +\frac {(\gamma-1)\sigma}{\rho},2-\gamma\right),
$$
where the last equality follows from the fact that, by the second equality in~\eqref{pepi9}, we have $\rho+\sigma\! =\! 1$. But the case $\gamma\!=\!1$ is impossible, since it contradicts inequality~\eqref{pepi10}. Thus, ne\-ce\-ssa\-rily
\begin{equation}
\st_{\rho,\sigma}(R) =\left(1 +\frac{\sigma}{\rho},0\right) =\left(\frac{1}{\rho},0\right).\label{pepi17}
\end{equation}
Note that from equalities~\eqref{pepi16} and~\eqref{pepi17} it follows that $\deg\bigl(f_{R,\rho,\sigma}^{(\rho)}\bigr)=2$. Hence, by Remark~\ref{f de un producto} and the first equality in~\eqref{pepi14},
$$
f_{P,\rho,\sigma}^{(\rho)} = a(x-\lambda)^{2m}\qquad\text{or}\qquad f_{P,\rho,\sigma}^{(\rho)} = a(x-\lambda)^m(x-\lambda')^m,
$$
where $a,\lambda,\lambda'\in K^{\times}$ and $\lambda'\ne\lambda$. Let $\mathfrak{f}$ be as in Proposition~\ref{preparatoria}. By statement~(4) of Proposition~\ref{pr v de un producto}, the first equality in~\eqref{pepi14} and equality~\eqref{pepi17},
$$
\st_{\rho,\sigma}(P) =\left(\frac{m}{\rho},0\right)\qquad\text{and}\qquad\mathfrak{f} = f_{P,\rho,\sigma}^{(\rho)}.
$$
Let $m_{\lambda}$ be the multiplicity of $x-\lambda$ in $\mathfrak{f}$. By statement~(6) of Proposition~\ref{preparatoria},
$$
\en_{\rho_1,\sigma_1}(P_1) =\st_{\rho,\sigma}(P_1) =\Bigl(\frac{m}{\rho}-m_{\lambda}\frac{\sigma}{\rho}, m_{\lambda}\Bigr) =\begin{cases} (m,m) &\text{if $m_{\lambda}=m$,}\\ (2m-m/\rho,2m) &\text{if $m_{\lambda}=2m$,}\end{cases}
$$
where for the computation we used that $\rho+\sigma=1$. In order to finish the proof it is enough to notice that the case $m_{\lambda}$ is impossible, since it contradicts statement~(3) of Theorem~\ref{central}.
\end{proof}

\section{Standard minimal pairs}\label{computinglowerbounds}

\setcounter{equation}{0}

The aim of this paper is to determine a lower bound for
$$
B:= \begin{cases}  \infty & \text{if DC is true,}\\ \min\bigl(\gcd(v_{1,1}(P),v_{1,1}(Q))\bigr) & \text{where $(P,Q)$ is a counterexample to DC, if DC is false.}
\end{cases}
$$
A minimal pair is a counterexample $(P,Q)$ such that $B=\gcd(v_{1,1}(P),v_{1,1}(Q))$. In this section we will prove that if $B<\infty$, then there exists a standard minimal pair as in Definition~\ref{Smp}.

\smallskip

Although it is known that a counterexample $(P,Q)$ to DC can be brought into a subrectangular shape, we have to prove that this can be done without changing $\gcd(\deg(P),\deg(Q))$, and in such a way that the pair obtained satisfies definition~\ref{Smp}.

\smallskip

First we will prove some properties of minimal pairs. It is easy to verify that $v_{1,1}(P)>1$, $v_{1,1}(Q)>1$ and none of them are in $K[X]$ or $K[Y]$.

\begin{remark}\label{a remark} Let $(\rho,\sigma)\in \mathfrak{V}^0$. If $v_{\rho,\sigma}(P),v_{\rho,\sigma}(Q)>0$ and $[P,Q]_{\rho,\sigma}=0$, then the hypothesis of statement~(2) of Theorem~\ref{f[] en A_1^{(l)}} are satisfied. Hence there exist $\lambda_P,\lambda_Q\in K^{\times}$, $m,n\in \mathds{N}$ and a $(\rho,\sigma)$-homogeneous polynomial $R\in L$, such that
\begin{equation}
\ell_{\rho,\sigma}(P)=\lambda_P R^m\qquad\text{and}\qquad \ell_{\rho,\sigma}(Q) = \lambda_Q R^n.\label{pepe0}
\end{equation}
Consequently, by statement~(2) of Proposition~\ref{pr v de un producto},
\begin{align*}
(\rho,\sigma)\in \Dir(P) & \Leftrightarrow \text{$\ell_{\rho,\sigma}(P)$ is not a monomial}\\
& \Leftrightarrow \text{$R$ is not a monomial}\\
& \Leftrightarrow \text{$\ell_{\rho,\sigma}(Q)$ is not a monomial}\\
& \Leftrightarrow (\rho,\sigma)\in \Dir(Q).
\end{align*}
Moreover by statements~(2)--(5) of Proposition~\ref{pr v de un producto},
\begin{equation}
\st_{\rho,\sigma}(P)=\frac{m}{n}\st_{\rho,\sigma}(Q),\quad\en_{\rho,\sigma}(P)= \frac{m}{n}\en_{\rho,\sigma}(Q)\quad\text{and}\quad\frac{m}{n} :=\frac{v_{\rho,\sigma}(P)}{v_{\rho,\sigma}(Q)}.\label{pepe}
\end{equation}
\end{remark}

\begin{lemma}\label{PQ similares en I} Let $\ov I:=\{(\rho,\sigma)\in\mathfrak{V} : (1,0)\le (\rho,\sigma)<(0,1)\}$. If $(P,Q)$ is a counterexample to DC, then

\begin{enumerate}

\smallskip

\item $v_{\rho,\sigma}(P),v_{\rho,\sigma}(Q)>0$ for each $(\rho,\sigma)\in\ov I$.

\smallskip

\item $[P,Q]_{\rho,\sigma}=0$ for each $(\rho,\sigma)\in\ov I$.

\smallskip

\item $\st_{\rho,\sigma}(P)= \frac{v_{\rho,\sigma}(P)}{v_{\rho,\sigma}(Q)} \st_{\rho,\sigma}(Q)$ and $\en_{\rho,\sigma}(P)=\frac{v_{\rho,\sigma}(P)} {v_{\rho,\sigma}(Q)} \en_{\rho,\sigma}(Q)$ for each $(\rho,\sigma)\in\ov I$.

\smallskip

\item For each $(\rho,\sigma)\in\ov I$, there exist $\lambda_P,\lambda_Q\in K^{\times}$, $m,n\in \mathds{N}$ and a $(\rho,\sigma)$-homogeneous polynomial $R\in L$, such that
\begin{equation*}
\ell_{\rho,\sigma}(P)=\lambda_P R^m\qquad\text{and}\qquad \ell_{\rho,\sigma}(Q) = \lambda_Q R^n.
\end{equation*}

\smallskip

\item $\Dir(P)\cap\ov I=\Dir(Q)\cap\ov I$.

\smallskip

\item $\frac{v_{1,1}(P)}{v_{1,1}(Q)} = \frac{v_{\rho,\sigma}(P)} {v_{\rho,\sigma}(Q)}$ for each $(\rho,\sigma)\in \Dir(P)\cap\ov I$.

\end{enumerate}
\end{lemma}

\begin{proof} Let $(\rho,\sigma)\in\ov I$. Since $P,Q\in A_1\setminus K[X]\cup K[Y]$ and $\rho,\sigma\ge 0$, we have $v_{\rho,\sigma}(P),v_{\rho,\sigma}(Q)>0$. This proves statement~(1). We now prove statement~(2). We must show that
\begin{equation}\label{proporcionales}
v_{\rho,\sigma}(P)+v_{\rho,\sigma}(Q)=\rho+\sigma
\end{equation}
is not possible. Note that $(\rho,\sigma)\in\ov I$ means that $\rho,\sigma\ge 0$. If $\rho=\sigma=1$, then the only possibility would be $v_{\rho,\sigma}(P)=v_{\rho,\sigma}(Q)=1$, but this is impossible. Assume $\sigma>\rho$. Since $P\notin K[x]$, there exists $(i,j)\in\Supp(P)$ with $j>0$. Then $v_{\rho,\sigma}(P)\ge\rho i+\sigma j\ge\sigma j\ge\sigma$. Hence, if equality~\eqref{proporcionales} holds, then $v_{\rho,\sigma}(Q)\le \rho < \sigma$, and so $Q\in K[x]$, which is impossible. The same argument applies to the case $\rho>\sigma$, and so statement~(2) is true. Therefore, the conditions required in the previous remark are satisfied. From this we obtain immediately statements~(3), (4) and~(5). Finally, by~\eqref{pepe}, statement~(3) of Proposition~\ref{le basico} and the fact that $\Dir(P)\cap\ov I$ is finite,
$$
\frac{v_{1,1}(P)}{v_{1,1}(Q)}=\frac{v_{\rho,\sigma}(P)}{v_{\rho,\sigma}(Q)},
$$
which proves statement~(6).
\end{proof}

\begin{proposition}\label{no se dividen} If $(P,Q)$ is a minimal pair, then neither $v_{1,1}(P)$ divides $v_{1,1}(Q)$ nor $v_{1,1}(Q)$ divides $v_{1,1}(P)$.
\end{proposition}

\begin{proof} Assume for example that $v_{1,1}(P)$ divides $v_{1,1}(Q)$. By statement~(5) of Lemma~\ref{PQ similares en I}, there exist $\lambda_P,\lambda_Q\in K^{\times}$, $m,n\in \mathds{N}$ and a $(1,1)$-homogeneous polynomial $R\in L$, such that
$$
\ell_{1,1}(P) = \lambda_P R^m\qquad\text{and}\qquad\ell_{1,1}(Q) = \lambda_Q R^n.
$$
Since $v_{1,1}(P)\mid v_{1,1}(Q)$, we have $n = k m$ for some $k\in\mathds{N}$. Hence, by statement~(2) of Pro\-position~\ref{pr v de un producto},
$$
\ell_{1,1}(Q) = \ell_{1,1}\biggl(\frac{\lambda_Q}{\lambda_P^k} P^k\biggr),
$$
and so $Q_1:= Q-\frac{\lambda_Q}{\lambda_P^k} P^k$ satisfy $v_{1,1}(Q_1)< v_{1,1}(Q)$. Moreover it is clear that $[P,Q_1] = [P,Q]=1$. Now we can construct successively $Q_2,Q_3,\dots$, such that $[P,Q_k]=1$ and $v_{1,1}(Q_k)< v_{1,1}(Q_{k-1})$, until that $v_{1,1}(P)$ does not divide $v_{1,1}(Q_k)$. Then $B_1\!:=\!\gcd(v_{1,1}(Q_k),v_{1,1}(P))\!<\!B$, which con\-tradicts the minimality of $B$. Similarly $v_{1,1}(Q)$ divides $v_{1,1}(P)$ is impossible.
\end{proof}

\begin{lemma} (Compare with~\cite{J}*{Corollary 2.4})\label{rho igual a uno} Let $P\in A_1$ and let $(\rho,\sigma)\in\Dir(P)$. Assume there exists $Q\in A_1$ such that with $[Q,P]=1$. The following assertions hold:

\begin{enumerate}

\smallskip

\item If $\sigma>\rho>0$, then $\rho=1$ and $\{\st_{\rho,\sigma}(P)\}= \Supp(\ell_{1,1}(P))= \{(v_{1,1}(P),0)\}$.

\smallskip

\item If $\rho>\sigma>0$, then $\sigma=1$ and $\{\en_{\rho,\sigma}(P)\}=\Supp(\ell_{1,1}(P))= \{(0,v_{1,1}(P))\}$.

\smallskip

\item If $\sigma>\rho>0$ or $\rho>\sigma>0$, then $f_{P,\rho,\sigma}$ is the power of a linear factor up to a constant.

\end{enumerate}

\end{lemma}

\begin{proof} Since $v_{\rho,\sigma}(P)>0$, by Theorem~\ref{central} there exists a $(\rho,\sigma)$-homogeneous $F\in A_1$ such that
$$
[P,F]_{\rho,\sigma}=\ell_{\rho,\sigma}(P)\quad\text{and}\quad v_{\rho,\sigma}(F)=\rho+\sigma.
$$
Hence
$$
\Supp(F)\subseteq\left\{(1,1),\left(\frac{\rho+\sigma}{\rho},0\right)\right\},
$$
since $(i,j)\in\Supp(F)$ with $j\ge 2$ leads to the contradiction
$$
v_{\rho,\sigma}(F)\ge v_{\rho,\sigma}(i,j)=i\rho+j\sigma\ge 2\sigma>\rho+\sigma.
$$
Consequently there exist $c_0,c_1\in K$, not both zero, such that
$$
F = c_0 X^{\frac{\rho+\sigma}{\rho}} + c_1 XY,
$$
and so
$$
\deg(f^{(1)}_{F,\rho,\sigma}) = \begin{cases} 1 & \text{if $c_0\ne 0$ and $c_1\ne 0$,}\\ 0 &\text{otherwise.}\end{cases}
$$
Thus, by statement~(1) of Corollary~\ref{pavadass}, the polynomial $f^{(1)}_{P,\rho,\sigma}$ has at most one irreducible factor, which is necessarily linear. On the other hand, since $(\rho,\sigma)\in\Dir(P)$, the polynomial $f^{(1)}_{P,\rho,\sigma}$ has degree greater than zero, and so, $f_{P,\rho,\sigma}$ has exactly one irreducible factor and this factor is linear. Therefore, again by statement~(1) of Corollary~\ref{pavadass}, the polynomial $f^{(1)}_{F,\rho,\sigma}$ has degree greater that zero, and consequently $c_0\ne 0$ and $c_1\ne 0$. Hence, $\rho=1$ (because $F\in A_1$) and
$$
\st_{\rho,\sigma}(F)=\left(\frac{\rho+\sigma}{\rho},0\right)=(\sigma+1,0).
$$
By statement~(1) of Theorem~\ref{central} there exists $\lambda\in K$ such that
$$
\st_{\rho,\sigma}(P)=\lambda(1+\sigma,0).
$$
Let $(i,j)\in\Supp(P)$. Since
$$
v_{1,1}(i,j)=i+j\le i+\sigma j=v_{\rho,\sigma}(i,j)\le v_{\rho,\sigma}(\lambda +\lambda \sigma,0) =\lambda+\lambda\sigma=v_{1,1}(\lambda+\lambda\sigma,0),
$$
with equality only if $i=\lambda+\lambda\sigma$ and $j=0$, the last assertion in statement~(1) is true. The same reasoning applies to the case $\rho>\sigma>0$.
\end{proof}

Let $P\in A_1\setminus\{0\}$ and let $a,b\in\mathds{N}$. We say that $P$ is {\em subrectangular with ver\-tex $(a,b)$} if
$$
(a,b)\in\Supp(P)\subseteq \{(i,j) : 0\le i\le a\text{ and } 0\le j\le b\}.
$$

\begin{lemma}\label{subrectangularconVal} An element $P\in A_1$ is subrectangular, if and only if $P\notin K[X]\cup K[Y]$ and
$$
\Dir(P)\cap I = \emptyset, \quad \text{where}\quad I:=\{(\rho,\sigma)\in\mathfrak{V} : (1,0)< (\rho,\sigma)< (0,1)\}.
$$
\end{lemma}

\begin{proof} Assume that $P\notin K[X]\cup K[Y]$ and let
$$
(\rho_1,\sigma_1)\ge (1,1)\ge (\rho_2,\sigma_2)
$$
be such that $(\rho_1,\sigma_1) > (\rho_2,\sigma_2)$ are consecutive elements of $\ov{\Dir}(P)$. If $\Dir(P)\cap I = \emptyset$, then
$$
(\rho_1,\sigma_1)\ge (1,0) > (1,1) > (0,1)\ge (\rho_2,\sigma_2).
$$
Consequently $\Supp(\ell_{1,1}(P))$ has only one element $(a,b)$ and by statement~(3) of Proposition~\ref{le basico}
$$
\en_{1,0}(P)=(a,b)=\st_{0,1}(P).
$$
Thus $v_{1,0}(P)=a$ and $v_{0,1}(P)=b$, and so $P$ is subrectangular. Assume now that $P$ is subrectangular with corner $(a,b)$. Since $a,b\in \mathds{N}$ it is clear that $P\notin K[X]\cup K[Y]$. Moreover $v_{1,0}(P) = a$ and $v_{0,1}(P) = b$. By Remark~\ref{sucesor y predecesor}, in order to prove that $\Dir(P)\cap I = \emptyset$, it suffices to show that if $(\rho,\sigma):=\Succ_P(1,0)$ exists, then $(\rho,\sigma)\ge (0,1)$. Since $(a,b)\in\Supp(\ell_{1,0}(P))$, we know that $\en_{1,0}(P) = (a,b+\gamma)$ for some $\gamma\in\mathds{Z}$. Hence,
$$
b+\gamma-a = v_{-1,1}(\en_{1,0}(P))\ge v_{-1,1}(a,b)= b-a,
$$
where the inequality follows using again that $(a,b)\in\Supp(\ell_{1,0}(P))$. This implies that $\gamma\ge 0$. On the other hand $b\ge v_{0,1}(\en_{1,0}(P)) = b+\gamma$, because $v_{0,1}(P) = b$. Consequently $\gamma = 0$, and so $\en_{1,0}(P)=(a,b)$. Thus $\st_{\rho,\sigma}(P)=(a,b)$, by statement~(1) of Lemma~\ref{predysucc}. If $(\rho,\sigma) < (0,1)$, then, by statement~(1) of Proposition~\ref{le basico1},
$$
v_{0,1}(P)\ge v_{0,1}(\en_{\rho,\sigma}(P)) > v_{0,1}(\st_{\rho,\sigma}(P)) = v_{0,1}(a,b) = b,
$$
which contradicts that $v_{0,1}(P) = b$.
\end{proof}

\begin{remark}\label{propiedad de la esquina} The proof of the previous lemma shows that for a subrectangular $P$ with corner $(a,b)$, it holds that
$$
\en_{1,0}(P)=(a,b)=\st_{0,1}(P).
$$
\end{remark}

\begin{lemma}\label{le subrectangular} Let $(P,Q)$ be a minimal pair. If $\ell_{1,1}(P) = \mu x^ay^b$ for some $a,b>0$, then $P$ is subrectangular with vertex $(a,b)$.
\end{lemma}

\begin{proof} Let $(i,j)\in\Supp(P)$. We only will prove that $j\le b$, since the fact that $i\le a$ follows by applying the automorphism $\varphi$ of $A_1$ defined by $\varphi(X):=Y$ and $\varphi(Y):=-X$. By the very definition of $\Dirsup$ below Remark~\ref{valuacion depende de extremos}, if $\Dirsup_P(1,1)=\emptyset$, then
$$
j-i \le v_{-1,1}(P) = b-a,
$$
and since $i+j\le v_{1,1}(P) = a+b$, we have $j\le b$, as desired. Hence we can assume that $\Dirsup_P(1,1)\ne \emptyset$. By Remark~\ref{sucesor y predecesor}, under this assumption $(\rho,\sigma):=\Succ_P(1,1)$ exists and
\begin{equation}
(\rho,\sigma) = \min\{(\rho',\sigma'):(\rho',\sigma')\in\Dir(P)\text{ and } (\rho',\sigma')>(1,1)\}.\label{e1}
\end{equation}
Then, by statement~(3) of Proposition~\ref{le basico},
\begin{equation}
\{\st_{\rho,\sigma}(P)\} = \Supp(\ell_{1,1}(P)) = \{(a,b)\}.\label{e2}
\end{equation}
We claim that $(\rho,\sigma)\ge (0,1)$. Suppose by contradiction that $(\rho,\sigma)<(0,1)$ or, equivalently, that $\rho>0$. Since $(1,1)<(\rho,\sigma)$, we also have $\sigma>\rho$. Hence, by statement~(1) of Lemma~\ref{rho igual a uno} we obtain $\st_{\rho,\sigma}(P)=(v_{1,1}(P),0)$. But this is impossible by~\eqref{e2} and the fact that $b>0$. Consequently $(\rho,\sigma)\ge (0,1)$. Now, since $(0,1)>(1,1)$, from equality~\eqref{e1} and statement~(3) of Proposition~\ref{le basico}, it follows that, if $(\rho,\sigma) > (0,1)$, then
$$
\{\st_{\rho,\sigma}(P)\} = \Supp(\ell_{0,1}(P)) = \{(a,b)\},
$$
while if $(\rho,\sigma) = (0,1)$, then $\st_{0,1}(P) = (a,b)$, by~\eqref{e2}. Hence,
$$
j = v_{0,1}(i,j) \le v_{0,1}(\st_{0,1}(P))=b,
$$
as desired.
\end{proof}

\begin{lemma}\label{esquinas} Let $P_1,P_2\in A_1\setminus\{0\}$ and let $n_i:=v_{1,1}(P_i)$. There exists $\varphi\in\Aut(A_1)$ such that $v_{1,1}\bigl(\varphi(P_i)\bigr) = n_i$ and $(n_i,0), (0,n_i)\in \Supp\bigl(\ell_{1,1}(\varphi(P_i))\bigr)$.
\end{lemma}

\begin{proof} Let $z:=x^{-1}y$ and let $f_1$ and $f_2$ be univariate polynomials with $\deg(f_i)\le n_i$ such that
$$
\ell_{1,1}(P_i) = x^{n_i}f_i(z).
$$
Take $\lambda\in K$ such that $f_1(\lambda)f_2(\lambda)\ne 0$ and consider the automorphism $\psi_{\lambda}\colon A_1\to A_1$ defined by $\psi^{\lambda}(X):=X$ and $\psi^{\lambda}(Y):=Y+\lambda X$. By Remark~\ref{re ell por automorfismos},
$$
\ell_{1,1}\bigl(\psi^{\lambda}(P_i)\bigr)=\psi^{\lambda}_L(\ell_{1,1}(P_i))= x^{n_i}f_i(z+\lambda).
$$
Since $\lambda$ is not a root of $f_i$, the point $(n_i,0)$ is in $\Supp\bigl(\ell_{1,1}(\psi^{\lambda}(P_i)) \bigr)$. Now let $t:=y^{-1}x$ and let $g_1,g_2\in K[t]$ be polynomials such that
$$
\ell_{1,1}(\psi^{\lambda}(P_i))=y^{n_i}g_i(t).
$$
Note that $\deg(g_i) = n_i$, since $(n_i,0)\in \Supp\bigl(\ell_{1,1}(\psi^{\lambda}(P_i))\bigr)$. Take $\mu\in K$ such that $g_1(\mu)g_2(\mu)\ne 0$ and consider the automorphism $\phi^{\mu}\colon A_1\to A_1$
defined by $\phi^{\mu}(X):=X+\mu Y$ and $\phi^{\mu}(Y):=Y$. By Remark~\ref{re ell por automorfismos},
$$
\ell_{1,1}(\phi^{\mu}\circ\psi^{\lambda}(P_i)) = \phi^{\mu}_L(\ell_{1,1} (\psi^{\lambda}(P_i))) = y^{n_i}g_i(t+\mu).
$$
Let $\varphi:=\phi_{\mu}\circ\psi_{\lambda}$. Since $g_i(\mu)\ne 0$ and $\deg (g_i)=n_i$ we have $(n_i,0),(0,n_i)\in \Supp(\varphi(P_i))$.
\end{proof}

\begin{proposition} (Compare with~\cite{J}*{Corollary 2.5})\label{minimalessubrectangular} For each minimal pair $(P,Q)$ there is a minimal pair $(\tilde{P},\tilde{Q})$, with $\tilde{P}$ and $\tilde{Q}$ subrectangular, such that
$$
v_{1,1}(\tilde{Q}) = v_{1,1}(Q)\quad\text{and}\quad v_{1,1}(\tilde{P}) = v_{1,1}(P).
$$
\end{proposition}

\begin{proof} Let $m:=v_{1,1}(P)$ and $n:=v_{1,1}(Q)$. By Lemma~\ref{esquinas} we can assume without loss of generality that $(0,m),(m,0)\in\Supp(P)$, which implies that $(m,0) =\st_{1,1}(P)$ and $(0,m) = \en_{1,1}(P)$. By Theorem~\ref{central} there exists a $(1,1)$-homogeneous element $F\in A_1$ such that $[P,F]_{1,1} = \ell_{1,1}(P)$ and $v_{1,1}(F)=2$. Consequently,
$$
\Supp(F)\subseteq\{(0,2),(1,1),(2,0)\}.
$$
From this it follows that $\deg\bigl(f^{(1)}_{F,1,1}\bigr)=1$ or $2$. So, by statement~(1) of Corollary~\ref{pavadass} the polynomial $f^{(1)}_{P,1,1}$ is divisible by at most two irreducible factors. Assume that $f^{(1)}_{P,1,1}$ has two different irreducible factors. Then by~\eqref{eq57} there exist $m_1,m_2\in\mathds{N}$ and $\mu,\lambda_1,\lambda_2\in K^{\times}$ with $\lambda_1\ne\lambda_2$
such that
$$
\ell_{1,1}(P) =\mu x^{m} (z-\lambda_1)^{m_1}(z-\lambda_2)^{m_2}\quad\text{where $z:=x^{-1}y$.}
$$
Furthermore $m_1+m_2=m$, because $(0,m)\in\Supp(\ell_{1,1}(P))$, and so,
$$
\ell_{1,1}(P) = \mu (y-\lambda_1x)^{m_1}(y-\lambda_2x)^{m_2}.
$$
Let $\psi',\psi''\in\Aut(A_1)$ defined by
$$
\psi'(X):=X,\quad\psi'(Y):=Y+\lambda_1 X,\quad\psi''(X):=X+\frac{1}{\lambda_2 -\lambda_1} Y \quad\text{and}\quad\psi''(Y):=Y.
$$
Consider the map $\psi:=\psi''\circ\psi'$ and take $\tilde{P}:=\psi(P)$ and $\tilde{Q}:=\psi(Q)$. A direct computation shows that $\psi_L:=\psi''_L\circ\psi'_L$. So, by Remark~\ref{re ell por automorfismos},
$$
v_{1,1}(\tilde{P})=v_{1,1}(P),\quad v_{1,1}(\tilde{Q})=v_{1,1}(Q)\quad\text{and}\quad \ell_{1,1}(\tilde{P}) = \psi_L\bigl(\ell_{1,1}(P)\bigr) = \mu (\lambda_1 -\lambda_2)^{m_2}y^{m_1}x^{m_2}.
$$
Consequently we can apply Lemma~\ref{le subrectangular} to conclude that $\tilde{P}$ is subrectangular. But, by Lemma~\ref{subrectangularconVal}, being subrectangular means precisely that $\Dir(P)\cap I = \emptyset$, where  $$
I:= \{(\rho,\sigma)\in\mathfrak{V}:(1,0)< (\rho,\sigma)< (0,1)\}.
$$
Hence, $\Dir(Q)\cap I=\emptyset$ by statement~(5) of Lemma~\ref{PQ similares en I}, and so $\tilde{Q}$ is also subrectangular, again by Lemma~\ref{subrectangularconVal}. Thus, in order to conclude the proof, it suffices to show that
\begin{equation}\label{un solo factor}
\#\factors(f^{(1)}_{P,1,1}) = 1
\end{equation}
is impossible. We will prove this showing that if equality~\eqref{un solo factor} is true, then there exists a counterexample $(P_1,Q_1)$ to DC with
\begin{equation}\label{faltaba}
\frac{v_{1,1}(P_1)}{v_{1,1}(Q_1)} = \frac{v_{1,1}(P)}{v_{1,1}(Q)} \qquad\text{and}\qquad v_{1,1}(P_1) < v_{1,1}(P),
\end{equation}
which leads to an immediate contradiction, since it implies that $\gcd(v_{1,1}(Q_1), v_{1,1}(P_1))<B$. A\-ssume that $f^{(1)}_{P,1,1}$ has only one irreducible factor. Then, by equalities~\eqref{eq57} we know that there exist $\mu,\lambda\in K^{\times}$ such that
$$
\ell_{1,1}(P) = \mu x^m (z\lambda+1)^k \quad\text{where $z:=x^{-1}y$.}
$$
Furthermore $k=m$, because $(0,m)\in\Supp(\ell_{1,1}(P))$, and so,
\begin{equation}
\ell_{1,1}(P) = \mu(x +\lambda y)^m.\label{ep1}
\end{equation}
Consider the automorphism $\varphi'$ of $A_1$ defined by
$$
\varphi'(X):=X-\lambda Y\quad\text{and}\quad\varphi'(Y):=Y,
$$
and set $\ddot{P}:=\varphi'(P)$ and $\ddot{Q}:=\varphi'(Q)$. By Remark~\ref{re ell por automorfismos},
\begin{equation}\label{cota inferior}
\ell_{1,1}(\ddot{P}) =\varphi'_L\bigl(\ell_{1,1}(P)\bigr) \quad\text{and}\quad v_{1,1}(\ddot{P})= m.
\end{equation}
On the other hand, a direct computation using~\eqref{ep1} shows that
$$
\ell_{1,1}\bigl(\ddot{P}\bigr)=\varphi'_L\bigl(\ell_{1,1}(P)\bigr) =\mu x^m ,
$$
which implies that
\begin{equation}
\Supp\bigl(\ell_{1,1}(\ddot{P})\bigr)=\{(m,0)\}.\label{starrr}
\end{equation}
Note that $\Dirsup_{\ddot{P}}(1,1)\ne \emptyset$. In fact, otherwise by the definition of $\Dirsup$ below Remark~\ref{valuacion depende de extremos},
$$
j-i = v_{-1,1}(i,j)\le v_{-1,1}(\ddot{P}) \le v_{-1,1}(\en_{1,1}(\ddot{P})) = v_{-1,1}(m,0) = -m\quad\text{for all $(i,j)\in \Supp(\ddot{P})$},
$$
but, since $i+j = v_{1,1}(i,j) \le v_{1,1}(\ddot{P}) = m$, this implies $j\le 0$ for all $(i,j)\in \Supp(\ddot{P})$, which is false. Let $(\rho,\sigma):=\Succ_{\ddot{P}}(1,1)$. By the fact that $(1,1)\notin \Dir(P)$ and Remark~\ref{sucesor y predecesor}, we know that $(\rho,\sigma)\in\Dir(P)$ and that there exists $(\rho',\sigma')\in \ov{\Dir}(P)$ such that
\begin{equation}\label{starr}
(\rho,\sigma)>(\rho',\sigma')\text{ are consecutive elements of $\ov{\Dir}(P)$ } \quad\text{and}\quad (\rho,\sigma)>(1,1)>(\rho',\sigma').
\end{equation}
So, by statement~(3) of Proposition~\ref{le basico} and the second equality in~\eqref{starrr},
\begin{equation}\label{star}
\{\st_{\rho,\sigma}(\ddot{P})\} = \Supp(\ell_{1,1}(\ddot{P})) =\{(m,0)\}.
\end{equation}
We claim that $(\rho,\sigma) < (0,1)$. In fact, if $(0,1)<(\rho,\sigma)$, then by statement~(2) of Proposition~\ref{le basico1},
$$
v_{0,1}\bigl(\en_{\rho,\sigma}(\ddot{P})\bigr)< v_{0,1}\bigl(\st_{\rho,\sigma}(\ddot{P})\bigr)= 0,
$$
which is impossible since $\ddot{P}\in A_1$; while if $(\rho,\sigma) = (0,1)$, then
$$
v_{0,1}(\ddot{P})= v_{0,1}\bigl(\st_{\rho,\sigma}(\ddot{P})\bigr) = v_{0,1}(m,0)=0,
$$
which leads to $\ddot{P}\in K[X]$, contradicting the comment above Remark~\ref{a remark} and finishing the proof of the claim.
By equality~\eqref{star} and the third equality in~\eqref{eq57}, we have
\begin{equation}\label{para despues}
\ell_{\rho,\sigma}(\ddot{P}) = \mu x^m f^{(1)}_{\ddot{P},\rho,\sigma}(z) \quad\text{where $z:=x^{-\frac{\sigma}{\rho}}y$.}
\end{equation}
Note that $(1,1)<(\rho,\sigma)<(0,1)$ means $\sigma>\rho>0$. So, by statements~(1) and~(3) of Lemma~\ref{rho igual a uno}, the polynomial $f^{(1)}_{\ddot{P},\rho, \sigma}$ has only one irreducible factor and $\rho\!=\!1$. Consequently, there exist $k\!\in\!\mathds{N}$ and $\eta\!\in\! K^{\times}$, such that
\begin{equation}\label{para despues1}
\ell_{1,\sigma}(\ddot{P})=\mu x^m(z-\eta)^k=\mu x^{m-k\sigma}(y -\eta x^\sigma)^k.
\end{equation}
Let $\varphi''\in\Aut(A_1)$ be the automorphism defined by $\varphi''(X):=X$ and $\varphi''(Y):=Y+\eta X^{\sigma}$ and set $P_1:=\varphi''(\ddot{P})$ and $Q_1:=\varphi''(\ddot{Q})$. By Remark~\ref{re ell por automorfismos}
\begin{equation}\label{para despues2}
\ell_{1,\sigma}(P_1) = \varphi''_L\bigl(\ell_{1,\sigma}(\ddot{P})\bigr) = \mu x^{m-k\sigma}y^k\qquad\text{and} \qquad v_{1,\sigma}(P_1) = v_{1,\sigma}(\ddot{P}).
\end{equation}
Hence, for each $(i,j)\in \Supp(P_1)$ we have
$$
v_{1,\sigma}(i,j)=i+j\sigma\le v_{1,\sigma}(P_1) = v_{1,\sigma}(\ddot{P}) = m,
$$
which, combined with the fact that $\sigma\ge 1$, gives
$$
v_{1,1}(i,j)=i+j\le i+j\sigma\le m= v_{1,1}(P).
$$
Since $\sigma>1$, the equality is only possible if $j=0$ and $i=m$. But $(m,0)\notin\Supp(P_1)$ since $k>0$, and so $v_{1,1}(P_1)<v_{1,1}(P)$, which is the second condition in~\eqref{faltaba}. It remains to check the first one, which follows from the series of equalities
$$
\frac{v_{1,1}(P_1)}{v_{1,1}(Q_1)}=\frac{v_{1,\sigma}(P_1)}{v_{1,\sigma}(Q_1)} =\frac{v_{1,\sigma}(\ddot{P})}{v_{1,\sigma}(\ddot{Q})}
=\frac{v_{1,1}(\ddot{P})}{v_{1,1}(\ddot{Q})}=\frac{v_{1,1}(P)}{v_{1,1}(Q)},
$$
where the first and the third one are true by statement~(6) of Lemma~\ref{PQ similares en I} and the second and fourth one, by Remark~\ref{re ell por automorfismos}.
\end{proof}

\begin{definition}\label{Smp} A minimal pair $(P,Q)$ is called a {\em standard minimal pair} if $P$ and $Q$ are sub\-rectangular and $v_{1,-1}(\st_{1,0}(P)), v_{1,-1}(\st_{1,0}(Q))<0$.
\end{definition}

\begin{corollary}\label{todos son Smp} For each minimal pair $(P,Q)$, there exists a standard minimal pair $(\tilde P,\tilde Q)$ such that
$$
v_{1,1}(\tilde{Q})=v_{1,1}(Q)\quad\text{and}\quad v_{1,1}(\tilde{P})= v_{1,1}(P).
$$
\end{corollary}

\begin{proof} By Proposition~\ref{minimalessubrectangular} we know that there exists a minimal pair $(P_1,Q_1)$ with $P_1$ and $Q_1$ subrectangular such that
\begin{equation}
v_{1,1}(P_1) = v_{1,1}(P)\quad\text{and}\qquad v_{1,1}(Q_1) = v_{1,1}(Q).\label{que}
\end{equation}
Let $(a,b)$ be such that $\{(a,b)\}=\Supp(\ell_{1,1}(P_1))$. Since $v_{1,1}(P_1) = a+b>0$ it follows from statement~(3) of Theorem~\ref{central} that $a\ne b$. Applying, if necessary, the automorphism $\varphi$ of $A_1$ defined by $\varphi(X):=Y$ and $\varphi(Y):=-X$, we can suppose that $b>a$. Hence, by Remark~\ref{propiedad de la esquina}
\begin{equation}
v_{1,-1}(\en_{1,0}(P_1)) = v_{1,-1}((a,b))<0.\label{que1}
\end{equation}
Moreover, since $(1,0)\in \ov I$, it follows from statements~(1) and~(3) of Lemma~\ref{PQ similares en I} that
\begin{equation}
v_{1,-1}(\en_{1,0}(Q_1))=\frac {v_{1,0}(Q_1)}{v_{1,0}(P_1)} v_{1,-1}(\en_{1,0}(P_1))<0.\label{que2}
\end{equation}
If $(1,0)\notin\Dir(P_1)$, then by statement~(5) of Lemma~\ref{PQ similares en I}, we know that $(1,0)\notin\Dir(Q_1)$. Thus
$$
v_{1,-1}(\st_{1,0}(P_1)) = v_{1,-1}(\en_{1,0}(P_1)) < 0\quad\text{and}\quad v_{1,-1}(\st_{1,0}(Q_1)) = v_{1,-1}(\en_{1,0}(Q_1)) < 0,
$$
and so we can take $(\tilde{P},\tilde{Q})\!:=\!(P_1,Q_1)$. Hence we can assume that $(1,0)\!\in\!\Dir(P_1)$. We claim that conditions (a)--(f) of Proposition~\ref{preparatoria} are satisfied for $P_1$, $Q_1$ and $(\rho,\sigma)\!:\!=(1,0)$. In fact condition~(a) is trivial and conditions (b), (c) and~(d) follow easily from Lemma~\ref{PQ similares en I}. Now, by~\eqref{que} we know that $(P_1,Q_1)$ is a minimal pair and so condition~(e) follows from Proposition~\ref{no se dividen} and statement~(6) of Lemma~\ref{PQ similares en I}. Finally condition~(f) is~\eqref{que1} and~\eqref{que2}. Let $\varphi$ be as in Proposition~\ref{preparatoria} and let $\tilde{P}:=\varphi(P_1)$ and $\tilde{Q} := \varphi(Q_1)$. Since
$$
(1,1)\in I\quad\text{and}\quad I\subseteq \{(\rho'',\sigma''):(1,0) <(\rho'',\sigma'')<(-1,1)\},
$$
it follows from statement~(5) of Proposition~\ref{preparatoria} that
\begin{align}
&\Dir(\tilde{P})\cap I=\Dir(P_1)\cap I, &&\Dir(\tilde{Q})\cap I=\Dir(Q_1)\cap I\label{que3},\\
\shortintertext{and}
&v_{1,1}(\tilde{P}) = v_{1,1}(P_1), && v_{1,1}(\tilde{Q}) = v_{1,1}(Q_1).\label{que4}
\end{align}
From~\eqref{que3} and Lemma~\ref{subrectangularconVal} we conclude that $\tilde{P}$
and $\tilde{Q}$ are subrectangular and from~\eqref{que4} it follows that $(\tilde{P},\tilde{Q})$ is a minimal pair. By statements~(2), (6) and~(7) of Proposition~\ref{preparatoria}
$$
v_{1,-1}(\st_{1,0}(\tilde{P}))<0\quad\text{and}\quad v_{1,-1}(\st_{1,0} (\tilde{Q}))<0,
$$
and so $(\tilde{P},\tilde{Q})$ is a standard minimal pair.
\end{proof}

\section{Computing lower bounds}

\setcounter{equation}{0}

In this last section we will prove that $B>15$. For this we start from a standard minimal pair $(P,Q)$ and find in Proposition~\ref{primitivo} a direction $(\rho,\sigma)$ such that $v_{1,-1}\bigl(\st_{\rho,\sigma}(P)\bigr)>0$ and $v_{1,-1} \bigl(\en_{\rho,\sigma}(P) \bigr)<0$. The condition this imposes on the involved corners are shown in Proposition~\ref{final} and allow to prove the desired result.

In a forthcoming article we will carry over this result from Dixmier pairs to Jacobian pairs. This will improve the lower bound for the greatest common divisor of the degrees given in~\cite{N1} and~\cite{N2} which is $B>8$.

\begin{lemma}\label{corchete se anula} Let $P,Q\in A_1$ and $(\rho,\sigma)\in \ov{\mathfrak{V}}$ with $\sigma<0$. If $[Q,P]=1$ and $\en_{\rho,\sigma}(P) =\lambda \en_{\rho,\sigma}(Q)$ for some $\lambda>0$, then

\begin{enumerate}

\smallskip

\item $v_{\rho,\sigma}(P),v_{\rho,\sigma}(Q)>0$,

\smallskip

\item $[Q,P]_{\rho,\sigma} = 0$,

\smallskip

\item $(\rho,\sigma)\in \Dir(P)$ if and only if $(\rho,\sigma)\in \Dir(Q)$.

\end{enumerate}

\end{lemma}

\begin{proof} (1)\enspace If $v_{\rho,\sigma}(P) \le 0$, then $v_{\rho,\sigma}(Q) = \lambda v_{\rho,\sigma}(P) \le 0$. When $(\rho,\sigma)=(1,-1)$, this is impossible by Remark~\ref{re v de un conmutador con 1,-1}. Assume now that $\rho+\sigma>0$. Then, by Remark~\ref{re v de un conmutador}, we have
$$
0 = v_{\rho,\sigma}([Q,P]) \le v_{\rho,\sigma}(P) + v_{\rho,\sigma}(Q) - (\rho+\sigma) < 0,
$$
which is also impossible. So $v_{\rho,\sigma}(P)>0$, and similarly $v_{\rho,\sigma}(Q) >0$.

\smallskip

\noindent (2)\enspace When $(\rho,\sigma)=(1,-1)$, it follows immediately from statement~(1) that $[Q,P]_{\rho,\sigma} = 0$. Assume now that $\rho+\sigma>0$ and suppose by contradiction that $[Q,P]_{\rho,\sigma}\ne 0$. Then $[Q,P]_{\rho,\sigma}= 1$ and so
\begin{equation}
v_{\rho,\sigma}(P) + v_{\rho,\sigma}(Q) = \rho + \sigma > 0.\label{quee}
\end{equation}
Moreover, by statement~(1),
\begin{equation}
\st_{\rho,\sigma}(P) \ne (0,0)\ne \st_{\rho,\sigma}(Q).\label{quee1}
\end{equation}
We claim that $F:= \Psi^{-1}(\ell_{\rho,\sigma}(PQ))$ satisfies
$$
[F,P]_{\rho,\sigma} = \ell_{\rho,\sigma}(P)\quad\text{and}\quad[F,Q]_{\rho,\sigma}= -\ell_{\rho,\sigma}(Q).
$$
In fact, we have
$$
v_{\rho,\sigma}([PQ,P]) = v_{\rho,\sigma}(P) =  v_{\rho,\sigma}(Q)+ 2v_{\rho,\sigma}(P) -(\rho+\sigma) = v_{\rho,\sigma}(PQ)+v_{\rho,\sigma}(P)-(\rho+\sigma),
$$
where the second equality follows from~\eqref{quee}, and the last one from statement~(3) of Proposition~\ref{pr v de un producto}. Hence $[PQ,P]_{\rho,\sigma}\ne 0$, and so
\begin{equation}
\ell_{\rho,\sigma}(P) = \ell_{\rho,\sigma}([PQ,P]) = [PQ,P]_{\rho,\sigma} = [F,P]_{\rho,\sigma},\label{quee4}
\end{equation}
where the last equality follows from Corollary~\ref{ell depende del ell}. Similarly $$
[F,Q]_{\rho,\sigma} =-\ell_{\rho,\sigma}(Q).
$$
So, by Proposition~\ref{extremosnoalineados},
\begin{equation}
\st_{\rho,\sigma}(F)=(1,1)\quad\text{or}\quad\st_{\rho,\sigma}(Q)\sim \st_{\rho,\sigma}(F) \sim \st_{\rho,\sigma}(P).\label{quee3}
\end{equation}
Moreover, by statement~(4) of Proposition~\ref{pr v de un producto},
\begin{equation}
\st_{\rho,\sigma}(F) = \st_{\rho,\sigma}(Q) + \st_{\rho,\sigma}(P).\label{quee2}
\end{equation}
If $\st_{\rho,\sigma}(F)=(1,1)$, then by~\eqref{quee1} and~\eqref{quee2}, we have
$$
\{\st_{\rho,\sigma}(P),\st_{\rho,\sigma}(Q)\}=\{(0,1),(1,0)\}
$$
which contradicts statement~(1) since $v_{\rho,\sigma}(0,1)=\sigma<0$. Thus, we can assume that the second condition in~\eqref{quee3} is satisfied. Then, by~\eqref{quee1}, there exist $\lambda_P,\lambda_Q\ge 0$ such that
\begin{equation}
\st_{\rho,\sigma}(F) = \lambda_P\st_{\rho,\sigma}(P)\quad\text{and}\quad \st_{\rho,\sigma}(F) = \lambda_Q\st_{\rho,\sigma}(Q)\label{quee5}
\end{equation}
Since, by equality~\eqref{quee4}, we have $v_{\rho,\sigma}(F) = \rho + \sigma$, there exists $\mu\ge-1$ such that
$$
\st_{\rho,\sigma}(F)=(1,1)+\mu(-\sigma/\rho,1).
$$
Note that
$$
\sigma<0\text{ and } \rho+\sigma>0 \Rightarrow \rho > 1 \Rightarrow \frac{\sigma}{\rho} \notin \mathds{Z}\Rightarrow \st_{\rho,\sigma}(F)\ne (1+\sigma/\rho,0) \Rightarrow \mu\ne -1.
$$
Because of $\st_{\rho,\sigma}(F)\ne(1,1)$, we have $\mu>0$. Hence
$$
v_{1,-1}(\st_{\rho,\sigma}(F))=-\mu(\sigma+\rho)/\rho<0,
$$
which combined with~\eqref{quee5} gives
$$
0 > v_{1,-1}(\st_{\rho,\sigma}(F)) = \lambda_P v_{1,-1}(\st_{\rho,\sigma}(P)) \quad \text{and}\quad 0 > v_{1,-1}(\st_{\rho,\sigma}(F)) = \lambda_Q v_{1,-1}( \st_{\rho,\sigma}(Q)).
$$
Since $\lambda_P,\lambda_Q\ge 0$, we conclude that
$$
0> v_{1,-1}(\st_{\rho,\sigma}(P)) =v_{1,-1}(\ell_{\rho,\sigma}(P)) \quad\text{and}\quad 0> v_{1,-1}(\st_{\rho,\sigma}(Q)) =v_{1,-1}(\ell_{\rho,\sigma}(Q)).
$$
But then, by Corollary~\ref{ell depende del ell} and Remark~\ref{re v de un conmutador},
$$
v_{1,-1}([Q,P]_{\rho,\sigma})= v_{1,-1}([\Psi^{-1}(\ell_{\rho,\sigma}(Q)), \Psi^{-1}(\ell_{\rho,\sigma}(P))]_{\rho,\sigma}) \le v_{1,-1}(\ell_{\rho,\sigma}(P))+v_{1,-1}(\ell_{\rho,\sigma}(Q))<0,
$$
which contradicts $[Q,P]_{\rho,\sigma} = 1$.

\smallskip

\noindent (3)\enspace This is a consequence of Remark~\ref{a remark} and statements~(1) and~(2).
\end{proof}

\begin{proposition}\label{primitivo} For each standard minimal pair $(P,Q)$ there exists $(\rho,\sigma)\in \Dir(P)\cap\Dir(Q)$, such that

\begin{enumerate}

\smallskip

\item $\sigma<0$,

\smallskip

\item $v_{\rho,\sigma}(P)>0$ and $v_{\rho,\sigma}(Q)>0$,

\smallskip

\item $[P,Q]_{\rho,\sigma}=0$,

\smallskip

\item $\frac{v_{\rho,\sigma}(P)}{v_{\rho,\sigma}(Q)}= \frac{v_{1,1}(P)} {v_{1,1}(Q)}$,

\smallskip

\item $\frac{v_{\rho,\sigma}(Q)}{v_{\rho,\sigma}(P)}\notin\mathds{N}$ and $\frac{v_{\rho,\sigma}(P)}{v_{\rho,\sigma}(Q)}\notin\mathds{N}$,

\smallskip

\item $v_{1,-1}\bigl(\st_{\rho,\sigma}(P)\bigr)>0$ and $v_{1,-1} \bigl(\en_{\rho,\sigma}(P) \bigr)<0$.

\end{enumerate}

\end{proposition}

\begin{proof} Consider the set $V_-(P)$ made out of directions in $\ov{\Dir}(P)$ lower than $(1,0)$ together with $(1,0)$. We have
$$
V_-(P)=\{(\rho_0,\sigma_0)=(1,-1)<(\rho_1,\sigma_1)<\dots<(\rho_k,\sigma_k)=(1,0)\}.
$$
By statement~(1) of Lemma~\ref{PQ similares en I}, there exist $m,n\in\mathds{N}$, coprime such that
$$
\frac{m}{n} = \frac{v_{1,1}(P)}{v_{1,1}(Q)}.
$$
Then, by statements~(3) and (6) of Lemma~\ref{PQ similares en I},
\begin{equation}\label{starting}
\st_{1,0}(P)=\frac mn\st_{1,0}(Q).
\end{equation}
Let $j\in\{0,\dots,k-1\}$. We claim that if
$$
\st_{\rho_{j+1},\sigma_{j+1}}(P)=\frac mn\st_{\rho_{j+1},\sigma_{j+1}}(Q),
$$
then
\begin{equation}\label{paso inductivo}
v_{\rho_j,\sigma_j}(P),v_{\rho_j,\sigma_j}(Q)>0,\qquad [P,Q]_{\rho_j,\sigma_j}=0 \qquad\text{and}\qquad\frac{v_{\rho_j,\sigma_j}(P)}{v_{\rho_j,\sigma_j}(Q)} = \frac{m}{n}.
\end{equation}
In order to prove this claim we take
$$
(\rho',\sigma'):=\max\{(\rho'',\sigma'')\in \ov\Dir(P)\cup \ov\Dir(Q):(\rho'',\sigma'')< (\rho_{j+1},\sigma_{j+1})\}
$$
and we apply statement~(3) of Proposition~\ref{le basico} to obtain
\begin{equation}\label{bar}
\en_{\rho',\sigma'}(P)=\st_{\rho_{j+1},\sigma_{j+1}}(P)=\frac mn\st_{\rho_{j+1}, \sigma_{j+1}}(Q) = \frac mn\en_{\rho',\sigma'}(Q).
\end{equation}
So, we are in the conditions to apply Lemma~\ref{corchete se anula}, since
$$
(\rho',\sigma')<(\rho_{j+1},\sigma_{j+1})\le (1,0)\Longrightarrow \sigma'<0.
$$
Consequently by statement~(3) of that Lemma $(\rho',\sigma') = (\rho_j,\sigma_j)$, and by statements~(1) and~(2), the first two conditions in~\eqref{paso inductivo} are satisfied. Finally the last one follows immediately from~\eqref{bar}, finishing the proof of the claim.

Now note that~\eqref{paso inductivo} and~\eqref{pepe} imply
$$
\st_{\rho_{j},\sigma_{j}}(P)=\frac mn\st_{\rho_{j},\sigma_{j}}(Q)\quad\text{for $j>0$}.
$$
So we can start from~\eqref{starting} and apply an inductive procedure in order to obtain~\eqref{paso inductivo} for all $j=0,\dots,k-1$. Consequently, $(\rho_j,\sigma_j)$ satisfies statements~(1), (2), (3) and~(4) for $j=0,\dots,k-1$. Moreover, by Proposition~\ref{no se dividen} they also satisfy statement~(5). We now are going to prove that one of this $(\rho_j,\sigma_j)$ satisfies statement~(6). Since, by statement~(3) of Proposition~\ref{le basico},
$$
\st_{\rho_1,\sigma_1}(P) = \en_{1,-1}(P),
$$
we have
\begin{equation}
v_{1,-1}(\st_{\rho_1,\sigma_1}(P))=v_{1,-1}(P)>0,\label{bar1}
\end{equation}
where the inequality follows from~\eqref{paso inductivo} for $j=0$. Furthermore, by statement~(2) of Proposition~\ref{le basico1} and statement~(3) of Proposition~\ref{le basico},
\begin{equation}
v_{1,-1}(\st_{\rho_j,\sigma_j}(P))> v_{1,-1}(\en_{\rho_j,\sigma_j}(P)) =
v_{1,-1}(\st_{\rho_{j+1},\sigma_{j+1}}(P))\quad\text{for all $0<j<k$.}\label{bar2}
\end{equation}
Moreover
\begin{equation}
v_{1,-1}(\st_{\rho_k,\sigma_k}(P))=v_{1,-1}(\st_{1,0}(P))<0,\label{bar3}
\end{equation}
since $(P,Q)$ is a standard minimal pair. Combining~\eqref{bar1}, \eqref{bar2} and~\eqref{bar3} we obtain that there exists $0<j<k$ such that
$$
v_{1,-1}(\st_{\rho_j,\sigma_j}(P))>0\quad\text{and}\quad v_{1,-1}(\en_{\rho_j,\sigma_j}(P))\le 0.
$$
But by statement~(3) of Theorem~\ref{central}, which can be applied since $v_{\rho_j,\sigma_j}(P)>0$, the equality $v_{1,-1}(\en_{\rho_j,\sigma_j}(P))=0$ is impossible.  So, in order to finish the proof it suffices to note that, by statement~(3) of lemma~\ref{corchete se anula}, we have $(\rho_j,\sigma_j)\in\Dir(P)\cap \Dir(Q)$.
\end{proof}

Let $(P,Q)$ be a standard minimal pair. By statement~(1) of Lemma~\ref{PQ similares en I}, there exist $m,n\in\mathds{N}$ coprime, such that
$$
\frac{m}{n} = \frac{v_{1,1}(P)}{v_{1,1}(Q)}.
$$
Let $(\rho,\sigma)$ be as in Proposition~\ref{primitivo}. Set
$$
C_0:=\frac 1m\en_{\rho,\sigma}(P)\qquad\text{and}\qquad C_1:=\frac 1m\st_{\rho,\sigma}(P).
$$
By Theorem~\ref{central} there exists a $(\rho,\sigma)$-homogeneous $F\in A_1$ such that
$$
[P,F]_{\rho,\sigma}=\ell_{\rho,\sigma}(P)\qquad\text{and}\qquad v_{\rho,\sigma}(F) = \rho+\sigma.
$$
Write $(f_1,f_2)=\en_{\rho,\sigma}(F)$, $(u,v)=C_0$, $(r',s')=C_1$ and set $\gamma:=(v-s')/\rho$. Note that
$$
v_{\rho,\sigma}(u,v) = v_{\rho,\sigma}(r',s')\qquad\text{and}\qquad r'<u,
$$
where the inequality follows from statement~(1) of Proposition~\ref{le basico1}, since $(\rho,\sigma)<(1,0)$.

\begin{proposition}\label{final} It holds that $C_0,C_1\in\mathds{N}_0\times \mathds{N}_0$ and $v_{\rho,\sigma}(C_0) = v_{\rho,\sigma}(C_1)$. Moreover,

\begin{enumerate}

\smallskip

\item $u<v$ and $r'>s'$,

\smallskip

\item $f_1\ge 2$ and $u\ge 3$,

\smallskip

\item $\gcd(u,v)>1$,

\smallskip

\item $\en_{\rho,\sigma}(F)=\mu C_0$ for some $0<\mu<1$,

\smallskip

\item $(\rho,\sigma)=\dir(f_1-1,f_2-1)$,

\smallskip

\item if $d:=\gcd(f_1-1,f_2-1)=1$ then
$$
C_2:=C_1+(\gamma-s')\left(-\frac{\sigma}{\rho},1 \right),
$$
is not of the form $\bigl(\ov r-\frac {1}{\rho},\ov r\bigr)$ for any $\ov r\ge 2$.

\smallskip

\end{enumerate}
\end{proposition}

\begin{proof} By statements~(1)--(4) of Proposition~\ref{primitivo} and statement~(2) of Theorem~\ref{f[] en A_1^{(l)}} there exist $\lambda_P,\lambda_Q\in K^{\times}$ and a $(\rho,\sigma)$-homogeneous element $R\in L$, such that
\begin{equation}\label{R existe}
\ell_{\rho,\sigma}(P) = \lambda_P R^m\quad\text{and}\quad\ell_{\rho,\sigma}(Q) = \lambda_P R^n.
\end{equation}
By statements~(4) and~(5) of Proposition~\ref{pr v de un producto}, necessarily $C_0=\en_{\rho,\sigma}(R)$ and $C_1=\st_{\rho,\sigma}(R)$, and so both are in $\mathds{N}_0\times\mathds{N}_0$ and $v_{\rho,\sigma}(C_0) = v_{\rho,\sigma}(C_1)$. Statement~(1) follows from the fact that $v_{1,-1}(\en_{\rho,\sigma}(P))<0$ and $v_{1,-1}(\st_{\rho,\sigma}(P))>0$ (statement~(5) of Proposition~\ref{primitivo}). In order to prove the rest of the proposition we need some computations. Note that
$$
\st_{\rho,\sigma}(F),\en_{\rho,\sigma}(F)\in\{(1,1)+\alpha(-\sigma/\rho,1):\alpha \in \mathds{Z},\ \alpha\ge -1\}.
$$
But $\alpha=-1$ cannot occur, since $(1+\sigma/\rho,0)$ is not in $\mathds{N}_0 \times \mathds{N}_0$, due to $\rho>-\sigma>0$. Since, by statement~(1) of Corollary~\ref{pavadass}, the element $F$ is not a monomial, there exist $\alpha_{\st}\ne\alpha_{\en}$ in $\mathds{N}_0$, such that
$$
\st_{\rho,\sigma}(F) = (1,1)+\alpha_{\st}(-\sigma/\rho,1)\qquad\text{and}\qquad \en_{\rho,\sigma}(F) = (1,1)+\alpha_{\en}(-\sigma/\rho,1).
$$
Since
$$
v_{1,-1}(\st_{\rho,\sigma}(F)) = -\alpha_{\st}\frac{\rho+\sigma}{\rho}\le 0,
$$
it is impossible that $\st_{\rho,\sigma}(F)\sim\st_{\rho,\sigma}(P)$, because $v_{1,-1}(\st_{\rho,\sigma}(P))>0$ (statement~(5) of Pro\-po\-sition~\ref{primitivo}). Hence, by statement~(1) of Theorem~\ref{central},
\begin{equation}
\st_{\rho,\sigma}(F) = (1,1).\label{foo1}
\end{equation}
But this means that $\alpha_{\st} = 0$ and so $\alpha_{\en}>0$. Consequently,
\begin{equation}
f_1 = 1- \alpha_{\en}\frac{\sigma}{\rho} \ge 2\label{foo}
\end{equation}
and, by statement~(2) of Theorem~\ref{central},
$$
\en_{\rho,\sigma}(F)\sim\en_{\rho,\sigma}(P).
$$
From $v_{\rho,\sigma}(F)\ne 0\ne v_{\rho,\sigma}(P)$ we obtain
$$
\en_{\rho,\sigma}(F)\ne (0,0)\ne\en_{\rho,\sigma}(P),
$$
and therefore
$$
(f_1,f_2) = \en_{\rho,\sigma}(F)=\frac{\mu}{m}\en_{\rho,\sigma}(P) = \mu C_0 = \mu(u,v) \quad\text{for some $\mu>0$.}
$$
Notice that
$$
v_{\rho,\sigma}(C_1)=r'\rho+s'\sigma=(r'-s')\rho+s'(\rho+\sigma)\ge(r'-s')\rho\ge \rho>\rho+\sigma,
$$
since $r'-s'\! = \!\frac 1m v_{1,-1}(\st_{\rho,\sigma}(P))\!>\!0$. Hence $\mu\!\ge \! 1$ is impossible, as it would lead to the contradiction
$$
v_{\rho,\sigma}(C_1)=v_{\rho,\sigma}(C_0)=\frac{1}{\mu}v_{\rho,\sigma}(F)\le v_{\rho,\sigma}(F)=\rho+\sigma.
$$
This completes the proof of statement~(4) and, combined with $f_1=\mu u$ and~\eqref{foo}, also proves statement~(2). Moreover, if $\gcd(u,v)=1$, then there is no $\mu\in]0,1[$ such that $\mu(u,v)\in\mathds{N}_0\times\mathds{N}_0$, and so statement~(3) is true. Statement~(5) holds by Remark~\ref{valuacion depende de extremos} and equality~\eqref{foo1}. Finally we prove statement~(6). By Proposition~\ref{primitivo}, the hypothesis of Proposition~\ref{preparatoria} are satisfied for $P,Q$ and $(\rho,\sigma)$, with the possible exception of the inequality $v_{1,-1}(\en_{\rho,\sigma}(Q))<0$. But this last condition is also satisfied, due to $v_{1,-1}(\en_{\rho,\sigma}(P))<0$, equalities~\eqref{R existe} and statement~(5) of Proposition~\ref{pr v de un producto}. Consider the automorphism $\varphi$ of $A_1^{(\rho)}$ and the direction $(\rho,\sigma)\in \mathfrak{V}$, obtained from Proposition~\ref{preparatoria}. As in Proposition~\ref{preparatoria} write $(r,s)= \st_{\rho,\sigma}(P)$. By statement~(5) of Proposition~\ref{pr v de un producto} we have $(r,s) = m(r',s')$. Consequently, by statement~(6) of Proposition~\ref{preparatoria},
$$
\frac 1m\en_{\rho',\sigma'}(\varphi(P)) = \frac 1m \Bigl(r+s\frac{\sigma}{\rho} -m_\lambda\frac{\sigma}{\rho},m_\lambda\Bigr)=\Bigl(r'+\frac{s'\sigma}{\rho} -\frac{m_{\lambda}}{m} \frac{\sigma}{\rho},\frac{m_{\lambda}}{m}\Bigr) = C_1+\Bigl(\frac{m_{\lambda}}{m}-s'\Bigr)\left(-\frac{\sigma}{\rho},1 \right),
$$
where $m_\lambda$ is the highest multiplicity of a linear factor in $f_{P,\rho,\sigma}$. Statements~(1), (2) and~(4) of Proposition~\ref{preparatoria} guarantee that the hypothesis of Proposition~\ref{lema general} are satisfied for $\varphi(P)$, $\varphi(Q)$ and $(\rho',\sigma')$. Consequently
$$
\frac{1}{m} \en_{\rho',\sigma'}(\varphi(P))\ne  \left(\ov r-\frac {1}{\rho},\ov r\right)\qquad\text{for all $\ov r\ge 2$.}
$$
Hence, if we prove $m_\lambda = m\gamma$, then $C_2=\frac{1}{m}\en_{\rho',\sigma'}(\varphi(P))$, which concludes the proof of~(6).
Note that by statement~(5) and~\eqref{val}
$$
(\rho,\sigma) = \pm \left(\frac{f_2-1}{d},\frac{1-f_1}{d}\right),
$$
where $d:=\gcd(f_1-1,f_2-1)$. Since $\rho+\sigma>0$ and, by statements~(1) and~(4), we have $f_2-f_1>0$, necessarily
$$
(\rho,\sigma) = \left(\frac{f_2-1}{d},\frac{1-f_1}{d}\right).
$$
If now $d=1$, then
$$
v_{1,0}(\en_{\rho,\sigma}(F))- v_{1,0}(\st_{\rho,\sigma}(F)) = f_2-1 = \rho
$$
and by statement~(3) of Corollary~\ref{pavadass}
$$
m_\lambda = \frac{1}{\rho}\left(v_{1,0}(\en_{\rho,\sigma}(P))- v_{1,0}(\st_{\rho,\sigma} (P))\right)= \frac{m(v-s')}{\rho} = m\gamma,
$$
as desired.
\end{proof}

\begin{corollary} Let $A_1$ be the Weyl algebra over a non-necessarily algebraically closed characteristic zero field $K$. We have $B>15$.
\end{corollary}

\begin{proof} Without loss of generality we can assume that $K$ is algebraically closed. Let $(P,Q)$ be a minimal pair. By Corollary~\ref{todos son Smp} we can also assume that $(P,Q)$ is standard. So, it is clear that if $(u,v)$ is as in Proposition~\ref{final}, then
$$
u+v=v_{1,1}(C_0)=\frac{1}{m}v_{1,1}(\en_{\rho,\sigma}(P))\le\frac{1}{m} v_{1,1}(P)= B.
$$
So it suffices to prove that there is no pair $(u,v)$ with $u+v\le 15$, for which
there exist $(f_1,f_2)$ and $C_1=(r',s')$, such that all the conditions of Proposition~\ref{final} are satisfied.

\smallskip

\begin{center}\begin{tabular}{|c|c|c|c|c|c|c|}\hline
              $C_0$ &$(f_1,f_2)$  & $(\rho,\sigma)$& $C_1$ & d & $\gamma$ & $C_2$\\\hline
              (3,6) & (2,4) & (3,-1) & (1,0) & 1& 2& $\left(2-\frac 13,2\right) $\\\hline
              (3,9) & (2,6) & (5,-1) & $\times$ & &&\\\hline
              (3,12) & (2,8) & (7,-1) & $\times$ &&&\\\hline
              (4,6) & (2,3) & (2,-1) & (1,0) & 1& 3& $\left(3-\frac 12,3\right) $\\\hline
              (4,8) & (2,4) & (3,-1) & $\times$ &&&\\\hline
              (4,8) & (3,6) & (5,-2) & $\times$ &&&\\\hline
              (4,10) & (2,5) & (4,-1) & $\times$ &&&\\\hline
              (5,10) & (2,4) & (3,-1) & (2,1) &1&3&  $\left(3-\frac 13,3\right) $\\\hline
              (5,10) & (3,6) & (5,-2) & (1,0) &1&2&  $\left(2-\frac 15,2\right) $\\\hline
              (5,10) & (4,8) & (7,-3) & $\times$ &&&\\\hline
              (6,8) & (3,4) & (3,-2) & $\times$ &&&\\\hline
              (6,9) & (2,3) & (2,-1) & (2,1) &1&4& $\left(4-\frac 12,4\right) $\\\hline
              (6,9) & (4,6) & (5,-3) & $\times$ &&&\\\hline
\end{tabular}
\end{center}

\smallskip

\noindent First we list all possible pairs $(u,v)$ with $v>u>2$, $\gcd(u,v)>1$ and $u+v\le 15$. We also list all the possible $(f_1,f_2)=\mu (u,v)$ with $f_1\ge 2$ and $0<\mu<1$. Then we compute the co\-rres\-ponding $(\rho,\sigma)$ using statement~(5) of Proposition~\ref{final} and we verify if there is a~$C_1:=(r',s')$ with $s'<r'<u$ and $v_{\rho,\sigma}(u,v) = v_{\rho,\sigma}(r',s')$. This happens in five cases. In all these cases $d:=\gcd(f_1-1,f_2-1)=1$. We compute $\gamma:=(v-s')/\rho$ and $C_2$ in each of the five cases and we verify that in none of them condition~(6) of Proposition~\ref{final} is satisfied, which concludes the proof.
\end{proof}

\end{document}